\documentclass[a4paper]{amsart}
\usepackage[T1]{fontenc}
\usepackage{lmodern}
\usepackage[all]{xy}
\usepackage{amssymb,enumitem,nicefrac,mathtools}
\usepackage[british]{babel}

\usepackage{array}
\usepackage[labelformat=simple]{subfig}

\usepackage[pdftitle={Universal coefficient theorems for C*-algebras over finite topological spaces},
  pdfauthor={Rasmus Bentmann and Manuel Kohler},
  pdfsubject={Mathematics; MSC 18E30, 19K35, 46L80, 55U35}
]{hyperref}
\newcommand*{\MRref}[2]{\linebreak[0] \href{http://www.ams.org/mathscinet-getitem?mr=#1}{MR \textbf{#1}}}

\usepackage[lite]{amsrefs}
\usepackage{microtype}

\DeclareMathOperator{\Hom}{Hom}
\DeclareMathOperator{\Prim}{Prim}
\DeclareMathOperator{\Ext}{Ext}
\DeclareMathOperator{\Tor}{Tor}
\DeclareMathOperator{\coker}{coker}
\DeclareMathOperator{\im}{im}
\DeclareMathOperator{\range}{range}

\newcommand*{\KK}{\textup{KK}}

\newcommand*{\K}{\textup{K}}
\newcommand*{\Homology}{\textup{H}}   % homology

\newcommand*{\Ch}{\textup{Ch}}  % order complex
\newcommand*{\into}{\rightarrowtail}
\newcommand*{\prto}{\twoheadrightarrow}

\newcommand*{\cl}[1]{\overline{#1}} % closure

\newcommand*{\CONT}{\mathcal{C}}    % continuous functions
\newcommand*{\Nattrafo}{\mathcal{NT}} % natural transformations on filtered K-theory
\newcommand*{\FK}{\textup{FK}}% filtrated K-theory

\newcommand*{\Cstarcat}{\mathfrak{C^*alg}}
\newcommand*{\Cstarsep}{\mathfrak{C^*sep}}
\newcommand*{\KKcat}{\mathfrak{KK}}

\newcommand*{\Cat}{{\mathfrak{C}}}  % generic category, often Abelian
\newcommand*{\Tri}{{\mathfrak{T}}}  % triangulated category
\newcommand*{\Ab}{\mathfrak{Ab}}      % category of Abelian groups

\newcommand*{\Mod}[1]{\mathfrak{Mod}\bigl(#1\bigr)} % category of modules
\newcommand*{\Modc}[1]{\mathfrak{Mod}\bigl(#1\bigr)_\textup{c}}
\newcommand*{\Modcsmallargument}[1]{\mathfrak{Mod}(#1)_\textup{c}}
\newcommand*{\CModsmallargument}[1]{\mathfrak{Mod}(#1)_{\mathrm{c}}}
\newcommand*{\CMod}[1]{\mathfrak{Mod}\bigl(#1\bigr)_{\mathrm{c}}}

\newcommand*{\op}{\mathrm{op}}
\newcommand*{\ID}{\mathrm{id}}

\newcommand*{\cyc}[1]{\mathbf{C}(#1)} %cyclic group of order #1

\newcommand*{\C}{\mathbb{C}}
\newcommand*{\R}{\mathbb{R}}
\newcommand*{\Z}{\mathbb{Z}}
\newcommand*{\N}{\mathbb{N}}
\newcommand*{\Ideals}{\mathbb{I}}   % ideal lattice of a C*-algebra
\newcommand*{\Open}{\mathbb{O}}     % open subsets of a top. space
\newcommand*{\Loclo}{\mathbb{LC}}   % locally closed subsets of a top. space

\newcommand*{\nb}{\nobreakdash}  % no break after this hyphen
\newcommand*{\Bootstrap}{{\mathcal B}}

\newcommand*{\Cstar}{\texorpdfstring{$C^*$\nb-}{C*-}}
\newcommand*{\Star}{\texorpdfstring{$^*$\nb-}{*-}}

\newcommand*{\blank}{\text{\textvisiblespace}}
\newcommand*{\inOb}{\mathrel{\in\in}}

\newcommand*{\defeq}{\mathrel{\vcentcolon=}}

\newcommand*{\inin}{\mathrel{{\in}{\in}}}
\newcommand*{\LC}{\mathbb{LC}}
\newcommand*{\NT}{\mathcal{NT}}
\newcommand*{\NTC}{{\mathcal{NT}^*}}
\newcommand*{\NTS}{\mathcal{NT}_\textup{6-term}}
\newcommand*{\NTSE}{\mathcal{NT}_\textup{even 6-term}}
\newcommand*{\NTCS}{\mathcal{NT}^*_\textup{6-term}}
\newcommand*{\NTCSE}{\mathcal{NT}^*_\textup{even 6-term}}
\newcommand*{\NTnil}{\mathcal{NT}_\textup{nil}}
\newcommand*{\NTss}{\mathcal{NT}_\textup{ss}}
\newcommand*{\Mss}{M_\textup{ss}}
\newcommand*{\sss}{\textup{ss}}
\newcommand*{\con}{*}
\newcommand*{\kk}{\mathfrak{KK}}
\newcommand*{\bigb}[1]{\bigl(#1\bigr)}
\newcommand*{\Bigb}[1]{\Bigl(#1\Bigr)}
\newcommand*{\id}{\textup{id}}
\newcommand*{\perdef}{\mathrel{\vcentcolon=}}
\newcommand*{\Rep}{\mathcal{R}}%representing object
\newcommand*{\gAb}{\mathfrak{Ab}^{\Z/2}}
\newcommand*{\Cst}{C^*}
\newcommand*{\Cuntz}{\mathcal O}
\newcommand*{\catI}{\mathfrak{I}}
\newcommand*{\CatA}{\mathcal{A}}%(Abelian) category
\newcommand*{\AIT}{\CatA_\catI\Tri}
\newcommand*{\perdeflongequi}{\mathrel{\vcentcolon\Longleftrightarrow}}
\newcommand*{\lrangle}{\mathopen\rangle}
\newcommand*{\rlangle}{\mathclose\langle}
\newcommand*{\epi}{\twoheadrightarrow}
\newcommand*{\embed}{\hookrightarrow}
\newcommand*{\mono}{\rightarrowtail}
\newcommand*{\opencup}{\mathrel{\tilde{{\cup}}}}
\newcommand*{\closedcup}{\mathrel{\bar{{\cup}}}}
\newdir{ >}{{}*!/-5pt/@{>}}
\newcommand*{\Cont}{\mathcal C}

\newcommand*{\Boot}{\mathcal{B}}
\newcommand*{\KSYZ}{\K^*\bigl(S(Y,Z)\bigr)}
\numberwithin{equation}{section}

\theoremstyle{plain}
\newtheorem{theorem}[equation]{Theorem}
\newtheorem{proposition}[equation]{Proposition}
\newtheorem{lemma}[equation]{Lemma}

\newtheorem{corollary}[equation]{Corollary}
\newtheorem{property}{Property}
\newtheorem{observation}[equation]{Observation}

\theoremstyle{definition}
\newtheorem{definition}[equation]{Definition}

\theoremstyle{remark}
\newtheorem{remark}[equation]{Remark}
\newtheorem{warning}[equation]{Warning}
\newtheorem{example}[equation]{Example}

\title[UCT for $C^*$-algebras over finite topological spaces]{Universal coefficient theorems for $C^*$-algebras\\ over finite topological spaces}

\author{Rasmus Bentmann}

\address{Department of Mathematical Sciences\\University of
Copenhagen\\Universitetsparken 5\\2100 Copenhagen \O \\Denmark}
\email{bentmann@math.ku.dk}

\author{Manuel K\"{o}hler}
\address{Mathematisches Institut\\
  Georg-August Universit\"at G\"ottingen\\
  Bunsenstra{\ss}e 3--5\\
  37073 G\"ottingen\\
  Germany}
\email{koehler@uni-math.gwdg.de}


\thanks{The first-named author was supported by the Danish National Research Foundation through the Centre for Symmetry and Deformation and by the Marie Curie Research Training Network EU-NCG. The second-named author was supported by the German National Merit Foundation}

\begin{document}

\begin{abstract}
We determine the class of finite $T_0$-spaces allowing for a universal coefficient theorem computing equivariant $\KK$-theory by filtrated $\K$-theory.
\end{abstract}
\maketitle

\section{Introduction}
The universal coefficient theorem of Rosenberg and Schochet~\cite{MR894590}
states that for separable \Cstar algebras $A$ and $B$ with $A$ being in a
certain bootstrap class there is a short exact sequence of $\Z/2$-graded
Abelian groups
\[ \Ext^1\bigl(\K_{\ast+1}(A),\K_{\ast}(B)\bigr) \into \KK_{\ast}(A,B) \prto \Hom\bigl(\K_{\ast}(A),\K_{\ast}(B)\bigr). \]
Apart from being very useful for computations of $\KK$-groups, it plays an 
important role in the classification of \Cstar algebras by $\K$-theoretic invariants.

The corresponding sequence for $A=B$ is an extension of rings with the product 
in $\Ext^1\bigl(\K_{\ast+1}(A),\K_{\ast}(A)\bigr)$ being zero. Therefore $\KK_{\ast}(A,A)$ is a 
nilpotent extension of $\Hom\bigl(\K_{\ast}(A),\K_{\ast}(A)\bigr)$; this shows that isomorphisms 
in $\K$-theory lift to isomorphisms in $\KK$-theory. Results by Kirchberg~\cite{MR1796912} and
Phillips~\cite{MR1745197} then show that every $\KK$-equivalence between~$A$ and~$B$
is induced by an actual \Star isomorphism of \Cstar algebras,
provided that $A$ and $B$ are stable, nuclear, separable, purely infinite and \emph{simple}.
Both facts together give the following strong classification 
result: \Cstar algebras $A$ with the above-mentioned properties are completely classified 
by the $\Z/2$-graded Abelian group $\K_{\ast}(A)$.

It is interesting to extend this result to the non-simple case. In~\cite{MR1796912}, Eberhard Kirchberg indicated a construction of an equivariant version $\KK(X)$ of bivariant $\K$-theory for \Cstar{}algebras over a given $T_0$-space~$X$ and proved a corresponding classification result: a \(\KK(X)\)-equivalence between two \Cstar algebras over~\(X\) lifts to an equivariant \Star isomorphism if both \Cstar algebras are stable, nuclear, separable, purely infinite and \emph{tight}---the notion of tightness generalises simplicity; its name was coined in~\cite{MR2545613}.

Our aim is therefore to compute $\KK_*(X;A,B)$ for a topological space~\(X\) and \Cstar algebras \(A\) and~\(B\) over~\(X\) by a universal coefficient theorem, that is, by an exact sequence of the form
\[
  \Ext^1_\Cat\bigl(\Homology_{*+1}(A),\Homology_*(B)\bigr) \into
  \KK_*(X;A,B) \prto
  \Hom_\Cat\bigl(\Homology_*(A),\Homology_*(B)\bigr)
\]
for some reasonably tractable homology theory~\(\Homology_*\) for \Cstar{}algebras over~\(X\), taking values in some Abelian category~\(\Cat\). Here~\(A\) is assumed to belong to the bootstrap class $\Bootstrap(X)$ introduced in~\cite{MR2545613}.
As in the non-equivariant case, a universal coefficient theorem of this form allows to lift an isomorphism \(\Homology_*(A)\cong \Homology_*(B)\) in~\(\Cat\) to a \(\KK(X)\)-equivalence \(A\simeq B\) if both \(A\) and~\(B\) belong to the bootstrap class \(\Bootstrap(X)\).

In~\cite{meyernestCalgtopspacfiltrKtheory}, Ralf Meyer and Ryszard Nest applied
their machinery of homological algebra in triangulated categories developed 
in~\cites{MR2193334,meyernesthomalgintricat} with the aim of deriving a UCT short 
exact sequence which computes $\KK(X;A,B)$ for a finite $T_0$-space $X$ 
by filtrated $\K$-theory (in the following denoted by $\FK$).
They obtain the desired short exact sequence
in the case of the totally ordered space $O_n$ with $n$ points, that is,
\[O_n = \{1,2,\ldots ,n\}, \ \ \ \tau_{O_n} = \bigl\{\emptyset,\{n\},\{n,n-1\}, \ldots ,X\bigr\}.\]
A \Cstar algebra $A$ over this space is essentially the same 
as a  \Cstar algebra $A$ together with a 
finite increasing chain of ideals
\[
\{0\} \triangleleft
I_n \triangleleft
I_{n-1} \triangleleft
\dotsb \triangleleft
I_{2} \triangleleft I_1 = A.
\]
On the other hand, Meyer and Nest give an example of a finite $T_0$-space $Y$ for which 
the following strong non-UCT statement holds: There are objects~$A$ and~$B$ in 
$\Bootstrap(Y)$ with isomorphic filtrated $\K$-theory which are not $\KK(Y)$-equivalent.

The aim of this article is to give a complete answer to the following question: given a finite $T_0$-space~$X$,
is there a UCT short exact sequence which computes $\KK(X;A,B)$ by filtrated $\K$-theory?
The assumption of the separation axiom~$T_0$ is not a loss of generality here, since all
that matters is the lattice of open subsets of~$X$ (see~\cite{MR2545613}*{\S2.5}).

In order to describe the general form of those finite $T_0$-spaces for which there is such a UCT short exact 
sequence, we have to introduce some notation.
For topological spaces $X$ and $Y$ and $x \in X$, $y \in Y$, let us denote by
$ X \bigvee_{x=y} Y$
the quotient space of the disjoint union $X\sqcup Y$ by the equivalence relation generated by $x \sim y$.
\begin{definition}
	\label{def:typeA}
Let $X$ be finite $T_0$-space. We say that $X$ is of \emph{type}~(A)---A~for accordion---if~$X$ is of the form
\[
O_{n_1} \bigvee_{ n_1 = n_2 }O_{n_2} \bigvee_{ 1 = 1} O_{n_3} \  \ldots \ O_{n_{m-1}}\bigvee_{n_{m-1} = n_{m} }O_{n_{m}}
\]
for $m\in 2\N_{>0}$, $n_i \in \N_{>0}$ and $n_i>1$ for $2\leq i\leq m-1$.
\end{definition}

To get an alternative description of type (A) spaces recall from~\cite{MR2193334} how finite spaces can be visualised as directed graphs:

\begin{definition} Let $X$ be a finite $T_0$-space. Define $\Gamma(X) = (V,E)$ by 
$V \coloneqq X$, and  $(x, y) \in E$ if and only if $x\neq y,\; x \in \overline{\{y\}}$
and $\bigb{x \in \overline{\{z\}}, \ z \in \overline{\{y\}} \Rightarrow z=x \text{ or } z=y}$.
\end{definition}

The graph of a space of type (A) looks as follows (see also Figure~\ref{fig:genform} on page~\pageref{fig:genform}):
\[\bullet \leftarrow \cdots \leftarrow \bullet \rightarrow \cdots \rightarrow \bullet \leftarrow \cdots \leftarrow \bullet \rightarrow \cdots \ \ \  \cdots \leftarrow \bullet \rightarrow \cdots \rightarrow \bullet. \]
In particular, every connected $T_0$-space with at most three points is of type~(A).
A~characterisation of type~(A) spaces in terms of unoriented edge degrees of the associated graphs is given in Lemma~\ref{lem: Char of type A}.

The main result of this paper now reads as follows; it is a consequence of Theorem~\ref{thm:final}, see also \S\ref{The UCT criterion}.
\begin{theorem}
Let $X$ be finite $T_0$-space. The following statements are equivalent:
\begin{enumerate}[label=\textup{(\roman*)}]
\item $X$ is a disjoint union of spaces of type \textup{(A)}.
\item  Let $A, B \inin \Bootstrap(X)$. Then $\FK(A) \cong \FK(B)$ implies that $A$ is $\KK(X)$-equiv\-a\-lent to $B$.
\item Let $A$ and $B$ be separable \Cstar algebras over $X$. Suppose $A\inin\Bootstrap(X)$. Then there is a natural short exact UCT sequence
\[
  \Ext^1_{\Nattrafo(X)}\bigl(\FK(A)[1],\FK(B)\bigr) \into
  \KK_*(X;A,B) \prto
  \Hom_{\Nattrafo(X)}\bigl(\FK(A),\FK(B)\bigr).
\]
\end{enumerate}
Here the subscript $\Nattrafo(X)$ denotes that $\Ext^1$ and $\Hom$ are taken in the category $\CMod{\Nattrafo(X)}$, the target category of~$\FK$.
\end{theorem}

\section{\texorpdfstring{$C^*$}{C*}-algebras over topological spaces}
Throughout this article, $X$ denotes a finite $T_0$-space.
In the following, we introduce  \Cstar algebras over~$X$
along the lines of~\cite{MR2545613}. The definition of a \Cstar algebra
over a topological space actually works in greater generality.

\subsection{Basic notions}
For a \Cstar algebra $A$ denote by $\Prim(A)$ its primitive ideal space.
A \emph{\Cstar{}algebra over~\(X\)} is a pair \((A,\psi)\)
consisting of a \Cstar{}al\-ge\-bra~\(A\) and a continuous map
\(\psi\colon \Prim(A)\to X\). 

Let $\Open(X)$ denote the set of open
subsets of~\(X\), partially ordered by~\(\subseteq\) and \(\Ideals(A)\)
the set of closed \Star{}ideals in~\(A\), partially ordered 
by~\(\subseteq\). The partially ordered sets $\left(\Open(X),\subseteq \right)$ and 
$\left(\Ideals(A),\subseteq \right)$ are complete lattices,
that is, any subset has both an infimum and
a supremum.  A continuous map \(\psi\colon \Prim(A)\to X\) induces a map 
$\psi^{\ast} \colon  \Open(X) \rightarrow \Ideals(A)$ which 
commutes with infima and suprema. 
By~\cite{MR2545613}*{Lemma~2.25}, this correspondence gives an equivalent
description of a \Cstar algebra over $X$ as a pair $(A,\psi^{\ast})$ where
\[\psi^{\ast} :  \Open(X) \rightarrow \Ideals(A), \ \ \ U \mapsto A(U)\]
commutes with infima and suprema.

A \Star{}homomorphism \(f\colon A\to B\) between two
\Cstar{}algebras over~\(X\) is \emph{\(X\)\nb-equivariant} if
\(f\bigl(A(U)\bigr)\subseteq B(U)\) for all \(U\in\Open(X)\).
The category of \Cstar{}algebras over~\(X\)
  with \(X\)\nb-equivariant \Star{}homomorphisms is denoted by \(\Cstarcat(X)\),
its  full subcategory consisting of all separable
  \Cstar{}algebras over~\(X\) is denoted by \(\Cstarsep(X)\).

A subset \(Y\subseteq X\) is \emph{locally closed} if and only
if \(Y=U\setminus V\) for open subsets \(V,U\in\Open(X)\) with
\(V\subseteq U\).  Then we define \(A(Y)\defeq A(U)/A(V)\) for
a \Cstar{}algebra~\(A\) over~\(X\); this does not depend on the
choice of \(U\) and~\(V\) by~\cite{MR2545613}*{Lemma~2.16}.
We write \(\Loclo(X)\) for the set of locally closed subsets of~\(X\).
By \(\Loclo(X)^*\) we denote the set of connected, non-empty locally closed
  subsets of~\(X\).

We write \(x\inOb\Cat\) for objects of a category~\(\Cat\) as opposed to morphisms.

\subsection{Functoriality}
A continuous map
  \(f\colon X\to Y\) induces a functor
  \[
  f_*\colon \Cstarcat(X) \to \Cstarcat(Y)
  \]
  which is given by \( (A,\psi) \mapsto (A,f \circ \psi) \).
We have \(g_*f_*=(gf)_*\) for composable continuous maps~$f$ and~$g$. 

 If \(f\colon X\to Y\) is the embedding of a subset with the
subspace topology, we also write
\(i_X^Y\) instead of \(f_*\) and call it \emph{extension}.

A locally closed subset \(Y \in \Loclo(X) \) induces the \emph{restriction} functor
\[
r_X^Y\colon \Cstarcat(X)\to\Cstarcat(Y)
\]
given by \((r_X^Y B)(Z)\defeq
B(Z)\) for all \(Z\in\Loclo(Y)\subseteq\Loclo(X)\). We have \(r_Y^Z \circ r_X^Y =r_X^Z\) if \(Z\subseteq
Y\subseteq X\) and \(r_X^X=\ID\).

Induction and restriction are related by \( r_X^Y\circ i_Y^X = \ID \) and various adjointness relations; see~\cite{MR2545613}*{Definition~2.19 and Lemma~2.20} for a discussion.

\subsection{Specialisation preorder}
There is the \emph{specialisation preorder} on~\(X\), defined by
 \(x\preceq y\) \(\iff\) \(\cl{\{x\}} \subseteq \cl{\{y\}}\). A subset $Y \subseteq X$ is locally closed if and only it is \emph{convex} with respect to $\preceq$, that is, if and only if
 $x \preceq y \preceq z$ and $x,z \in Y$ implies $y \in Y$ for all $x,y,z\in X$.
 A subset $Y \subseteq X$ has a \emph{locally closed hull} $LC(Y)$ defined as 
\[LC(Y) \coloneqq \{x \in X \mid \exists y_1, y_2 \in Y \colon y_1 \preceq x \preceq y_2 \}.\]

\begin{lemma}
We have $LC\bigb{LC(Y)} = LC(Y)$. Moreover, $LC(Y)$ is the smallest locally closed subset of~$X$ containing~$Y$.
\end{lemma}

\begin{proof}
Obviously $Y \subseteq LC(Y)$. Let $y \in LC\bigb{LC(Y)}$. Then there are $y_1,y_2 \in LC(Y)$ such that 
$y_1 \preceq y \preceq y_2$. By definition there are $z_1, z_2, z_3, z_4 \in Y$ such that $z_1 \preceq y_1 \preceq z_2$,  $z_3 \preceq y_2 \preceq z_4$. Hence $z_1\preceq y_1 \preceq y \preceq y_2 \preceq z_4$ and therefore $y \in LC(Y)$.
Using the characterization of locally closed subsets as convex subsets, the second statement is obvious.
 \end{proof}

A map $f: X_1 \rightarrow X_2$ between two finite topological spaces is continuous if and only if it is \emph{monotone} with respect to $\preceq$, that is, if 
 $ x \preceq y \Rightarrow f(x)\preceq f(y)$.
 
Note that $\preceq$ is
 a partial order if and only if $X$ is $T_0$.  By~\cite{MR2545613}*{Corollary~2.33}, this yields a bijection
 of $T_0$-topologies and partial orders on a given finite set. The preimage of a partial order~$\preceq$ is called the Alexandrov topology associated to~$\preceq$ and denoted by~$\tau_\preceq$.

\subsection{Representation as finite directed graphs}
\label{elementary notions of graph theory}
We describe a well-known way to represent finite $T_0$-spaces via finite \emph{directed acyclic graphs}. Several examples can be found in~\cite{MR2545613}*{\S2.8}.

To establish notation, we first collect a few elementary notions of graph theory:
a \emph{directed graph} is a tuple $\Gamma = (V,E)$, where $V$ is a set and $E \subseteq (V \times V) \setminus \Delta(V)$; elements of $V$ are called \emph{vertices} and elements of $E$ are called \emph{edges}. We will also write $E(\Gamma)$ and $V(\Gamma)$ to denote the edges and vertices associated to $\Gamma$. Hence we do neither allow loops nor multiple edges to exist.
A graph $(V',E')$ is a \emph{subgraph} of $(V,E)$ if and only if $V' \subseteq V$ and $E' = \{ (a,b) \in E \mid a,b \in V'\}$. 

A \emph{directed path} $\rho$ is a sequence $\rho = (v_i)_{i = 0,\ldots,n}$ such that $(v_i,v_{i+1}) \in E$ for $i =1,\ldots,n$ with all $(v_i)_{i = 1,\ldots,n}$ being pairwise distinct. The \emph{length} of $\rho = (v_i)_{i = 0,\ldots,n}$ is $n$. We say that $\rho$ is a path \emph{from $a$ to $b$} if $v_0 =a$ and $v_n =b$.

A \emph{directed cycle} is a directed path of length greater than $1$ such that $v_0 = v_n$. For two paths $\rho_1 = (v_i)_{i =0,\ldots,n}$ and $\rho_2 = (w_i)_{i =0,\ldots,m}$ we define sets
\[\rho_1 \cap \rho_2 \defeq \{v_i\mid i =0,\ldots,n\} \cap \{ w_i\mid i =0,\ldots,m\}\]
 and 
\[ \rho_1 \cup \rho_2 \defeq \{v_i\mid  i =0,\ldots,n\} \cup \{ w_i\mid i =0,\ldots,m\}.\]
An edge $(v_0,v_1)$ is called \emph{outgoing edge of $v_0$} and \emph{incoming edge of $v_1$}. The \emph{unoriented degree} $d(v)$ of $v \in V$ is defined as
\[ d(v) \defeq \#\{ e \in E\mid  e\text{ outgoing edge of }v\} + \#\{ e \in E\mid e\text{ incoming edge of }v\},\]
while the \emph{oriented degree} $d_o(v)$ of $v \in V$ is defined as
\[
d_o(v) \defeq \#\{ e \in E\mid e\text{ outgoing edge of }v\} - \#\{ e \in E\mid e\text{ incoming edge of }v\}.
\]
An \emph{undirected path} is a sequence $(v_i)_{i = 0,\ldots,n}$ such that for $i =1,\ldots,n$ either $(v_i,v_{i+1}) \in E$ or $(v_{i+1},v_i) \in E$ with all $(v_i)_{i = 1,\ldots,n}$ being pairwise distinct. We say that $\rho$ is an undirected path \emph{from $a$ to $b$} if $v_0 =a$ and $v_n =b$.

A \emph{cycle} is an undirected path $\rho = (v_i)_{i = 0,\ldots,n}$ of length greater than $0$ such that $v_0 =v_n$. 
A directed graph is called \emph{acyclic} if it has no cycles.

To a partial order~\(\preceq\) on~\(X\), we associate a finite directed acyclic graph $\Gamma(X)$ as follows. We write $x\prec y$ to denote that $x\preceq y$ and $x\neq y$.

\begin{definition}
	\label{def:Hasse}
Let $\Gamma(X)$ be the directed graph 
with vertex set~\(X\) and with an
edge \(x\leftarrow y\) if and only if \(x\prec y\) and there
is no \(z\in X\) with \(x\prec z\prec y\).
\end{definition}
In other words, $\Gamma(X)$ is the Hasse diagram corresponding to the specialisation order on~$X$.

We can recover the partial order from this graph by letting \(x\preceq y\) if and only if the graph contains a directed path from $y$ to $x$. This is the \emph{reachability relation} on the vertex set of $\Gamma(X)$, which makes sense for every finite directed acyclic graph.

Note that we cannot obtain every finite acyclic directed graph in this way. In fact, a finite directed acyclic graph is of the form $\Gamma(X)$ for some $T_0$-space $X$ if and only if it is \emph{transitively reduced}, that is, if it is (isomorphic to) the graph associated to its reachability relation (see, for instance,~\cite{AhoGareyUllman}).

For later reference, we list restrictions on $\Gamma(X)$ in the following lemma which follows directly from the definitions.

\begin{lemma}
\label{restriction}
The directed graph $\Gamma(X)$ is acyclic. Let $x, y$ be vertices in $\Gamma(X)$. If $\rho_1$ and $\rho_2$ are two distinct directed paths from $x$ to $y$. Then $\rho_1$ and $\rho_2$ have length at least~$2$.
\end{lemma}

Let $S$ be a finite set. If $\Gamma$ is a directed graph with vertex set $S$, then we can define a preorder on $S$ by setting $s_1 \preceq_{\Gamma} s_2$ if and only if there is a directed path from $s_2$ to $s_1$. Note that $\preceq_{\Gamma}$ is a partial order if and only if $\Gamma$ is acyclic.
Let $E(S)$ be the set of acyclic directed graphs with vertex set~$S$ having the following property: if $\rho_1$ and $\rho_2$ are two distinct directed paths in $\Gamma$ from $x$ to $y$, then $\rho_1$ and $\rho_2$ have length at least~$2$. It is easy to check that ${\preceq} \mapsto \Gamma(S,\tau_{\preceq})$ and $\Gamma \mapsto {\preceq_{\Gamma}}$ yield inverse bijections between the set of partial orders on $S$ and the set $E(S)$.

\begin{lemma}
\label{lem: graphs + connectedness}
A finite $T_0$-space~$X$ is connected if and only if $\Gamma(X)$ is connected as an undirected graph.
\end{lemma}

\begin{proof}
Assume first that $X$ is connected. Let $x_0 \in X$ and set 
\[ X_1 \defeq \{x \in X\mid \text{there is an undirected path from $x_0$ to $x$ in $\Gamma(X)$}\}. \]
Note that if $y \in \overline{\{x\}}$ then there is an undirected path from $x$ to $y$. Hence, if $x \in X_1$, then $\overline{\{x\}} \subseteq X_1$, therefore $\bigcup_{x \in X_1}\overline{\{x\}} = X_1$ and $X_1$ is closed. On the other hand, if $x \notin X_1$, then $ \overline{\{x\}} \subseteq X \setminus X_1$, hence
$X_1 = \bigcap_{x \notin X_1} X \setminus \overline{\{x\}}$ is open. Since $X$ is connected and $X_1$ is nonempty, we have $X =X_1$.

Now assume that $\Gamma(X)$ is connected as a graph and that $X =X_1 \sqcup X_2$ can be written as a disjoint union of nonempty clopen subsets $X_1$ and $X_2$. Let $x_i \in X_i$, $i =1,2$, and let $\rho$ be an undirected path from $x_1$ to $x_2$. We find neighbouring vertices $y_1$ and $y_2$ on the path $\rho$  such that $y_i \in X_i$ for $i=1,2$. Without loss of generality we may assume that $y_2 \in \overline{\{y_1\}}$. Since $X_1$ is closed we have $y_2 \in \overline{\{y_1\}} \subseteq X_1$ which is a contradiction.
\end{proof}

\section{Filtrated K-theory}
\subsection{Equivariant KK-theory}
As explained in~\cite{MR2545613}*{\S3.1}, there is a version of bivariant $\K$-theory for \Cstar algberas over $X$. Let $A, B \in \in \Cstarsep(X)$. A cycle in $\KK_*(X;A,B)$ is given by a cycle $(E,T)$ for $\KK_*(A,B)$ which is $X$-\emph{equivariant}, that is, $A(U)\cdot E \subseteq E \cdot B(U)$ for all $ U \in \mathbb{O}(X)$.
There is also a Kasparov product
\[\KK_*(X;A,B)\otimes \KK_*(X;B,C) \rightarrow \KK_*(X;A,C).\] 
Thus we may define the category $\KKcat(X)$ whose objects are separable \Cstar algebras over $X$ and morphisms from $A$ to $B$ are given by $\KK_0(X;A,B)$. As shown in~\cite{MR2545613}*{\S3.2}, $\KKcat(X)$ carries all basic structures we would expect from a bivariant $\K$-theory. In particular, it is additive, has countable coproducts, exterior products, satisfies Bott periodicity and has six-term exact sequences for \emph{se\-mi-split} extensions of \Cstar algebras over $X$.

Moreover, $\KKcat(X)$ carries the structure of a triangulated category~(\cite{MR2545613}*{\S3.3}); the suspension functor is given by the exterior product with $\CONT_{0}(\R)$ and a sequence $SB\rightarrow C\rightarrow A\rightarrow B$ is an exact triangle if and only if it is isomorphic to a mapping cone triangle $SB'\rightarrow C_{\phi}\rightarrow A'\rightarrow B'$ for some $X$-equivariant \Star homomorphism $\phi\colon A'\rightarrow B'$.

The \emph{bootstrap class} \(\Bootstrap(X)\) defined in~\cite{MR2545613}*{\S4} is the
localising subcategory of \(\KKcat(X)\) generated by the
objects \(i_x\C\) for all \(x\in X\).  That is, it is the
smallest class of objects containing these generators that is
closed under suspensions, \(\KK(X)\)-equivalence, se\-mi-split
extensions and countable direct sums. Here \(i_x\C\perdef i_{\{x\}}^X\C\),
where $\C$ is regarded as a \Cstar algebra over the one-point space
in the obvious way.

\subsection{The definition of filtrated K-theory}
We recall the definition of filtrated $\K$-theory from~\cite{meyernestCalgtopspacfiltrKtheory}*{\S4}.
  For each locally closed subset \(Y\subseteq X\), one defines a
  functor
  \[
  \FK(X)_Y\colon \KKcat(X)\to\Ab^{\Z/2},\qquad
  \FK(X)_Y(A) \defeq \K_*\bigl(A(Y)\bigr).
  \]
These functors are
\emph{stable} and \emph{homological}, that is, they intertwine
the suspension on \(\KKcat(X)\) with the translation functor on
\(\Ab^{\Z/2}\) and they map exact triangles to long exact sequences.

Let~\(\Nattrafo(X)\) be the \(\Z/2\)-graded category whose
  object set is~\(\Loclo(X)\) and whose morphism space \(Y \to Z\)
  is \(\Nattrafo_*(X)(Y,Z)\), the \(\Z/2\)-graded Abelian group of all natural transformations
  \(\FK_Y\Rightarrow \FK_Z\).

A \emph{module} over~\(\Nattrafo(X)\) is an additive, grading preserving
  functor \(G\colon \Nattrafo(X)\to\Ab^{\Z/2}\).
  Let \(\Mod{\Nattrafo(X)}\) be the category of
  \(\Nattrafo(X)\)\nb-modules.  The morphisms in
  \(\Mod{\Nattrafo(X)}\) are the natural transformations of
  functors or, equivalently, families of grading preserving
  group homomorphisms \(G_Y\to G'_Y\) that commute with the
  action of~\(\Nattrafo(X)\).  Let \(\CMod{\Nattrafo(X)}\) be the
  full subcategory of countable modules. 

\emph{Filtrated K\nb-theory} is the functor
  \[
  \FK(X)=\bigl(\FK(X)_Y\bigr)_{Y\in\Loclo(X)}\colon
  \KKcat(X) \to \CModsmallargument{\Nattrafo},
  \;
  A \mapsto
  \Bigl(\K_*\bigl(A(Y)\bigr)\Bigr)_{Y\in\Loclo(X)}.
  \]
To keep notation short, we often write $\NT$ for \(\Nattrafo(X)\) and $\FK$ for $\FK(X)$.
  
\begin{remark}
	\label{rem:connected}
Restriction to \emph{connected, non-empty} locally closed subsets
of~$X$ does not lose any relevant information: since~$X$ is finite,
every subset of~$X$ is the finite union of its connected components.
Moreover, this decomposition $Y=\bigsqcup_{i\in\pi_0(Y)} Y_i$ into
connected components corresponds to a biproduct decomposition
$Y\cong\bigoplus_{i\in\pi_0(Y)} Y_i$ in $\NT$ yielding a canonical
isomorphism
\[
G(Y)\cong\bigoplus_{i\in\pi_0(Y)} G(Y_i)\qquad
\textup{for all $Y\in\LC(X)$ and $G\in\Modcsmallargument{\NT}$.}
\]
In particular, the empty subset of~$X$ is a zero object in~$\NT$.

It follows that, denoting by $\NTC$ the full subcategory of $\NT$
consisting of connected, non-empty locally closed subsets
of~$X$, we have a canonical equivalence of categories
\[
\Upsilon\colon\Modcsmallargument{\NT}\to\Modcsmallargument{\NTC},
\]
which is just given by composing an $\NT$-module $M\colon\NT\to\gAb$
with the inclusion $\NTC\hookrightarrow\NT$. A pseudo-inverse $\Upsilon^{-1}$ is given by taking direct sums
over connected components of objects, that is, by
$\Upsilon^{-1}(G)(Y)\perdef\bigoplus_{i\in\pi_0(Y)} G(Y_i)$
on objects $Y\in\LC(X)$, and a similar direct sum operation on morphisms.
Hence, we can minimise our calculations by replacing filtrated $\K$-theory
with the reduced version $\FK^\con\perdef\Upsilon\circ\FK$.
\end{remark}

\subsection{Functoriality}
The canonical functor $\Cstarsep(X) \rightarrow \KKcat(X)$ is the universal split-exact, \Cstar stable functor~(\cite{MR2545613}*{Theorem~3.7}). Using this universal property, we may extend the functoriality results for $\Cstarcat(X)$ in the space variable to $\KKcat(X)$: a~continuous map \(f\colon X\to Y\) induces a functor $f_*\colon \KKcat(X) \to \KKcat(Y)$, in particular this yields an extension functor $i_X^Y$ for a subspace $X \subseteq Y$. Similarly, for $Y \in \Loclo(X)$ the restriction functor descends to a functor $r_X^Y\colon \KKcat(X) \rightarrow \KKcat(Y)$.

Our next aim is to construct an algebraic variant of $f_{\ast}$, that is, a functor
\[
f_{\ast}\colon\CMod{\Nattrafo(X)}\rightarrow\CMod{\Nattrafo(Y)}
\]
such that 
\begin{equation}
	\label{eq:functoriality}
\begin{split}
\begin{xy}
\xymatrix{
 \KKcat(X) \ar[rr]^-{\FK(X)} \ar[d]_{f_{\ast}} & &\CMod{\Nattrafo(X)} \ar[d]^{f_{\ast}}\\
\KKcat(Y) \ar[rr]^-{\FK(Y)} &  & \CMod{\Nattrafo(Y)}
}
\end{xy}
\end{split}
\end{equation}
commutes. Let us do so by first constructing a functor $f^{\ast}\colon\Nattrafo(Y) \rightarrow \Nattrafo(X)$.

For $Z \in \in \Nattrafo(Y) = \Loclo(Y)$ set $f^{\ast}(Z) = f^{-1}(Z)$. A morphism $\tau \in \Nattrafo(Y)(Z,Z')$ is a natural transformation $\tau \colon \FK(Y)_Z \rightarrow \FK(Y)_{Z'}$, that is, a collection $$\{ \tau_{A} \}_{A \in \in \KKcat(Y) }$$ of morphisms of Abelian groups 
\[\tau_{A}\colon \FK(Y)_Z(A) = \K_{\ast}\bigb{A(Z)} \rightarrow \K_{\ast}\bigb{A(Z')}= \FK(Y)_{Z'}(A )\] 
that is natural with respect to morphisms in $\Cstarcat(Y)$. 
For $B \in \in \KKcat(X)$ and $Z \in \Loclo (Y)$ we have 
\[\FK(Y)_Z(f_{\ast}B) = \K_{\ast}\Bigb{B\bigb{f^{-1}(Z)}} = \FK(X)_{f^{-1}(Z)}(B).\]
Hence $\tau_{f_{\ast}B}$ is also a morphism from $\FK(X)_{f^{-1}(Z)}(B)$ to $\FK(X)_{f^{-1}(Z')}(B)$ and it makes sense to define
\[f^{\ast}(\tau) \coloneqq \{ \tau_{f_{\ast}B} \}_{B \in \in \KKcat_X }.\]
We therefore have constructed an additive, grading preserving functor
\[ f^{\ast} \colon \Nattrafo(Y) \rightarrow \Nattrafo(X). \]
This gives rise to an additive, grading preserving functor
\[ f_{\ast} \colon \CMod{\Nattrafo(X)} \rightarrow \CMod{\Nattrafo(Y)}, \quad f_{\ast}(M) \coloneqq M \circ f^{\ast}.\]

\begin{lemma}
Let $X$, $Y$, $f$ and $f_{\ast}$ be as above. The diagram~\eqref{eq:functoriality} commutes.
\end{lemma}

\begin{proof}
Recall that there is a canonical functor $\KK(X)\colon \Cstarcat(X) \rightarrow \KKcat(X)$. 
By the universal property of $\KK(X)$ (see~\cite{MR2545613}*{Theorem~3.7}) we see that it suffices to check that 
\[  f_{\ast} \circ \FK(X) \circ \KK(X) = \FK ^Y\circ f_{\ast}\circ \KK(X).\]
On objects there is no difference anyway:
let $A\in \in \KKcat(X)$ and $Z \in \Loclo(Y)$. Then 
\[ f_{\ast} \circ \FK(X)(A)(Z) =\K_{\ast}\bigb{A(f^{-1}(Z)} = \FK ^Y\circ f_{\ast}(A)(Z).\]
Let $\phi\colon A \rightarrow B$ be a morphism of \Cstar algebras over $X$ and $Z \in \Loclo(Y)$. Passing to subquotients, $\phi$ induces a \Star homomorphism $\phi(Z')\colon A(Z') \rightarrow B(Z')$ for all $Z' \in \Loclo(X)$.

The push-forward $f_{\ast}(\phi)\colon f_{\ast}(A) \rightarrow f_{\ast}(B)$ is a morphism of \Cstar algebras over $Y$ which is given by $\phi$ as a \Star homomorphism from $A$ to $B$ if we forget the structure over $X$ (or $Y$). Note that $f_{\ast}(\phi)(Z) = \phi\bigb{f^{-1}(Z)}$ as \Star homomorphisms.
Now the equalities
\begin{multline*}
f_{\ast} \circ \FK(X)\circ \KK(X)(\phi)(Z) =  f_{\ast} \circ \FK(X)([\phi])(Z)\\
 = \FK(X)([\phi])\bigb{f^{-1}(Z)} = \K_{\ast}\Bigb{\phi\bigb{f^{-1}(Z)}}
\end{multline*} and
\begin{multline*}
\FK(Y)\circ f_{\ast}\circ \KK(X)(\phi)(Z) = \FK(Y)([f_{\ast}(\phi)])(Z)\\
= \K_{\ast}\bigb{f_{\ast}(\phi)(Z)} = \K_{\ast}\Bigb{\phi\bigb{f^{-1}(Z)}}
\end{multline*}
give the desired result.
\end{proof}

\subsection{Canonical transformations and relations}
\label{sec:relationsNT}
In this section we describe certain canonical elements and relations in the
category~$\NT$. Moreover, we provide indecomposability criteria for these
canonical transformations assuming that there are no ``hidden'' relations.
The results we establish will be used for concrete computations
in later chapters.

\begin{proposition}
	\label{pro:trans}
Let $U$ be a relatively open subset of a locally closed subset $Y$
of~$X$. Then there are the following natural transformations:
\begin{enumerate}[label=\textup{(\roman*)}]
\item
an even transformation
\[
i_U^Y\colon\FK_U\Rightarrow\FK_Y
\]
induced by the inclusion $A(U)\embed A(Y)$;
\item
an even transformation
\[
r_Y^{Y\setminus U}\colon\FK_Y\Rightarrow\FK_{Y\setminus U}
\]
incuced by the projection $A(Y)\epi A(Y\setminus U)$;
\item
an odd transformation
\[
\delta_{Y\setminus U}^U\colon\FK_{Y\setminus U}\Rightarrow\FK_U
\]
defined as the six-term sequence boundary map
$$
\K_*\bigb{A(Y\setminus U)}\to\K_{*+1}\bigb{A(U)}.
$$
\end{enumerate}
% Moreover, the compositions $r_Y^{Y\setminus U}\circ i_U^Y$,
% $\delta_{Y\setminus U}^U\circ r_Y^{Y\setminus U}$ and
% $i_U^Y\circ\delta_{Y\setminus U}^U$ vanish.
\end{proposition}

\begin{proof}
This is a consequence of the naturality %and exactness
of the six-term sequence in $\K$-theory associated to the ideal $A(U)\triangleleft A(Y)$.
\end{proof}

\begin{definition}
The natural transformations introduced in Proposition~\ref{pro:trans}
are called \emph{canonical} transformations in $\NT$.
We call $i_U^Y$ an \emph{extension transformation},
$r_Y^{Y\setminus U}$ a \emph{restriction transformation} and
$\delta_{Y\setminus U}^U$ a \emph{boundary transformation}.
\end{definition}

In all cases that have been investigated so far, the category $\NT$ is generated by these
canonical transformations. The absence of a general proof for
this phenomenon motivates the following definition.

\begin{definition}
Let $\NTS$ be the subcategory of $\NT$ generated by all canonical transformations, that is, by the set of morphisms
\[
\bigcup_{\genfrac{}{}{0pt}{}{Y\subset X\textup{\;locally closed,}}
	{U\subset Y\textup{\; relatively open}}}
	\left\lbrace i_U^Y, r_Y^{Y\setminus U},
	\delta_{Y\setminus U}^U\right\rbrace.
\]
Let $\NTSE$ be the subcategory of $\NTS$ generated by all \emph{even}
canonical transformations, that is, by the set of
morphisms
\[
\bigcup_{\genfrac{}{}{0pt}{}{Y\subset X\textup{\;locally closed,}}
	{U\subset Y\textup{\; relatively open}}}
	\left\lbrace i_U^Y, r_Y^{Y\setminus U}\right\rbrace.
\]
According to our previous convention, the respective full subcategories
with object set $\LC(X)^*$ are denoted by $\NTCS$ and $\NTCSE$.
Similarly, $\NT^\con_\textup{even}$ is the subcategory of $\NTC$
generated by even transformations.
\end{definition}

\begin{warning}
The subcategory $\NTSE$ of $\NTS$ need not exhaust the whole even part
of $\NTS$. However, this is true if any product of two odd natural
transformations vanishes. This fails to be true for the four-point
space~$S$ defined in \S\ref{sec:counterexamples} which was investigated
in~\cite{Rasmus}*{\S6.2}.
\end{warning}

The manifest elements of $\NT$ we have just discussed fulfill some
canonical relations, which we present in this section.
The following proposition investigates compositions of even
six-term sequence maps, that is, compositions in~$\NTSE$.

\begin{proposition}
	\label{pro:evenrels}
Let $Y$ be a locally closed subset of $X$.
\begin{enumerate}[label=\textup{(\roman*)}]
\item
Let $U$ be a relatively open subset of $Y$ and let $V$ be a relatively
open subset of $U$. Then $V$ is relatively open in $Y$ and
\[
i_U^Y\circ i_V^U = i_V^Y.
\]
Moreover, $i_Y^Y=\id_Y$.
\item
Let $C$ be a relatively closed subset of $Y$ and let $D$ be a relatively
closed subset of $C$. Then $D$ is relatively closed in $Y$ and
\[
r_C^D\circ r_Y^C = r_Y^D.
\]
Moreover, $r_Y^Y=\id_Y$.
\item
Let $U$ be a relatively open subset of $Y$ and let $C$ be a relatively
closed subset of $Y$. Then $U\cap C$ is relatively closed in $U$ and
relatively open in $C$, and
\[
r_Y^C\circ i_U^Y = i_{U\cap C}^C\circ r_U^{U\cap C} = \sum_{Z\in\pi_0(U\cap C)} i_Z^C\circ r_U^Z.
\]
In particular, if~$U$ and~$C$ are disjoint, then $i_Y^C\circ r_U^Y=0$.
\end{enumerate}
\end{proposition}

\begin{proof}
Assertions~(i) and~(ii) and the first equation in~(iii) follow from
respective identities on the \Cstar algebraic level.
The remaining assertions in~(iii) follow from the biproduct decomposition
in Remark~\ref{rem:connected} and the fact that the empty set is a zero object.
\end{proof}

\begin{definition}
A morphism $Y\to Z$ in a category $\Cat$ is called \emph{indecomposable}
if it cannot be written as a composite $Y\to W\to Z$ except for the
trivial ways involving identity morphisms.
\end{definition}

\begin{definition}
	\label{def:boundaries}
Let $Y\subset X$ be a subset. Since $X$ is finite there is a smallest
open subset $\widetilde Y$ of~$X$ containing~$Y$. This set is given
by the intersection of all open subsets of~$X$ containing~$Y$.

We define the boundary operations corresponding to the usual
and to the above closure operation by
\begin{equation*}
\overline\partial Y\perdef\overline Y\setminus Y\qquad\textup{and}\qquad
\widetilde\partial Y\perdef \widetilde Y\setminus Y.
\end{equation*}
\end{definition}

\begin{proposition}
	\label{pro:evenindecomp}
Let $Y$ be a connected, locally closed subset of~$X$. Suppose that the relations
in $\NTCSE$ are spanned by the canonical ones listed in
Proposition~\textup{\ref{pro:evenrels}.}
\begin{enumerate}[label=\textup{(\roman*)}]
\item
The natural transformation $i_U^Y$ for an open subset~$U$ of~$Y$ is
indecomposable in $\NTCSE$ if and only if $Y$ is of the form
\[
U\opencup y\perdef U\cup\left\lbrace x\in X\mid \textup{$x\succeq y$,
but $x\nsucc u$ for all $u\in U$}\right\rbrace 
\]
for a maximal element~$y$ of~$\overline\partial U$.
\item
The natural transformation $r_Y^C$ for a closed subset~$C$ of~$Y$ is
indecomposable in $\NTCSE$ if and only if~$Y$ is of the form
\[
C\closedcup y\perdef C\cup\left\lbrace x\in X\mid \textup{$x\preceq y$,
but $x\nprec c$ for all $c\in C$}\right\rbrace 
\]
for a minimal element~$y$ of~$\widetilde\partial C$.
\end{enumerate}
\end{proposition}

\begin{proof}
We prove (i). The second assertion follows in an analogous manner or
by considering the dual partially ordered set of~$X$.

Suppose that $i_U^Y$ is indecomposable in $\NTCSE$. Then $U$ is a
maximal connected proper open subset of $Y$ because otherwise $i_U^Y$
could be written as a composition of two proper extension
transformations, that is, extension transformations which are not
identity transformations.

We choose a minimal element~$y$ of
$\overline U\cap Y$. We may assume $y\in\overline\partial U$ because
otherwise~$U$ were a proper clopen subset of~$Y$ contradicting
connectedness of $Y$. Moreover, $y$
is a maximal element of $\overline\partial U$. To see this, assume
that there is $z\in\overline\partial U$ with $z\succ y$. Then $z\in Y$
because~$Y$ is locally closed. Hence $U\cup(\widetilde{\{z\}}\cap Y)$
is a \emph{proper} connected open subset of $Y$ containing~$U$ as a
proper subset. This contradicts our previous observation that~$U$ is a
maximal connected proper open subset of~$Y$.

We claim that $y$ is a least element of $Y$. Assume, conversely, that
there is $w\in Y$ with $w\nsucc y$. Then
$U\cup(\widetilde{\{y\}}\cap Y)$ is a \emph{proper} connected open
subset of $Y$ containing $U$ as a proper subset, which again yields a
contradiction. For this reason and since~$Y$ contains~$U$ as an open
subset, we have $Y\subset U\opencup y$.

Now we observe that~$Y$ is closed in $U\opencup y$---this holds for
every connected locally closed subset of $U\opencup y$ containing~$U$.
Hence $i_U^Y=r_{U\opencup y}^Y\circ i_U^{U\opencup y}$, and the
indecomposability of $i_U^Y$ implies $Y=U\opencup y$.

For the converse implication, let $Y=U\opencup y$ for a maximal element
$y$ of~$\overline\partial U$. Then $U$ is a maximal connected proper
open subset of~$Y$ and hence $i_U^Y$ does not decompose as the composite
of two proper extension transformations. On the other hand, $i_U^Y$ does
not decompose as $r_W^Y\circ i_U^W$ with $Y\subsetneq W$ either. To see
this, we assume the opposite: let $W$ be a connected locally closed
subset of $X$ containing $Y$ as a proper closed subset. Since $Y$ cannot
be open in $W$, there are $w\in W\setminus Y$ and $y'\in Y$ with
$w\succ y'$. Consequently, we either have $w\succ u$ for some $u\in U$,
or $w\succ y$. But, since $w\not\in Y=U\opencup y$, the inequality
$w\succ y$ implies $w\succ u$ for some $u\in U$ as well. This follows
from the definition of $U\opencup y$. Thus~$U$ is not
open in~$W$---a contradiction.
\end{proof}

In the following, we examine the category $\NTS$,
so that boundary transformations come into play.

\begin{definition}
A \emph{boundary pair} in $\NT$ is a pair $(U,C)$ of disjoint subsets
$U,C\in\LC(X)$ such that
\begin{itemize}
\item
$U\cup C$ is locally closed,
\item
$U$ is relatively open in $U\cup C$,
\item
$C$ is relatively closed in $U\cup C$.
\end{itemize}
\end{definition}

The third condition is of course redundant since it is equivalent
to the second one. Since local closedness is preserved under finite
intersections, $U$ and $C$ are locally closed. For each boundary pair
we have the natural transformation $\delta_C^U$ defined in
Proposition~\ref{pro:trans}.

We begin by investigating compositions of boundary transformations with \emph{even}
six-term sequence transformations.

\begin{proposition}
	\label{pro:oddrels}
Let $(U,C)$ be a boundary pair in $\NT$ and define $Y=U\cup C$.
\begin{enumerate}[label=\textup{(\roman*)}]
\item
Let $C'\subset C$ be a relatively open subset. Then $U\cup C'$ is
relatively open in $U\cup C$, the set $C'$ is relatively closed in
$U\cup C'$, and we have
\[
\delta_C^U\circ i_{C'}^C = \delta_{C'}^U.
\]
\item
Let $U'\subset U$ be a relatively closed subset. Then $U'\cup C$ is
relatively closed in $U\cup C$, the set $U'$ is relatively
open in $U'\cup C$, and
\[
r_U^{U'}\circ\delta_C^U = \delta_C^{U'}.
\]
\item
Let $U'$ be a subset of $U$ with the property that $U'\cup C$ is
relatively open in $U\cup C$. Then $U'$ is relatively open in $U$ and
in $U'\cup C$, and we have
\[
i_{U'}^U\circ\delta_C^{U'} = \delta_C^U.
\]
\item
Let $C'$ be a subset of $C$ with the property that $U\cup C'$ is
relatively closed in $U\cup C$. Then $C'$ is relatively closed in $C$
and in $U\cup C'$, and
\[
\delta_{C'}^U\circ r_C^{C'} = \delta_C^U.
\]
\end{enumerate}
\end{proposition}

\begin{proof}
This follows from the fact that $\K$-theoretic boundary
maps are natural with respect to morphisms of extensions.
\end{proof}

It is, however, not true that every morphism of extensions decomposes
as a composition of pullbacks and pushouts as above. To see this, consider
the morphism
\begin{equation}
    \label{eq:extensionmorphism}
\begin{split}
\xymatrix{
A(U)\ar@{ >->}[r]\ar@{ >->}[d] & A(Y)\ar@{->>}[r]\ar@{=}[d] & A(C)\ar@{->>}[d] \\
A(U')\ar@{ >->}[r] & A(Y)\ar@{->>}[r] & A(C')
}
\end{split}
\end{equation}
for appropriate boundary pairs $(U,C)$ and $(U',C')$ in~$X$.
This morphism need not split into pullbacks and pushouts because
$U\cup C'$ need not be locally closed, and the union~$U'\cup C$ need
not be disjoint. We phrase the relation corresponding to the above
morphism in the following proposition.

\begin{proposition}
	\label{pro:oddrel2}
Let $(U,C)$ and $(U',C')$ be boundary pairs in $\NT$ with $U\cup C=U'\cup C'$,
and such that $U$ is an open subset of $U'$ and $C'$ is a closed subset
of~$C$. Then
\[
i_U^{U'}\circ\delta_C^U=\delta_{C'}^{U'}\circ r_C^{C'}.
\]
In particular, $i_U^{U\cup C}\circ\delta_C^U=0$
and $\delta_{C}^{U}\circ r_{U\cup C}^{C}$
for every boundary pair $(U,C)$ in~$\NT$.
\end{proposition}

\begin{proof}
 The first assertion follows from the naturality of the boundary map in $\K$-theory with respect to the morphism of extensions~\eqref{eq:extensionmorphism}. For the second assertion, consider the boundary pairs $(U\cup C,\emptyset)$ and $(\emptyset,U\cup C)$ and use that~$\emptyset$ is a zero object.
\end{proof}

\begin{corollary}
	\label{pro:vanishingboundarycomp}
Let $Y,Z\in\LC(X)$ such that $W\perdef Y\cap Z$ is open in $Y$ and
closed in~$Z$. If $Y\cup Z$ is locally closed then
\[
\delta_W^{Z\setminus W}\circ\delta_{Y\setminus W}^W=0.
\]
\end{corollary}

\begin{proof}
By Proposition~\ref{pro:oddrels}(ii) we have
$\delta_{Y\setminus W}^W=r_Z^W\circ\delta_{Y\setminus W}^Z$.
Hence, by Proposition~\ref{pro:oddrel2},
\[
\delta_W^{Z\setminus W}\circ\delta_{Y\setminus W}^W
=\left(\delta_W^{Z\setminus W}\circ r_Z^W\right)
\circ\delta_{Y\setminus W}^Z=0.
\qedhere
\]
\end{proof}

\begin{proposition}
	\label{pro:vanishingsum}
Let $Y,Z\in\LC(X)$.
\begin{enumerate}[label=\textup{(\roman*)}]
\item
Let $Z$ be a proper open subset of~$Y$.
Let $C_1,\ldots,C_k$ be the connected components of~$Y\setminus Z$.
Then
\[
\sum_{j=1}^k \delta_{C_j}^Z\circ r_Y^{C_j} = 0.
\]
\item
Let $Y$ be a proper closed subset of~$Z$.
Let $C_1,\ldots,C_k$ be the connected components of~$Z\setminus Y$.
Then
\[
\sum_{j=1}^k i_{C_j}^Z\circ\delta_Y^{C_j} = 0.
\]
\end{enumerate}
\end{proposition}

\begin{proof}
Let $C\perdef Y\setminus Z$. Then $A(C)=\prod_{j=1}^k A(C_j)$ for every
$\Cst$-algebra~$A$ over~$X$ and we
get
\[
\sum_{j=1}^k \delta_{C_j}^Z\circ r_Y^{C_j}
=\delta_C^Z\circ \sum_{j=1}^k i_{C_j}^C\circ r_Y^{C_j}
=\delta_C^Z\circ r_Y^C
=0
\]
using Proposition~\ref{pro:oddrel2}. The second assertion follows analogously.
\end{proof}

\begin{definition}
	\label{def:canonical relations}
Restricting to \emph{connected, non-empty} locally closed closed subsets of~$X$, Propositions~\ref{pro:evenrels}, \ref{pro:oddrels}, \ref{pro:oddrel2} and~\ref{pro:vanishingsum} establish various relations in the category~$\NTC$. These are referred to as \emph{canonical relations in~$\NTC$.}
\end{definition}

\begin{remark}
Since we have ruled out the empty set as an object of the category~$\NTC$, the last relations of Propositions~\ref{pro:evenrels} and~\ref{pro:oddrel2}, respectively, become independent from the other relations in these propositions.
\end{remark}

\begin{remark}
In \S\ref{sec:Proof_of_ungraded_iso} we will see that, for a finite space~$W$ of type~(A), all relations in the category~$\NTC(W)$ follow from these canonical relations. This is also true for all four-point spaces considered in \S\ref{sec:counterexamples}.
\end{remark}

In the following we make some definitions in order to describe the
boundary pairs $(U,C)$ that correspond to \emph{indecomposable}
boundary transformations $\delta_C^U$ in $\NTCS$.

\begin{definition}
\emph{A boundary pair in $\NTC$} is a boundary pair $(U,C)$ in $\NT$
such that $U$, $C$ and $U\cup C$ are connected.
\end{definition}

\begin{definition}
	\label{def:boundarysubpair}
For two boundary pairs $(U,C)$ and $(U',C')$ in $\NTC$ we say that $(U',C')$ is an
\emph{extension} of $(U,C)$ if
\begin{itemize}
\item
$U$ is a relatively closed subset of $U'$,
\item
$C$ is a relatively open subset of~$C'$.
\end{itemize}
\end{definition}

\begin{example}
Let $X=\{1,2,3,4\}$ with Alexandrov topology given by the partial order generated by $1\prec 2\prec 3\prec 4$:
\[
\xymatrix{
4\ar[r] & 3\ar[r] & 2\ar[r] & 1.
}
\]
Then $\bigb{\{3,4\},\{1,2\}}$ is an extension of
$\bigb{\{3\},\{2\}}$.
\end{example}

\begin{lemma}
For an extension $(U',C')$ of $(U,C)$ we have the relation
\[
\delta_{C}^{U}=r_{U'}^{U}\circ\delta_{C'}^{U'}\circ i_{C}^{C'}.
\]
\end{lemma}

\begin{proof}
This follows immediately from Proposition~\ref{pro:oddrels}(i) and (ii).
\end{proof}

\begin{definition}
A boundary pair in $\NTC$ is called \emph{complete} if it
has no proper extension in $\NTC$.
\end{definition}

\begin{proposition}
	\label{pro:complete}
A boundary pair $(U,C)$ in $\NTC$ is complete if and only if~$U$ is open and~$C$ is closed.
\end{proposition}

\begin{proof}
Suppose that~$U$ is open and~$C$ is closed. Let $(V,D)$ be an extension
of~$(U,C)$. Then~$U$ is clopen in~$V$ and~$C$ is clopen in~$D$.
Since $V$ and $D$ are connected we get $U=V$ and $C=D$.

Conversely, let $(U,C)$ be complete. Assume that~$U$ is not open, so
that there is $b\in\widetilde\partial U$. Define $Y\perdef U\cup C$
and
$U'\perdef U\cup\bigb{\overline{\{b\}}\cap\widetilde\partial Y}\supsetneq U$.
We show that $(U',C)$ is an extension of $(U,C)$.

Recall that a subset of~$X$ is locally closed if and only if it is
convex with respect to the specialisation preorder.
The union
$U'\cup C=Y\cup\bigb{\overline{\{b\}}\cap\widetilde\partial Y}$ is
convex because $Y$ and $\overline{\{b\}}\cap\widetilde\partial Y$ are
convex, and if
$Y\ni y\prec x\prec z\in\overline{\{b\}}\cap\widetilde\partial Y$
for some $x\in X$ then $y\prec x\prec b$ and thus
$x\in\overline{\{b\}}\cap\widetilde Y\subset U'\cup C$. Note that the
situation
$Y\ni y\succ x\succ z\in\overline{\{b\}}\cap\widetilde\partial Y$
is impossible because~$Y$ is convex.

The subset $C\subset U'\cup C$ is closed. Otherwise, there were $c\in C$
and $z\in\overline{\{b\}}\cap\widetilde\partial Y$ with $c\succ z$.
Since $z\in\widetilde\partial Y$ there were $y\in Y$ with $z\succ y$,
and we get the contradiction $c\succ y$.

Up to now we have shown that $(U',C)$ is a boundary pair in~$X$.
It remains to show that $U$ is closed in $U'$. This is equivalent to
$\overline{\{b\}}\cap\widetilde\partial Y$ being open in~$U'$.
To see this, consider $z\in\overline{\{b\}}\cap\widetilde\partial Y$
and $w\in U'$ with $w\succ z$. Since $z\in\widetilde\partial Y$ there is
$y\in Y$ with $z\succ y$. Now $w\succ z\succ y$ implies $w\not\in Y$
since~$Y$ is convex. Consequently
$w\in \overline{\{b\}}\cap\widetilde\partial Y$.

This proves that $(U',C)$ is an extension of $(U,C)$. Finally, if~$C$ is
not closed, we can construct an extension $(U,C')$ of $(U,C)$ in a
similar fashion.
\end{proof}

\begin{definition}
	\label{def:subboundarypair}
For two boundary pairs $(U,C)$ and $(U',C')$ in $\NTC$ we say that $(U',C')$ is a
\emph{sub-boundary pair} of $(U,C)$ if
\begin{itemize}
\item
$U'$ is a (relatively open) subset of $U$,
\item
$C'$ is a (relatively closed) subset of~$C$,
\item
$U'\cup C$ is relatively open in $U\cup C$,
\item
$U\cup C'$ is relatively closed in $U\cup C$.
\end{itemize}
In fact, the assumptions that $U'$ be relatively open in
$U$ and that $C'$ be relatively closed in $C$ are redundant.
\end{definition}

\begin{example}
Let $X=\{1,2,3,4\}$ with Alexandrov topology given by the partial order generated by $1\prec 3$, $1\prec 4$ and
$2\prec 4$:
\[
\xymatrix{
3\ar[r]  & 1 &   4\ar[l]\ar[r] & 2. 
}
\]
Then $\bigb{\{4\},\{1\}}$ is a sub-boundary pair of
$\bigb{\{2,4\},\{1,3\}}$.
\end{example}

\begin{lemma}
For a sub-boundary pair $(U',C')$ in $(U,C)$ we have the relation
\[
\delta_C^U=i_{U'}^U\circ\delta_{C'}^{U'}\circ r_C^{C'}.
\]
\end{lemma}

\begin{proof}
By assumption $U'\cup C$ is open in $U\cup C$. This implies that
$U'\cup C'$ is open in $U\cup C'$ and thus
$\delta_{C'}^U=i_{U'}^U\circ\delta_{C'}^{U'}$ by
Proposition~\ref{pro:oddrels}(iii). The second step follows from
Proposition~\ref{pro:oddrels}(iv).
\end{proof}

\begin{definition}
A boundary pair in $\NTC$ is called \emph{reduced} if it has
no proper sub-boundary pair in $\NTC$.
\end{definition}

\begin{proposition}
	\label{pro:reduced}
A boundary pair $(U,C)$ in $\NTC$ is reduced if and only if
$\overline{U}\supset C$ and $\widetilde{C}\supset U$.
\end{proposition}

\begin{proof}
Suppose that $\overline{U}\supset C$ and $\widetilde{C}\supset U$, and
let $(V,D)$ be a sub-boundary pair of $(U,C)$. Set $Y\perdef U\cup C$.
Then, by definition, $V\cup C$ is open in $Y$ and hence
$V\cup C\supset \widetilde C\cap Y\supset U$. This shows $V=U$.
Analogously, $U\cup D\supset\overline U\cap Y\supset C$, so that
$C=D$.

Conversely, let $(U,C)$ be reduced. Assume that
$\widetilde C\not\supset U$. Define $U'\perdef U\cap\widetilde C$.
Then $\emptyset\neq U'\subsetneq U$.
We will show that $(U',C)$ is a sub-boundary pair of~$(U,C)$---this
yields a contradiction to the reducedness of~$(U,C)$.

The set $U'\cup C=Y\cap\widetilde C$ is locally closed as a finite intersection
of locally closed subsets and connected because $C$ is connected and
$C\subset Y$. Since $C$ is closed in $Y$ it is also closed in the subset
$U'\cup C$. This shows that $(U',C)$ is a boundary pair.

The subset $U'\cup C=Y\cap\widetilde C$ is open in $Y$ because
$\widetilde C$ is open in~$X$. Hence $(U',C)$ is a sub-boundary pair of
$(U,C)$.

Assuming, on the other hand, that $\overline U\not\supset C$, we find
the sub-boundary pair $(U,C\cap\overline U)$ of $(U,C)$.
\end{proof}

\begin{corollary}
	\label{cor:oddindecomp}
Let $(U,C)$ be a boundary pair in $\NTC$. Suppose that the relations in $\NTCS$ are
spanned by the canonical ones listed in Definition~\ref{def:canonical relations}.
Then the natural transformation~$\delta_C^U$
is indecomposable in $\NTCS$ if and only if~$U$ is open, $C$ is closed,
$\overline U\supset C$ and $\widetilde C\supset U$.
\end{corollary}

\begin{proof}
Under the assumption that the relations in $\NTCS$ are spanned by the
canonical ones, the natural transformation $\delta_C^U$ is indecomposable
if and only if the boundary pair $(U,C)$ is complete and reduced.
Hence the assertion follows from Propositions~\ref{pro:complete}
and~\ref{pro:reduced}. Notice that the relations in
Propositions~\ref{pro:oddrel2} and~\ref{pro:vanishingsum} cannot be used
to decompose the boundary transformation corresponding to a boundary pair.
\end{proof}

\subsection{The representability theorem and its consequences}
	\label{sec:representability}

The representability theorem is a powerful tool. It enables us to
describe the category~$\NT$ by computing plain topological $\K$-groups.
We follow~\cite{meyernestCalgtopspacfiltrKtheory}*{\S2.1}.

\begin{theorem}[Representability Theorem~\cite{meyernestCalgtopspacfiltrKtheory}*{Theorem~2.5}]
	\label{thm:representability}
Let $Y$ be a locally closed
subset of~$X$. The functor $\FK_Y$ is representable; more precisely, there are
a separable \Cstar algebra over~$X$ such that $\Rep_Y(Y)$ is unital and a natural isomorphism
\[
\KK_*(X;\Rep_Y,\blank)\cong\FK_Y
\]
given by 
\[
\KK_*(X;\Rep_Y,A)\to\FK_Y(A),\quad f\mapsto f_*\bigb{[1_{\Rep_Y(Y)}]}
\]
for all $A\inin\kk(X)$.
Here $[1_{\Rep_Y(Y)}]$ is the class of the unit
element~$1_{\Rep_Y(Y)}$ of $\Rep_Y(Y)$ in
$\FK_Y(\Rep_Y)$, and $f_*=\FK_Y(f)$.
\end{theorem}

Let $\Ch(X)$ denote the order complex corresponding to the specialisation
preorder on~$X$ as defined in~\cite{meyernestCalgtopspacfiltrKtheory}*{\S2}.
This order complex comes with two functions $m,M\colon\Ch(X)\to X$ with the property
that the map
\[
(m,M)\colon\Ch(X)\to X^\op\times X\]
is continuous.
Here $X^\op$ denotes the topological space whose underlying set is $X$
and whose open subsets are the closed subsets of~$X$.

The primitive ideal space of the commutative $\Cst$-algebra
\[
\Rep\perdef\Cont\bigl(\Ch(X)\bigr)
\]
is $\Ch(X)$. Hence the map $(m,M)$ turns $\Rep$ into a $\Cst$-algebra
over~$X^\op\times X$.
For locally closed subsets $Y$, $Z$ of $X$, we define
\[
S(Y,Z)\perdef m^{-1}(Y)\cap M^{-1}(Z)\subset\Ch(X).
\]
This is a locally closed subset of $\Ch(X)$.

\begin{definition}
	\label{def:repobjects}
Let $Y$ be a locally closed subset of~$X$.
We define $\Rep_Y$ to be the restriction of $\Rep$ to $Y^\op\times X$,
regarded as a $\Cst$-algebra over $X$ via the coordinate projection
$Y^\op\times X\to X$.
More explicitly, we have
\begin{equation}
	\label{eq:defrepobjects}
\Rep_Y(Z)=\Rep(Y^\op\times Z)=\Cont_0\bigl(S(Y,Z)\bigr).
\end{equation}
\end{definition}

\begin{lemma}[\cite{meyernestCalgtopspacfiltrKtheory}*{Lemma~2.14}]
If $Y,Z\in\LC(X)$, then
\[
S(Y,Z)=\Ch(\widetilde Y\cap\overline Z)\setminus\bigb{\Ch(\widetilde Y
\cap\overline\partial Z)\cup\Ch(\widetilde\partial Y\cap\overline Z)}.
\]
\end{lemma}

An application of the Yoneda Lemma yields graded Abelian group isomorphisms
\[
\NT_*(Y,Z)\cong\KK_*(X;\Rep_Z,\Rep_Y)\cong\FK_Z(\Rep_Y)
\cong\K_*\bigb{\Rep_Y(Z)}%=\K_*\bigb{\Rep(Y^\op\times Z)}
\cong\K^*\bigb{S(Y,Z)}.
\]
However, it is not obvious how to express the composition of natural
transformations
\[
\NT_*(Y,Z)\times\NT_*(W,Y)\to\NT_*(W,Z)
\]
directly in terms of these topological $\K$-groups. In principle, it is of course
always possible to lift elements back to the respective $\KK$-groups
and then compose them.

We have identified natural transformations
$\FK_Y\Rightarrow\FK_Z$ with $\KK(X)$-morphisms $\Rep_Z\to\Rep_Y$ and
with classes of vector bundles over the topological space~$S(Y,Z)$.
Now we explicitly describe the $\FK$- and $\KK(X)$-elements corresponding
under the above identifications to compositions of the natural transformations
introduced in Proposition~\ref{pro:trans}.

Let $Y\in\LC(X)$ and let $U\subset Y$ be an open subset. Then
$U^\op\times Z$ is a closed subset of $Y^\op\times Z$ and
$(Y\setminus U)^\op\times Z$ is an open subset of $Y^\op\times Z$
for every $Z\in\LC(X)$. By~\eqref{eq:defrepobjects}, we have an
extension of $\Cst$-algebras
$\Rep_{Y\setminus U}(Z)\mono\Rep_Y(Z)\epi\Rep_U(Z)$
for every $Z\in\LC(X)$. This, in turn, is nothing but an extension
$\Rep_{Y\setminus U}\mono\Rep_Y\epi\Rep_U$
of $\Cst$-algebras over $X$. Since $\Rep_Y$ is commutative and therefore
nuclear, this extension is se\-mi-split and hence has a class in
$\KK_1(X;R_U,R_{Y\setminus U})$ which produces an exact triangle
\begin{equation}
	\label{eq:exacttrianglerepobjects}
\xymatrix{
\Sigma\Rep_U\ar[r] & \Rep_{Y\setminus U}\ar[r] & \Rep_Y\ar[r] &\Rep_U
}
\end{equation}
in $\kk(X)$.

We recall a lemma from~\cite{meyernestCalgtopspacfiltrKtheory} which we will refine in the remainder of this section.
\begin{lemma}[\cite{meyernestCalgtopspacfiltrKtheory}*{Lemma~2.19}]
	\label{lem:correspondences}
Let $Y\in\LC(X)$, let $U\in\Open(Y)$, and set $C\perdef Y\setminus U$.
In the notation of Proposition~\textup{\ref{pro:trans}} and within the
meaning of the above correspondences,
\begin{enumerate}[label=\textup{(\roman*)}]
\item
the transformation $i_U^Y\colon\FK_U\Rightarrow\FK_Y$ corresponds to
the class of \[\Rep_Y\epi\Rep_U\] in $\KK_0(X;\Rep_Y,\Rep_U)$ and to the
class of the trivial rank-one vector bundle in
$\K^0\bigb{S(U,Y)}=\K^0\bigb{\Ch(U)}$;
\item
the transformation
$r_Y^{C}\colon\FK_Y\Rightarrow\FK_{C}$
corresponds to the class of \[\Rep_{C}\mono\Rep_Y\] in
$\KK_0(X;\Rep_{C},\Rep_Y)$ and to the class of the trivial
rank-one vector bundle in
$\K^0\bigb{S(Y,C)}=\K^0\bigb{\Ch(C)}$;
\item
the transformation
$\delta_{C}^U\colon\FK_{C}\Rightarrow\FK_U$
corresponds to the class of the extension
\[\Rep_{C}\mono\Rep_Y\epi\Rep_U\] in $\KK_1(X;\Rep_U,\Rep_C)$
and to the class $f^*(\upsilon)$ in
$\K^1\bigb{S(C,U)}=\K^1\Bigb{\Ch(Y)\setminus\bigb{\Ch(U)\sqcup\Ch(C)}}$,
where $\upsilon$ denotes a generator of the group
$\K^1\bigb{\left(0,1\right)}\cong\Z$ and $f$ is a continuous map
$\Ch(Y)\to\left[0,1\right]$ with $f^{-1}(0)=\Ch(U)$ and $f^{-1}(1)=\Ch(C)$.
\end{enumerate}
\end{lemma}

% \begin{corollary}
% 	\label{cor:evengenerators0}
% \begin{enumerate}[label=\textup{(\roman*)}]
% \item
% If $U$ is open in $Y\in\LC(X)$ and $\K^0\bigb{S(U,Y)}\cong\Z$, then
% $\NT_0(U,Y)$ is generated by the natural transformation $i_U^Y$.
% \item
% If $C$ is closed in $Y\in\LC(X)$ and $\K^0\bigb{S(Y,C)}\cong\Z$, then
% $\NT_0(Y,C)$ is generated by the natural transformation $r_Y^C$.
% \end{enumerate}
% \end{corollary}
% 
% \begin{proof}
% Both assertions follow from the fact that the $\K^0$-group of
% a compact space is generated by the class of the trivial rank-one
% vector bundle once it is isomorphic to~$\Z$.
% \end{proof}

\begin{lemma}
	\label{lem:intersection_decomposition}
Let $Y$ and $Z$ be locally closed subsets of~$X$, and let
$C\subseteq Y\cap Z$ be closed in~$Y$ and open in~$Z$. The transformation
$i_C^Z\circ r_Y^C\colon\FK_Y\Rightarrow\FK_Z$ corresponds to the class
of the composition
$\Rep_Z\epi\Rep_C\mono\Rep_Y$ in $\KK_0(X;\Rep_Z,\Rep_Y)$ and to the
class $[\xi_C]$ in $\K^0\bigb{S(Y,Z)}$ which is induced by 
rank-one trivial vector bundle~$\xi_C$ on the compact open subspace
$\Ch(C)\subset S(Y,Z)$.
\end{lemma}

\begin{proof}
It is a consequence of Lemma~\ref{lem:correspondences} that
$i_C^Z\circ r_Y^C$ corresponds to the composition
$\Rep_Z\epi\Rep_C\mono\Rep_Y$.
Since $(r_Y^C)_{\Rep_Y}\colon\Rep_Y(Y)\to\Rep_Y(C)$ is the restriction
$\Cont\bigb{\Ch(Y)}\epi\Cont\bigb{\Ch(C)}$ and
$(i_C^Z)_{\Rep_Y}\colon\Rep_Y(C)\to\Rep_Y(Z)$ is the embedding
$\Cont\bigb{\Ch(C)}\mono\Cont_0\bigb{S(Y,Z)}$, the trivial rank-one
bundle on $\Ch(Y)$ is restricted to $\Ch(C)$ and then extended by 0
to $S(Y,Z)$.
\end{proof}

Let~$Y$ and~$Z$ be locally closed subsets of~$X$. Since the property of
being relatively closed in~$Y$ and relatively open in~$Z$ is preserved
under finite unions, there is a \emph{maximal} subset $R(Y,Z)$ of
$Y\cap Z$ with this property.

\begin{corollary}
	\label{cor:evengen}
The monomials $i_C^Z\circ r_Y^C$, where $Y$ and $Z$ are locally closed
subsets of $X$, and $C$ is a connected component of $R(Y,Z)$, form a
$\Z$-basis of the category $\NTSE$.
\end{corollary}

\begin{proof}
Every morphism in $\NTSE$ is a $\Z$-linear combination of monomials in
composable extension and restriction transformations. The relations
given in Proposition~\ref{pro:evenrels} show that such a monomial
can be rewritten as $i_D^Z\circ r_Y^D$ for locally closed subsets
$D$, $Y$ and $Z$ of $X$, such that $D$ is a closed subset of $Y$ and
an open subset of $Z$. In this case, $D$ is a clopen subset of
$Y\cap Z$, and therefore a union of connected components of $R(Y,Z)$.
Hence $i_D^Z\circ r_Y^D$ is the sum of the transformations
$i_C^Z\circ r_Y^C$, where $C$ runs through the connected components
of $R(Y,Z)$ contained in~$D$.

The independence of the above-mentioned transformations is a
consequence of Lemma~\ref{lem:intersection_decomposition} and
the existence of a natural surjective homomorphism from~$\K^0$
onto the zeroth cohomology group with compact support, carrying
the trivial rank-one bundles on compact connected components to
independent generators: the dimension function.
\end{proof}

\begin{corollary}
	\label{cor:evengenerators}
Let $Y$ and $Z$ be locally closed subsets of~$X$, let
$Y\cap Z$ be closed in~$Y$ and open in~$Z$, and let~$n$ be the
number of connected components of $Y\cap Z$. If
$\K^0\bigb{S(Y,Z)}\cong\Z^n$ then $\NT_0(Y,Z)$ is generated by the
natural transformations $i_C^Z\circ r_Y^C$ with
$C\in\pi_0\bigb{Y\cap Z}$.
\end{corollary}

\begin{proof}
This follows from the observation that, in the above situation, the group
$\K^0\bigb{S(Y,Z)}=\K^0\bigb{\Ch(Y\cap Z)}$ is generated by the
classes of the trivial rank-one bundles~$\xi_C$ on
$\Ch(C)\subset\Ch(Y\cap Z)$ with $C\in\pi_0\bigb{Y\cap Z}$.
\end{proof}

\begin{warning}
	\label{war:notgenerator}
It is in general not true that the group $\NT_1(C,U)$ for a boundary
pair~$(U,C)$ is generated by~$\delta_C^U$ once it is isomorphic to~$\Z$;
a counterexample is given in~\cite{Rasmus}*{3.3.19}.
\end{warning}

\begin{lemma}
  \label{lem:oddgenerators}
Let $(U,C)$ be a boundary pair in $\NT$, and let $U',C'\in\LC(X)$ such that
$U$ is an open subset of~$U'$ and $C$ is a closed subset of~$C'$.
The transformation
$i_U^{U'}\circ\delta_C^U\circ r_{C'}^C\colon\FK_{C'}\Rightarrow\FK_{U'}$
corresponds to the composition
\[
\xymatrix{\Rep_{U'}\ar@{->>}[r] & \Rep_U\ar[r]|\circ &
\Rep_C\ar@{ >->}[r] & \Rep_{C'}}
\]
and to the class
$\left(i_{S(C,U)}^{S(C',U')}\circ f\right)^*(\upsilon)$ in
$\K^1\bigb{S(C',U')}$, where $f$ is defined as in
Lemma~\textup{\ref{lem:correspondences}(iii).}
\end{lemma}

\begin{proof}
The first assertion follows from Lemma~\ref{lem:correspondences}.
For the second assertion, note that $S(C,U)$ is open in~$S(C',U)$.
This follows from the definition $S(Y,Z)\perdef m^{-1}(Y)\cap M^{-1}(Z)$
because $m$ is continuous as a map from~$\Ch(X)$ to~$X^\op$.
We get the following commutative diagram:
\[
\xymatrix{
&
\K^0\bigb{S(C,C)}\ar@{=}[d]\ar[r]^\delta &
\K^1\bigb{S(C,U)}\ar[d]^i &
\\
\K^0\bigb{S(C',C')}\ar[r]^r &
\K^0\bigb{S(C',C)}\ar[r]^\delta &
\K^1\bigb{S(C',U)}\ar[r]^i &
\K^1\bigb{S(C',U')}.
}
\]
The class $[\xi_{C'}]\in\K^0\bigb{S(C',C')}$ is mapped to $[\xi_{C}]\in\K^0\bigb{S(C,C)}$, to $f^*(\upsilon)\in\K^1\bigb{S(C,U)}$ and, finally, to
$\left(i_{S(C,U)}^{S(C',U')}\right)^*\bigb{f^*(\upsilon)}\in\K^1\bigb{S(C',U')}$.
\end{proof}

\begin{corollary}
	\label{cor:oddgenerators}
Let $(U,C)$ be a boundary pair in $\NT$, and let $U',C'\in\LC(X)$ such
that~$U$ is an open subset of~$U'$ and $C$ is a closed subset of~$C'$.
Assume that $\K^1\bigb{S(C',U')}\cong\Z$ and further that the composition
\[
\xymatrix{
\K^0\bigb{S(C,C)}\ar[r]^\delta & \K^1\bigb{S(C,U)}\ar[r]^i &
\K^1\bigb{S(C',U')}
}
\]
maps the class of the trivial rank-one bundle in $\K^0\bigb{S(C,C)}$ to
a generator of $\K^1\bigb{S(C',U')}$. Then $\NT_1(C',U')$ is generated
by the composition $i_U^{U'}\circ\delta_C^U\circ r_{C'}^C$.

The above condition is fulfilled, in particular, when
$\K^0\bigb{S(C,C)}$ is isomorphic to~$\Z$ and the groups
$\K^1\bigb{S(C,C)\cup S(C,U)}$ and
$\K^1\bigb{S(C',U')\setminus S(C,U)}$ vanish.
\end{corollary}

\begin{proof}
The main assertion is a direct consequence of Lemma~\ref{lem:oddgenerators}. The addendum follows from the exactness of the six-term sequence in $\K$-theory.
\end{proof}

\section{The UCT criterion}
\label{The UCT criterion}

Theorem 4.8 in~\cite{meyernestCalgtopspacfiltrKtheory} shows what is actually
needed to obtain a UCT short exact sequence which computes
$\KK(X, \blank, \blank)$ in terms of filtrated $\K$-theory:
\begin{theorem}
  \label{the:UCT_conditional}
  Let \(A,B\inOb\KKcat(X)\).  Suppose that
  \(\FK(X)(A)\inOb\CMod{\Nattrafo(X)}\) has a projective resolution
  of length~\(1\) and that \(A\inOb\Bootstrap(X)\).  Then there
  are natural short exact sequences
  \begin{multline*}
  \Ext^1_{\Nattrafo(X)}\bigl(\FK(X)(A)[j+1],\FK(B)\bigr)
  \into \KK_j(X;A,B) \\
  \prto \Hom_{\Nattrafo(X)}\bigl(\FK(X)(A)[j],\FK(X)(B)\bigr)
  \end{multline*}
  for \(j\in\Z/2\), where \(\Hom_{\Nattrafo(X)}\) and
  \(\Ext^1_{\Nattrafo(X)}\) denote the morphism and extension groups
  in the Abelian category \(\CMod{\Nattrafo(X)}\).
\end{theorem}

Since we are interested in determining the class of spaces that allow for a UCT
short exact sequence for filtrated $\K$-theory, it makes sense to view
the crucial assumption in the theorem above as a property of the space~$X$.

\begin{definition}
Let~$X$ be a finite $T_0$-space. We say that \emph{$UCT(X)$ holds}
if for all $A \inOb\Bootstrap(X)$, \(\FK(X)(A)\inOb\CMod{\Nattrafo(X)}\)
has a projective resolution of length~\(1\).
\end{definition}

Let us also mention an important conclusion which can be drawn
from the existence of a UCT short exact sequence.

\begin{corollary}[\cite{meyernestCalgtopspacfiltrKtheory}*{Corollary~4.9}]
  \label{cor:lift_iso}
  Let \(A,B\inOb\Bootstrap(X)\) and suppose that both
  \(\FK(A)\) and \(\FK(B)\) have projective resolutions of
  length~\(1\) in \(\CModsmallargument{\Nattrafo}\). Then any homomorphism
  \(\FK(A)\to\FK(B)\) in \(\CModsmallargument{\Nattrafo}\) lifts to an
  element in \(\KK_0(X;A,B)\), and an isomorphism
  \(\FK(A)\cong\FK(B)\) lifts to an isomorphism in
  \(\Bootstrap(X)\).
\end{corollary}
 
As indicated in the introduction, the possibility of lifting isomorphisms in filtrated $\K$-theory to isomorphisms in $\KKcat(X)$ is one of the main reasons why one is interested in a UCT short exact sequence. On the other hand, the impossibility of lifting isomorphisms in $\FK(X)$ can of course be viewed as an obstruction to the existence of a UCT short exact sequence.

\begin{definition}
Let $X$ be a finite $T_0$-space. We say that \emph{$\neg UCT(X)$ holds} if there
are $A,B\in \in \Boot(X)$ such that $A \ncong B$ in $ \KKcat(X)$ and $\FK(X)(A) \cong \FK(X)(B)$ in $\CMod{\Nattrafo(X)}$.
\end{definition}

It is clear that there is no finite $T_0$-space such that both $UCT(X)$ and $\neg UCT(X)$ hold. Moreover,
as suggested by the notation, we will show that for every such~$X$ either $UCT(X)$ or $\neg UCT(X)$ holds.
We may restrict attention to connected spaces:

\begin{lemma}
\label{lem: UCT + connectedness}
Let $X$ be a finite $T_0$-space which is a disjoint union of topological spaces $X_1, \ldots, X_n$. Then
$UCT(X)$ holds if and only if $UCT(X_i)$ holds for $i =1, \ldots ,n$. Similarly, $\neg UCT(X)$ holds if and only if $\neg UCT(X_i)$ holds for all~$X_i$.
\end{lemma}

\begin{proof}
This is a consequence of the natural product decompositions $\kk(X)\cong\prod_{i=1}^n\kk(X_i)$ and $\CMod{\Nattrafo(X)}\cong\prod_{i=1}^n\CMod{\Nattrafo(X_i)}$, which are compatible with filtrated $\K$-theory.
\end{proof}

The next proposition roughly tells us that, if $X$ has a subspace for which there is no UCT, then there cannot exist a UCT for $X$ either.

\begin{proposition}
    \label{embedding}
Let $X$ be a finite $T_0$-space.
\begin{enumerate}[label=\textup{(\roman*)}]
\item
Let $Y\in\Loclo(X)$ such that $\neg UCT(Y)$ holds. Then $\neg UCT(X)$ holds as well.
\item
Let~$Y$ be a finite $T_0$-space and $f\colon X \rightarrow Y$, $g\colon Y \rightarrow X$ continuous maps with $ f \circ g = \ID_Y$. Suppose that $\neg UCT(Y)$ holds. Then $\neg UCT(X)$ holds as well.
\end{enumerate}
\end{proposition}

\begin{proof}
By assumption there are $A,B\inOb \Bootstrap(Y)$ such that $A\ncong B$ in $\KKcat(Y)$ and $\FK(Y)(A)\cong\FK(Y)(B)$. As already noted above, we have $r_X^Y\circ i_Y^X =\ID_Y$ (see also~\cite{MR2545613}*{Lemma~2.20(c)}); therefore $i_Y^X(A) \ncong i_Y^X(B)$ in $\KKcat(X)$. Recall that $i_Y^X$ is just $\iota_{\ast}$ for the embedding $\iota\colon Y \hookrightarrow X$. Hence $\FK(X)\bigb{i_Y^X(A)} = \iota_{\ast}\bigb{\FK(Y)(A)} \cong \iota_{\ast}\bigb{\FK(Y)(B)} = \FK(X)\bigb{i_Y^X(B)}$. The bootstrap $\Bootstrap(Y)$ is generated by $i_y\C, y \in Y$, and $i_Y^X \circ i_y\C =i_y \C $; therefore $ i_Y^X \Bootstrap(Y) \subseteq \Bootstrap(X)$. This shows the first statement.

To prove the second statement let $A, B \inOb \Bootstrap(Y)$ such that $A \ncong B$ in $\KKcat(Y)$ and $\FK(Y)(A) \cong \FK(Y)(B)$. Since $f_{\ast} \circ g_{\ast} =\ID_{\KKcat(Y)}$ we have that $g_{\ast} (A) \ncong g_{\ast} (B)$. $g_{\ast}  \circ \FK(Y) =\FK(X) \circ g_{\ast} $ implies $\FK(X)\bigb{g_{\ast} (A)} \cong \FK(X)\bigb{g_{\ast} (B)}$. Since $g_{\ast} i_y \C = i_{g(y)} \C$, we have $g_{\ast} \Bootstrap(Y) \subseteq \Bootstrap(X)$.
\end{proof}

\section{Positive results}
	\label{sec:positive results}
Let~$W$ be a finite $T_0$-space of type~(A).
In this section we show that $UCT(W)$ holds, that is,
we prove that the filtrated $\K$-module $\FK(A)$ has a projective resolution of length 1 in $\CMod{\Nattrafo(W)}$ for every $A\inOb\Bootstrap(W)$. This proof was given in~\cite{Rasmus}; it relies on methods developed in~\cite{meyernestCalgtopspacfiltrKtheory}: our argument is based on a comparison of the category $\NTC(W)$ and the category $\NTC(O_n)$ associated to the totally ordered space $O_n$ of the same cardinality as~$W$ (see Theorem~\ref{thm:ungraded_iso}).

\begin{definition}
	\label{def:freemodule}
For $Y\in\LC(W)^*$ we define the \emph{free $\NTC(W)$-module on $Y$} by
\[
P_Y(Z)\perdef\NT_*(Y,Z)\qquad\textup{for all $Z\in\LC(W)^*$.}
\]
An $\NTC(W)$-module is called \emph{free} if it is isomorphic to a direct
sum of degree-shifted free modules $P_Y[j]$, $j\in\Z/2$.
\end{definition}

\begin{definition}
	\label{def:exactmodule}
An $\NT(W)$-module $M$ is called \emph{exact} if the $\Z/2$-graded chain complexes
\[
\cdots\to M(U)\xrightarrow{i_U^Y} M(Y)\xrightarrow{r_Y^{Y\setminus U}}
M(Y\setminus U)\xrightarrow{\delta_{Y\setminus U}^U} M(U)[1]\to\cdots
\]
are exact for all $U,Y\in\LC(W)$ with~$U$ open in~$Y$.

An $\NTC(W)$-module $M$ is called \emph{exact} if the corresponding
$\NT(W)$-module $\Upsilon^{-1}(M)$ is exact (see Remark~\ref{rem:connected}).
\end{definition}

\begin{definition}
Let $\NTnil\subset\NTC$ be the ideal generated by all
natural transformations between different objects.

Let $\NTss\subset\NTC(W)$ be the subgroup spanned by all identity
transformations $\id_Y^Y$, that is,
$\NTss\perdef\bigoplus_{Y\in\LC(W)^*}\Z\cdot\id_Y^Y$.
This is, in fact, a se\-mi-simple subring of $\NTC$.
\end{definition}

% The subgroup $\NTss$ is in fact a semi-simple subring
% of $\NTC$---semi-simple in the sense that it is isomorphic to a
% direct sum of copies of~$\Z$---namely, $\NTss\cong\Z^{\LC(W)^*}$.

\begin{definition}
	\label{def:semisimplepart}
Let $M$ be an $\NTC$-module. We define
\[
\NTnil\cdot M\perdef\{x\cdot m\mid x\in\NTnil,\: m\in M\},\qquad
\Mss\perdef M/\NTnil\cdot M.
\]
An $\NTC$-module is called \emph{entry-free} if $M(Y)$ is a free Abelian group
for all $Y\in\LC(W)^*$.
\end{definition}

\begin{lemma}
	\label{lem:projmodules}
Let $M$ be an $\NTC(W)$-module. The following assertions are equivalent:
\begin{enumerate}[label=\textup{(\arabic*)}]
\item
$M$ is a free $\NTC(W)$-module.
\item
$M$ is a projective $\NTC(W)$-module.
\item
$\Mss(Y)$ is a free Abelian group for all $Y\in\LC(W)^*$ and
$$
\Tor_1^{\NTC(W)}(\NTss,M)=0.
$$
\item
$M$ is entry-free and exact.
\end{enumerate}
\end{lemma}

\begin{lemma}
	\label{lem:projresolutions}
Let $M$ be a countable $\NTC(W)$-module.
The following assertions are equivalent:
\begin{enumerate}[label=\textup{(\arabic*)}]
\item
$M=\FK^\con(A)$ for some $A\inin\kk(W)$.
\item
$M$ is exact.
\item
$\Tor_2^{\NTC(W)}(\NTss,M)=0$ and $\Tor_1^{\NTC(W)}(\NTss,M)=0$.
\item
$\Tor_2^{\NTC(W)}(\NTss,M)=0$ and $\Tor_1^{\NTC(W)}(\NTss,M)$ is a free Abelian group. 
\item
$M$ has a free resolution of length \textup{1} in $\Modc{\NTC(W)}$.
\item
$M$ has a projective resolution of length \textup{1} in $\Modc{\NTC(W)}$.
\item
$M$ has a projective resolution of finite length in $\Modc{\NTC(W)}$.
\end{enumerate}
\end{lemma}

\begin{remark}
In~\cite{meyernestCalgtopspacfiltrKtheory}, Meyer and Nest prove these lemmas for the special case of the totally ordered space $W=O_n$.
In Lemma~\ref{lem:projresolutions}, we have replaced condition (3) from \cite{meyernestCalgtopspacfiltrKtheory}*{Theorem~4.14} by two conditions which we are able to prove equivalent to the rest. We remark that (4) and (5) in Lemma~\ref{lem:projresolutions} are equivalent even for underlying spaces that only have Property~\ref{ass:semi} below.
\end{remark}

An investigation of the proofs in~\cite{meyernestCalgtopspacfiltrKtheory} shows that the only properties of the category $\NTC(O_n)$ Meyer and Nest actually use are the following (we formulate these properties for our general type~(A) space~$W$ as underlying space, because we will show in Theorem~\ref{thm:ungraded_iso} that they are indeed present in this generality):

\begin{property}
	\label{ass:semi}
The ideal $\NTnil$ is nilpotent and the ring $\NTC(W)$ decomposes as the
se\-mi-di\-rect product
\[
\NTC(W)=\NTnil\rtimes\NTss.
\]
\end{property}
This se\-mi-di\-rect product decomposition just means that $\NTnil$ is an ideal, $\NTss$ is a subring, and $\NTC(W)=\NTnil\oplus\NTss$ as Abelian groups. Notice that in this case we have $\Mss=\NTss\otimes_{\NTC(W)} M$.

\begin{property}
	\label{ass:free}
The Abelian group $\NT_*(W)(Y,Z)$ is free for all $Y,Z\in\LC(W)^*$.
\end{property}

\begin{property}
	\label{ass:kernel}
For every $Y\in\LC(W)^*$ there is $Z\in\LC(W)$ and a natural
transformation $\nu\in\NT_*(W)(Y,Z)$ such that
\[
(\NTnil\cdot M)(Y)=\ker\bigb{\nu\colon M(Y)\to M(Z)}
\]
holds for every exact $\NTC(W)$-module $M$.
\end{property}
Here, $M$ is regarded as an $\NT(W)$-module in the canonical way
described in Remark~\ref{rem:connected}, so that
the action of $\nu\in\NT_*(W)(Y,Z)$ is well-defined also if $Z$ is not connected.
% A useful device for verifying Property~\ref{ass:kernel} is the following
% elementary lemma taken from~\cite{meyernestCalgtopspacfiltrKtheory}*{\S3.3}:
% \begin{lemma}
% 	\label{lem:ranges}
% Let $f_1\colon A_1\to B$ and $f_2\colon A_2\to B$ be homomorphisms
% of Abelian groups. Assume that there are Abelian groups $C_1$ and $C_2$
% and homomorphisms $g_1\colon B\to C_1$ and $g_2\colon C_1\to C_2$,
% such that the sequences
% \[
% A_1\xrightarrow{f_1} B\xrightarrow{g_1} C_1,\qquad
% A_2\xrightarrow{g_1\circ f_2} C_1\xrightarrow{g_2} C_2
% \]
% are exact. Then
% \[
% \range(f_1)+\range(f_2)=\ker(g_2\circ g_1).
% \]
% \end{lemma}

In the following, we prove Lemma~\ref{lem:projmodules} and Lemma~\ref{lem:projresolutions} using only the properties listed above. Afterwards we will see that the category $\NTC(W)$ has these properties if $W$ is of type (A). We often abbreviate $\NTC(W)$ by $\NTC$ in the following.

\begin{proof}[Proof of Lemma~\textup{\ref{lem:projmodules}} using Properties~\textup{\ref{ass:semi}}, \textup{\ref{ass:free}}, \textup{\ref{ass:kernel}}]
Let $Y\in\LC(W)^*$. Yoneda's Lemma implies $\Hom_\NTC(P_Y,M)\cong M(Y)$
for all $\NTC$-modules $M$. This shows that the functor
$\Hom_\NTC(P_Y,\blank)$ is exact, which means that $P_Y$ is projective.
Since projectivity is preserved by direct sums, every free
$\NTC$-module is projective, that is, (1)$\Longrightarrow$(2).

If $M$ is a projective $\NTC$-module, then
$\Mss=\NTss\otimes_\NTC M$ is a projective $\NTss$-module. Since
$\NTss\cong\Z^{\LC(W)^*}$, this shows that $M_\sss(Y)$ is a projective
and thus free Abelian group for every $Y\in\LC(W)^*$. We have
$\Tor_1^\NTC(\NTss,M)=0$ because $M$ is projective. Altogether, we get
(2)$\Longrightarrow$(3).

Now we prove (3)$\Longrightarrow$(1). For this we need the following
proposition.
\begin{proposition}
	\label{pro:Nakayama}
In the presense of Property~\textup{\ref{ass:semi},}
let $M$ be an $\NTC$-module with $\Mss=0$. Then $M=0$.
\end{proposition}
\begin{proof}
If $\Mss=0$ then $M=\NTnil\cdot M$ and hence $M=\NTnil^j\cdot M$ for
all $j\in\N$. This implies $M=0$ since $\NTnil$ is nilpotent.
\end{proof}
The module $\Mss$ is free over $\NTss\cong\Z^{\LC(W)^*}$ because $\Mss(Y)$
is free for all $Y\in\LC(W)^*$. Hence $P\perdef\NT\otimes_{\NTss}\Mss$
is a free $\NTC$-module. Since $\Mss$ is free over $\NTss$, the
projection $M\twoheadrightarrow\Mss=M/\NTnil\cdot M$ splits by an
$\NTss$-module homomorphism. This induces an $\NTC$-module homomorphism
$f\colon P\to M$ (by tensoring over $\NTss$ with the identity on
$\NTC$ and composing with the multiplication map from
$\NTC\otimes_{\NTss} M$ to $M$). We will show that $f$ is invertible,
which implies that $M\cong P$ is free over $\NTC$.

We have an isomorphism $P_\sss\cong\NTss\otimes_\NTC P\cong\Mss$, which is
induced by $f\colon P\to M$. Using the right-exactness
of the functor $M\mapsto\Mss=\NTss\otimes_{\NTC} M$, we find
$\coker(f)_\sss=\coker(f_\sss)=0$ and hence $\coker(f)=0$ by
Proposition~\ref{pro:Nakayama}. Therefore, $f$ is surjective.
Since $P$ is projective
the extension $\ker(f)\rightarrowtail P\twoheadrightarrow M$ induces
the following long exact $\Tor$-sequence:
\begin{equation}
	\label{eq:torsequence}
0\to\Tor_1^\NTC(\NTss,M)\to\ker(f)_\sss\to P_\sss
\xrightarrow{\cong}\Mss\to 0.
\end{equation}
Notice that $\Tor_1^\NTC(\NTss,P)=0$ since~$P$ is projective.
Therefore, the assumption $\Tor_1^\NTC(\NTss,M)=0$ implies $\ker(f)_\sss=0$,
and hence $\ker(f)=0$ by Proposition~\ref{pro:Nakayama}. Therefore, $f$ is
invertible.

Up to now we have shown the equivalence of the first three conditions
using only Property~\ref{ass:semi}. The implication
(1)$\Longrightarrow$(4) follows from Property~\ref{ass:free} and from
the fact that free modules are exact. This can be seen as follows:
let $U$ be an open subset of a locally closed subset $Y$ of $W$.
We have the exact triangle~\ref{eq:exacttrianglerepobjects}
\[
\xymatrix{
\Sigma\Rep_U\ar[r] & \Rep_{Y\setminus U}\ar[r] & \Rep_Y\ar[r] &\Rep_U,
}
\]
which induces the long exact sequence
\begin{multline*}
\cdots\to\KK_*(W;\Rep_U,A)\to\KK_*(W;\Rep_Y,A)\\
\to\KK_*(W;\Rep_{Y\setminus U},A)
\to\KK_{*+1}(W;\Rep_U,A)\to\cdots
\end{multline*}
for all $A\in\kk(W)$.
In particular, when $A=\Rep_V$ for some $V\in\LC(W)^*$, by the
Representability Theorem~\ref{thm:representability} and Yoneda's
Lemma this sequence translates to the sequence
\[
\cdots\to\NT_*(V,U)\to\NT_*(V,Y)\to\NT_*(V,{Y\setminus U})
\to\NT_{*+1}(V,U)\to\cdots,
\]
proving the desired exactness. Notice that exactness is preserved
by direct sums and degree-shifting, so that indeed every free
$\NTC$-module is exact.

We complete the proof by showing (4)$\Longrightarrow$(3). By
Property~\ref{ass:kernel}, $\Mss(Y)$ is isomorphic to a subgroup of
$M(Z)$ for some $Z\in\LC(W)$ and hence a free Abelian group by
assumption because $M(Z)=\bigoplus_{C\in\pi_0(Z)} M(C)$.
The assertion now follows from Proposition~\ref{pro:vanishingtor}.
\end{proof}

\begin{proposition}
	\label{pro:vanishingtor}
Let $M$ be an exact $\NTC$-module. If Properties~\textup{\ref{ass:semi}}
and~\textup{\ref{ass:kernel}} are fulfilled, then $\Tor_1^\NTC(\NTss,M)=0$.
\end{proposition}

\begin{proof}
Choose an epimorphism $f\colon P\to M$ with a projective
$\NTC$-module~$P$. We get the long exact
sequence~\eqref{eq:torsequence}. We have seen that any projective
$\NTC$-module is free and thus exact. Hence $P$ is exact. By the
two-out-of-three property, $\ker(f)$ is exact as well. Using
Property~\ref{ass:kernel}, we identify $\ker(f)_\sss(Y)$ and
$P_\sss(Y)$ with subgroups of $\ker(f)(Z)$ and $P(Z)$ for some $Z\in\LC(W)$.
Therefore, the injectivity of the map $\ker(f)(Z)\to P(Z)$ implies
the injectivity of the map $\ker(f)_\sss(Y)\to P_\sss(Y)$. This shows
that $\ker(f)_\sss\to P_\sss$ is a monomorphism and hence that
$\Tor_1^\NTC(\NTss,M)=0$ by~\eqref{eq:torsequence}.
\end{proof}

\begin{proof}[Proof of Lemma~\textup{\ref{lem:projresolutions}} using Lemma~\textup{\ref{lem:projmodules}} and Properties \textup{\ref{ass:semi}}, \textup{\ref{ass:free}}, \textup{\ref{ass:kernel}}]
The exactness of the six-term sequence yields (1)$\Longrightarrow$(2). The implication
(5)$\Longrightarrow$(1) follows from~\cite{meyernestCalgtopspacfiltrKtheory}*{Theorem~4.11},
and the implications (3)$\Longrightarrow$(4) and (5)$\Longrightarrow$(6)$\Longrightarrow$(7) are
trivial. We will complete the proof by showing (7)$\Longrightarrow$(2),
(2)$\Longrightarrow$(5), and (4)$\Longrightarrow$(5)$\Longrightarrow$(3).

For (7)$\Longrightarrow$(2), let $0\to P_m\to\cdots\to P_0\to M$ be a
projective resolution. Define
$Z_j\perdef\ker(P_j\to P_{j-1})=\im(P_{j+1}\to P_j)$.
Then $P_j/Z_j\cong\im(P_j\to P_{j-1})=Z_{j-1}$, yielding the short
exact sequences $Z_j\rightarrowtail P_j\twoheadrightarrow Z_{j-1}$.
Starting with $Z_m=0$, the two-out-of-three property applied to the
extensions $Z_j\rightarrowtail P_j\twoheadrightarrow Z_{j-1}$
inductively implies the exactness of $Z_j$ for $j=m-1, m-2, \ldots, 0$.
Hence $M\cong P_0/Z_0$ is exact as well.

In order to prove (2)$\Longrightarrow$(5), we choose an epimorphism
$P\to M$ with a countable free $\NTC$-module~$P$, and set
$K\perdef\ker(P\to M)$. By the two-out-of-three property, $K$ is exact.
Since~$P$ is a free $\NTC$-module, its entries are free Abelian groups
by Lemma~\ref{lem:projmodules}. This property is inherited by the
submodule~$K$. Hence $K$ is free, again by Lemma~\ref{lem:projmodules},
and $0\to K\to P\twoheadrightarrow M$ is a free resolution of length~1.

Now we show (4)$\Longrightarrow$(5). Choose an epimorphism
$P\to M$ with a countable free $\NTC$-module~$P$, and set
$K\perdef\ker(P\to M)$. Since $K$ is a first syzygy of~$M$, the
assumption $\Tor_2^\NTC(\NTss,M)=0$ implies $\Tor_1^\NTC(\NTss,K)=0$.
The long exact sequence~\eqref{eq:torsequence} shows that
$K_\sss$ is an extension of free Abelian groups and thus
has free entries itself. By Lemma~\ref{lem:projmodules}, the $\NTC$-module~$K$ is free, and
$0\to K\to P\twoheadrightarrow M$ is a free resolution of length~1.

Finally, we prove (5)$\Longrightarrow$(3). We have already established
the implication (5)$\Longrightarrow$(2). Hence~$M$ is exact, and
Proposition~\ref{pro:vanishingtor} shows that $\Tor_1^\NTC(\NTss,M)=0$.
The $\NTC$-module $\Tor_2^\NTC(\NTss,M)$ also vanishes because, by (4),
the flat dimension of~$M$ is at most~1.
\end{proof}

We now introduce ungraded $\NT$-modules and use them to formulate a
central observation.

\begin{definition}
Let $\mathfrak{Mod}^\mathrm{ungr}\bigb{\NTC(W)}_\textup{c}$ denote the category of \emph{ungraded}, countable $\NTC(W)$-modules, that is, of additive functors from $\NTC(W)$ to the category of countable Abelian groups.
As in Remark~\ref{rem:connected}, there is a forgetful functor
\[
\Upsilon\colon\mathfrak{Mod}^\mathrm{ungr}\bigb{\NT(W)}_\textup{c}\to\mathfrak{Mod}^\mathrm{ungr}\bigb{\NTC(W)}_\textup{c}
\]
and a pseudo-inverse $\Upsilon^{-1}$.
An ungraded module $M\inin\mathfrak{Mod}^\mathrm{ungr}\bigb{\NT(W)}_\textup{c}$ is called \emph{exact} if the chain complexes
\[
\cdots\to M(U)\xrightarrow{i_U^Y} M(Y)\xrightarrow{r_Y^{Y\setminus U}}
M(Y\setminus U)\xrightarrow{\delta_{Y\setminus U}^U} M(U)\to\cdots
\]
are exact for all $U,Y\in\LC(W)$ with~$U$ open in~$Y$.

An ungraded module $M\inin\mathfrak{Mod}^\mathrm{ungr}\bigb{\NTC(W)}_\textup{c}$ is called \emph{exact} if $\Upsilon^{-1}(M)$ is an exact $\NT(W)$-module.
\end{definition}

As mentioned above, Meyer and Nest verified Properties~\ref{ass:semi}, \ref{ass:free} and~\ref{ass:kernel} for the special case of the totally ordered space $O_n$. The key observation made in~\cite{Rasmus} allowing to generalise this to a general space of type (A) is the following:
\begin{theorem}
	\label{thm:ungraded_iso}
Let~$W$ be a finite $T_0$-space of type \textup{(A)}. Let~$n$ be
the number of points in~$W$, and let~$O_n$ denote the totally
ordered space with~$n$ points. There is an \textup{(}ungraded\textup{)} isomorphism
$\Phi\colon\NTC(W)\to\NTC(O_n)$, and
\[
\Phi^*\colon\mathfrak{Mod}^\mathrm{ungr}\bigb{\NTC(O_n)}_\textup{c}
\to\mathfrak{Mod}^\mathrm{ungr}\bigb{\NTC(W)}_\textup{c}
\]
restricts to a bijective correspondence between exact ungraded
$\NTC(O_n)$-modules and exact ungraded $\NTC(W)$-modules.
Moreover, the isomorphism~$\Phi$ restricts to isomorphisms
from $\NTss(W)$ onto $\NTss(O_n)$ and from
$\NTnil(W)$ onto $\NTnil(O_n)$.
\end{theorem}
We postpone the proof of Theorem~\ref{thm:ungraded_iso} to \S\ref{sec:Proof_of_ungraded_iso}.
Combining Theorem~\ref{the:UCT_conditional} and Lemma~\ref{lem:projresolutions} we obtain the desired UCT:

\begin{theorem}
	\label{theorem: positive}
Let $W$ be a finite $T_0$-space of type \textup{(A)}. Then $UCT(W)$ holds.
\end{theorem}
\begin{proof}
It remains to verify Properties~\ref{ass:semi}, \ref{ass:free} and~\ref{ass:kernel} for the category $\NTC(W)$. This follows from Theorem~\ref{thm:ungraded_iso} once it is done for $\NTC(O_n)$---this has been accomplished in~\cite{meyernestCalgtopspacfiltrKtheory}.
% In order to verify the assertion concerning Property~\ref{ass:kernel}, fix an exact graded module $M\inin\mathfrak{Mod}\bigb{\NTC(W)}_\textup{c}$, regard it as an ungraded module and map it via $(\Phi^*)^{-1}$ to $\mathfrak{Mod}^\mathrm{ungr}\bigb{\NTC(O_n)}_\textup{c}$. It follows from the investigations in~\cite{meyernestCalgtopspacfiltrKtheory}, that, for every $Y\in\LC(O_n)^*$, we can find a $Z\in\LC(O_n)$ and a natural transformation $\nu\in\NT_*(O_n)(Y,Z)$ with
% \[
% \bigb{\NTnil(O_n)\cdot (\Phi^*)^{-1}(M)}(Y)=\ker\bigb{\nu\colon (\Phi^*)^{-1}(M)(Y)\to (\Phi^*)^{-1}(M)(Z)}.
% \]
% Therefore $\bigb{\NTnil(W)\cdot M}\bigb{\Phi(Y)}=\ker\Bigl(\Phi(\nu)\colon M\bigl(\Phi(Y)\bigr)\to M\bigl(\Phi(Z)\bigr)\Bigr)$. This shows that $\NTC(W)$ has Property~\ref{ass:kernel}.
\end{proof}

\section{Proof of Theorem~\ref{thm:ungraded_iso}}
	\label{sec:Proof_of_ungraded_iso}
We introduce a more explicit notation for the type (A) space $W$ which involves certain parameters, namely, an even natural number $m$ and natural numbers $n_1,\ldots,n_m$. We number the underlying set of $W$ as
\begin{equation*}
\begin{split}
\Bigl\{
1^0=1^1, &2^1, \ldots, (n_1-1)^1, n_1^1=n_2^2, (n_2-1)^2, \ldots, 2^2, 1^2=1^3,\\
         &2^3, \ldots, (n_3-1)^3, n_3^3=n_4^4, (n_4-1)^4, \ldots, 2^4, 1^4=1^5,\\
         &\qquad\qquad\qquad\qquad\vdots\\
         &2^{m-1}, \ldots, (n_{m-1}-1)^{m-1}, n_{m-1}^{m-1}=n_m^m,\\
         &\qquad\qquad\qquad\qquad\qquad\qquad\qquad(n_m-1)^m, \ldots, 2^m, 1^m=1^{m+1}
\Bigr\},
\end{split}
\end{equation*}
such that the specialisation order corresponding to the topology on $W$ is generated by the relations
\begin{equation*}
\begin{split}
&1^1 \prec 2^1 \prec \ldots \prec (n_1-1)^1 \prec n_1^1 = n_2^2 \succ (n_2-1)^2 \succ \ldots\succ 2^2 \succ 1^2 = 1^3,\\
&1^3 \prec 2^3 \prec \ldots \prec (n_3-1)^3 \prec n_3^3 = n_4^4 \succ (n_4-1)^4 \succ \ldots \succ 2^4 \succ 1^4 = 1^5,\\
&\qquad\qquad\qquad\qquad\qquad\qquad\vdots\\
&1^{m-1} \prec 2^{m-1} \prec \ldots \prec (n_{m-1}-1)^{m-1} \prec n_{m-1}^{m-1} = n_m^m,\\
&\qquad\qquad\qquad\qquad\qquad\qquad\qquad n_m^m \succ (n_m-1)^m \succ \ldots \succ 2^m \succ 1^m = 1^{m+1}.
\end{split}
\end{equation*}
Without loss of generality, we can assume that the numbers
$n_2,\ldots,n_{m-1}$ are larger than~1. This makes the
description of the space~$W$ by the parameters $m$ and $n_1,\ldots,n_m$
unique up to reversion of the order of the superscripts. The total
number of points in $W$ is $n\perdef\sum_{i=1}^m n_i-(m-1)$.

The specialisation order on the topological space $W$ corresponds to the directed graph
displayed in Figure~\ref{fig:genform}.
\begin{figure}[htbp]
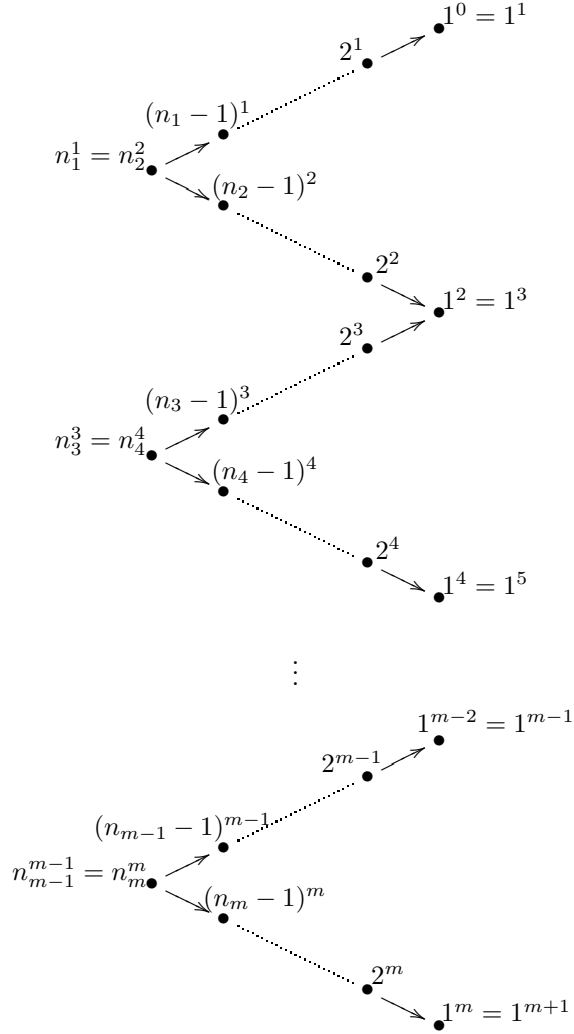

\begin{equation*}
\begin{split}
\xy
0;/r.16pc/:
		(-10,-25)*+{n_1^1=n_2^2}; (9,-17)*+{(n_1-1)^1}; (39,-4)*+{2^1}; (65,3)*+{1^0=1^1};
(56,0)*+{\bullet}="11"; (42,-7)*+{\bullet}="21"; (14,-21)*+{\bullet}="n1-"; (0,-28)*+{\bullet}="n1";
{\ar "n1"; "n1-"}; {\ar@{.} "n1-"; "21"}; {\ar "21"; "11"};
		(65,-53)*+{1^2=1^3}; (22,-31)*+{(n_2-1)^2}; (46,-46)*+{2^2};
(14,-35)*+{\bullet}="n2-"; (42,-49)*+{\bullet}="22"; (56,-56)*+{\bullet}="12";
{\ar "n1"; "n2-"}; {\ar@{.} "n2-"; "22"}; {\ar "22"; "12"};
		(9,-73)*+{(n_3-1)^3}; (39,-60)*+{2^3}; (65,-109)*+{1^4=1^5};
(42,-63)*+{\bullet}="23"; (14,-77)*+{\bullet}="n3-"; (0,-84)*+{\bullet}="n3";
{\ar "n3"; "n3-"}; {\ar@{.} "n3-"; "23"}; {\ar "23"; "12"};
		(-10,-81)*+{n_3^3=n_4^4}; (22,-87)*+{(n_4-1)^4}; (46,-102)*+{2^4};
(56,-112)*+{\bullet}="14"; (42,-105)*+{\bullet}="24"; (14,-91)*+{\bullet}="n4-";
{\ar "n3"; "n4-"}; {\ar@{.} "n4-"; "24"}; {\ar "24"; "14"};
(28,-125)*+{\vdots};
		(67,-136)*+{1^{m-2}=1^{m-1}}; (6,-157)*+{(n_{m-1}-1)^{m-1}}; (39,-144)*+{2^{m-1}}; (69,-193)*+{1^m=1^{m+1}};
(56,-140)*+{\bullet}="1m1"; (42,-147)*+{\bullet}="2m1"; (14,-161)*+{\bullet}="nm1-"; (0,-168)*+{\bullet}="nm1";
{\ar "nm1"; "nm1-"}; {\ar@{.} "nm1-"; "2m1"}; {\ar "2m1"; "1m1"};
		(-14,-166)*+{n_{m-1}^{m-1}=n_m^m}; (22,-171)*+{(n_m-1)^m}; (46,-186)*+{2^m};
(56,-196)*+{\bullet}="1m"; (42,-189)*+{\bullet}="2m"; (14,-175)*+{\bullet}="nm-";
{\ar "nm1"; "nm-"}; {\ar@{.} "nm-"; "2m"}; {\ar "2m"; "1m"};
\endxy
\end{split}
\end{equation*}
\caption{Directed graph corresponding to the type (A) space $W$}
\label{fig:genform}
\end{figure}

\subsection{Computation of the groups of natural transformations}
	\label{sec:ordcom2}
The order complex $\Ch(W)$ is a union of simplices $\Delta_k$, $k=1,\ldots,m$, of dimensions~$n_k-1$.
For $i<j$ the intersection $\Delta_i\cap \Delta_j$ is a point if $i+1=j$, and otherwise is empty.

The connected, locally closed subsets of $W$ are exactly the
``chain-like'' subsets. In order to define them, we introduce a
total order $\leq$ on $W$:
\[
a^i\leq b^j\perdeflongequi
\begin{cases}
\bigl\{ i < j\bigr\} & \textup{or} \\
\bigl\{ \textup{$i=j$ is odd and $a\preceq b$}\bigr\}  & \textup{or} \\
\bigl\{ \textup{$i=j$ is even and $a\succeq b$}\bigr\}.  &
\end{cases}
\]
This means, that $x\leq y$ exactly if $y$ is ``further down''
in Figure~\ref{fig:genform} than~$x$.

Now we define the chain-like subsets $\langle x, y\rangle$ for
$x,y\in W$ as the intervals with respect to the total
order~$\leq$:
\[
\langle x, y\rangle\perdef\left\lbrace z\in W\mid
	x\leq z\leq y\right\rbrace.
\]
Then
$\LC(W)^* = \bigl\{ \langle x,y\rangle \mid
	x,y\in W, x\leq y\bigr\}$.
Analogously, we define
\begin{align*}
\langle x,y\rlangle\perdef\left\lbrace z\in W\mid
	x\leq z< y\right\rbrace, \\
\lrangle x,y\rangle\perdef\left\lbrace z\in W\mid
	x< z\leq y\right\rbrace, \\
\intertext{and}
\lrangle x,y\rlangle\perdef\left\lbrace z\in W\mid
	x< z< y\right\rbrace.
\end{align*}
We observe that the number of elements in $\LC(W)^*$ is
$\frac{n(n+1)}{2}$.

The next step is to compute for a connected locally closed subset
$Y = \langle a^i, b^j\rangle$ the two closures $\widetilde Y$
and~$\overline Y$ and the corresponding boundaries defined in Definition~\ref{def:boundaries}.
For this computation we do a case differentiation with respect to the parity of
the numbers $i$ and~$j$. The result is given in Table~\ref{tab:closures2}.
\begin{table}[htbp]
\renewcommand{\arraystretch}{1.8}
\[
\begin{array}[t]{c|m{2.2cm}|m{2.6cm}|m{2.6cm}|m{2.1cm}}
\langle a^i, b^j\rangle &
\centering \text{$i$ and $j$ odd},\\ $a\neq 1$, $b\neq n_j$ &
\centering \text{$i$ odd, $j$ even,}\\ $a\neq 1$, $b\neq 1$ &
\centering \text{$i$ even, $j$ odd,}\\ $a\neq n_i$, $b\neq n_j$ &
\centering \text{$i$ and $j$ even,}\\ $a\neq n_i$, $b\neq 1$ \tabularnewline \hline
\overline{\langle a^i, b^j\rangle} &
\centering $\langle 1^i, b^j\rangle$ &
\centering $\langle 1^i, 1^j\rangle$  &
\centering $\langle a^i, b^j\rangle$ &
\centering $\langle a^i, 1^j\rangle$ \tabularnewline \hline
\overline\partial\langle a^i, b^j\rangle &
\centering $\langle 1^i, a^i\rlangle$ &
\centering $\langle 1^i, a^i\rlangle \cup \lrangle b^j, 1^j\rangle$ &
\centering $\emptyset$ &
\centering $\lrangle b^j, 1^j\rangle$ \tabularnewline \hline
\widetilde{\langle a^i, b^j\rangle} &
\centering $\langle a^i, n_j^j \rangle$ &
\centering $\langle a^i, b^j\rangle$  &
\centering $\langle n_i^i,n_j^j \rangle$  &
\centering $\langle n_i^i,b^j\rangle$ \tabularnewline \hline
\widetilde\partial\langle a^i, b^j\rangle &
\centering $\lrangle b^j, n_j^j \rangle$ &
\centering $\emptyset$ &
\centering $\langle n_i^i,a^i \rlangle \cup \lrangle b^j, n_j^j \rangle$ &
\centering $\langle n_i^i,a^i \rlangle$ \tabularnewline \hline
\end{array}
\]
\caption{Closures and boundaries of locally closed subsets of the space $W$}
\label{tab:closures2}
\end{table}

Now let $Y = \langle a_1^i, b_1^j\rangle$ and
$Z = \langle a_2^k, b_2^l\rangle$ be connected, locally
closed subsets of~$W$. We calculate
$S(Y,Z)=\Ch(\widetilde Y\cap\overline Z)\setminus\left( \Ch(\widetilde
	Y\cap\overline\partial Z)\cup\Ch(\widetilde\partial Y\cap\overline
	Z)\right)$
and the associated $\K$-groups (which describe the category $\NT$)
by distinguishing six cases concerning the
order of the points $a_1^i$, $b_1^j$, $a_2^k$ and $b_2^l$ with respect
to~$\leq$, and subcases concerning the parity of the numbers $i$, $j$, $k$
and~$l$.

The cases 1b, 2b, 3b are very similar to the cases 1a, 2a, and 3a,
respectively. Therefore, we only give the results for them without
repeating the arguments. For the sake of clarity, we provide small
sketches of the relative location of the sets~$Y$ and~$Z$.

\textbf{Case 1a} $a_1^i\leq a_2^k\leq b_1^j\leq b_2^l$
\begin{enumerate}[label=\textup{(\roman*)}]
\item
\emph{Let $j$ and $k$ be even, $b_1\neq 1$ and $a_2\neq n_k$.}
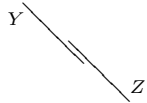
\[
\xy
(-8,8)*{}; (0,0)*{} **\dir{-};
(-2,3)*{}; (6,-5)*{} **\dir{-};
(-9,6)*+{\scriptstyle Y};
(7,-3)*+{\scriptstyle Z};
\endxy
\]
Then $S(Y,Z) = \Ch\left( \langle a_2^k, b_1^j\rangle\right)$
is contractible. Thus $\K^*\bigl(S(Y,Z)\bigr)\cong\Z[0]$.

\item
\emph{Let $j$ be even, $k$ odd, $b_1\neq 1$ and $a_2\neq 1$.}
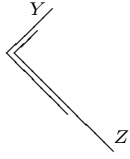
\[
\xy
(-8,8)*{}; (0,0)*{} **\dir{-};
(-8,8)*{}; (-2,14)*{} **\dir{-};
(-7,8)*{}; (6,-5)*{} **\dir{-};
(-7,8)*{}; (-4,11)*{} **\dir{-};
(-4,14)*+{\scriptstyle Y};
(7,-3)*+{\scriptstyle Z};
\endxy
\]
\begin{itemize}
\item
For $a_1^i=a_2^k$ the space $S(Y,Z)=\Ch(Y)$ is contractible and thus
$\K^*\bigl(S(Y,Z)\bigr)\cong\Z[0]$.

\item
If $a_1^i< a_2^k$ then
\[
S(Y,Z)=\Ch\left( \langle 1^k, b_1^j\rangle \right)
	\setminus\Ch\left( \langle 1^k, a_2^k\rlangle \right)
\]
	for $a_1^i\leq 1^k$, and
\[
S(Y,Z)=\Ch\left( \langle a_1^i, b_1^j\rangle \right)
	\setminus\Ch\left( \langle a_1^i, a_2^k\rlangle \right)
\]
	otherwise.
This is the difference of a contractible compact pair and we have $\K^*\bigl(S(Y,Z)\bigr)=0$.
\end{itemize}

\item
\emph{Let $j$ be odd, $k$ even, $b_1\neq n_j$ and $a_2\neq n_k$.}
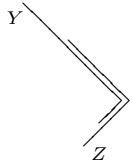
\[
\xy
(8,-8)*{}; (0,0)*{} **\dir{-};
(8,-8)*{}; (2,-14)*{} **\dir{-};
(7,-8)*{}; (-6,5)*{} **\dir{-};
(7,-8)*{}; (4,-11)*{} **\dir{-};
(-7,3)*+{\scriptstyle Y};
(4,-15)*+{\scriptstyle Z};
\endxy
\]
Analogously to (ii), we obtain $\K^*\bigl(S(Y,Z)\bigr)\cong\Z[0]$
for $b_1^j=b_2^l$, and $\K^*\bigl(S(Y,Z)\bigr)=0$ for
$b_1^j< b_2^l$.

\item
\emph{Let $j$ and $k$ be odd, $b_1\neq n_j$ and $a_2\neq 1$.}
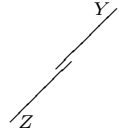
\[
\xy
(-8,-8)*{}; (0,0)*{} **\dir{-};
(-2,-1)*{}; (6,7)*{} **\dir{-};
(4,7)*+{\scriptstyle Y};
(-6,-8)*+{\scriptstyle Z};
\endxy
\]
\begin{itemize}
\item
If $a_1^i=a_2^k$ and $b_1^j=b_2^l$ then $S(Y,Z)=\Ch(Y)$ is
contractible, so $\K^*\bigl(S(Y,Z)\bigr)\cong\Z[0]$.

\item
For $a_1^i< a_2^k$ and $b_1^j=b_2^l$ we have
\[ S(Y,Z)=\Ch\left( \langle 1^k,b_2^l\rangle \right)\setminus
	\Ch\left( \langle 1^k,a_2^k\rlangle \right)
	\]
for $a_1^i\leq 1^k$, and
\[
S(Y,Z)=\Ch\left( \langle a_1^i,b_2^l\rangle \right)\setminus
	\Ch\left( \langle a_1^i,a_2^k\rlangle \right)
	\]
	otherwise.
This is the difference of a contractible compact pair. Hence $\K^*\bigl(S(Y,Z)\bigr)=0$.
\item
Analogously, $\K^*\bigl(S(Y,Z)\bigr)=0$ for $a_1^i=a_2^k$ and
$b_1^j< b_2^l$.

\item
Finally, in the case $a_1^i< a_2^k$, $b_1^j< b_2^l$, the space
is the difference of a compact pair $(K,L)$ with $K$ contractible
and $L$ the disjoint union of two contractible subspaces.
Hence that $\K^*\bigl(S(Y,Z)\bigr)\cong\Z[1]$.
\end{itemize}
\end{enumerate}

\textbf{Case 1b} $a_2^k\leq a_1^i\leq b_2^l\leq b_1^j$

Proceeding as in case 1a we obtain the following results:
\begin{enumerate}[label=\textup{(\roman*)}]
\item
\emph{If $i$ and $l$ are odd, $a_1\neq 1$, and $b_2\neq n_l$,}
then $\K^*\bigl(S(Y,Z)\bigr)=\Z[0]$.

\item
\emph{If $i$ is odd, $l$ is even, $a_1\neq 1$, $b_2\neq 1$ and
$b_2^l=b_1^j$,} then $\K^*\bigl(S(Y,Z)\bigr)=\Z[0]$.

\item
\emph{If $i$ is even, $l$ is odd, $a_1\neq n_i$, $b_2\neq n_l$ and
$a_2^k=a_1^i$,} then $\K^*\bigl(S(Y,Z)\bigr)=\Z[0]$.

\item
\emph{If $i$ and $l$ are even, $a_1\neq n_i$, $b_2\neq 1$,} then
\begin{itemize}
\item $\K^*\bigl(S(Y,Z)\bigr)=\Z[0]$, when $a_2^k=a_1^i$ and
	$b_2^l=b_1^j$;
\item $\K^*\bigl(S(Y,Z)\bigr)=\Z[1]$, when $a_2^k< a_1^i$ and
	$b_2^l< b_1^j$.
\end{itemize}

\item
\emph{In all other cases} $\K^*\bigl(S(Y,Z)\bigr)=0$.
\end{enumerate}

\textbf{Case 2a} $a_1^i\leq b_1^j< a_2^k\leq b_2^l$
\begin{enumerate}[label=\textup{(\roman*)}]
\item
\emph{Let $j$ and $k$ be even, $b_1\neq 1$ and $a_2\neq n_k$.}
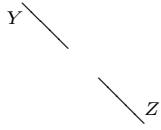
\[
\xy
(-8,8)*{}; (-2,2)*{} **\dir{-};
(2,-2)*{}; (8,-8)*{} **\dir{-};
(-9,6)*+{\scriptstyle Y};
(9,-6)*+{\scriptstyle Z};
\endxy
\]
Then $S(Y,Z)=\emptyset$ and we get $\K^*\bigl(S(Y,Z)\bigr)=0$.

\item
\emph{Let $j$ be even, $k$ odd, and $b_1\neq 1$.}
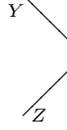
\[
\xy
(-8,8)*{}; (-2,2)*{} **\dir{-};
(-2,-2)*{}; (-8,-8)*{} **\dir{-};
(-9,6)*+{\scriptstyle Y};
(-6,-8)*+{\scriptstyle Z};
\endxy
\]
Then $S(Y,Z)$ is again empty and $\K^*\bigl(S(Y,Z)\bigr)=0$.

\item
\emph{Let $j$ be odd, $k$ even, and $a_2\neq n_k$.}
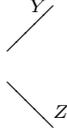
\[
\xy
(8,8)*{}; (2,2)*{} **\dir{-};
(2,-2)*{}; (8,-8)*{} **\dir{-};
(6,8)*+{\scriptstyle Y};
(9,-6)*+{\scriptstyle Z};
\endxy
\]
Once more, $S(Y,Z)=\emptyset$ and $\K^*\bigl(S(Y,Z)\bigr)=0$.

\item
\emph{Let $j$ and $k$ be odd.}
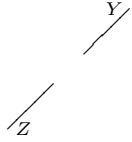
\[
\xy
(-8,-8)*{}; (-2,-2)*{} **\dir{-};
(2,2)*{}; (8,8)*{} **\dir{-};
(6,8)*+{\scriptstyle Y};
(-6,-8)*+{\scriptstyle Z};
\endxy
\]
\begin{itemize}
\item
If $j<k$ then $S(Y,Z)=\emptyset$ and thus $\K^*\bigl(S(Y,Z)\bigr)=0$.
\item
If $j=k$ and $b_1+1 < a_2$ then $S(Y,Z)$ is the difference of a
contractible compact pair and thus $\K^*\bigl(S(Y,Z)\bigr)=0$.
\item
However, if $j=k$ and $b_1+1 = a_2$ then $S(Y,Z)$ is the difference
of a compact pair $(K,L)$ as in Case 1a(iv), and we get
$\K^*\bigl(S(Y,Z)\bigr)=\Z[1]$.
\end{itemize}
\end{enumerate}

\textbf{Case 2b} $a_2^k\leq b_2^l< a_1^i\leq b_1^j$

Similarly to Case 2a we get:
\begin{enumerate}[label=\textup{(\roman*)}]
\item
\emph{If $l=i$ are even and $b_2^l+1=a_1^i$,}
then $\K^*\bigl(S(Y,Z)\bigr)=\Z[1]$.

\item
\emph{In all other cases} $\K^*\bigl(S(Y,Z)\bigr)=0$.
\end{enumerate}

\textbf{Case 3a} $a_2^k< a_1^i\leq b_1^j< b_2^l$

\begin{enumerate}[label=\textup{(\roman*)}]
\item
\emph{Let $i$ be odd, $j$ even, $a_1\neq 1$ and $b_1\neq 1$.}
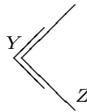
\[
\xy
(0,0)*{}; (4,4)*{} **\dir{-};
(0,0)*{}; (4,-4)*{} **\dir{-};
(1,0)*{}; (8,7)*{} **\dir{-};
(1,0)*{}; (8,-7)*{} **\dir{-};
(0,2)*+{\scriptstyle Y};
(9,-5)*+{\scriptstyle Z};
\endxy
\]
Then $S(Y,Z)$ is contractible and thus $\K^*\bigl(S(Y,Z)\bigr)=\Z[0]$.

\item
\emph{Let $i$ and $j$ be even, $a_1\neq n_i$ and $b_1\neq 1$.}
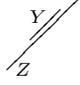
\[
\xy
(8,8)*{}; (-2,-2)*{} **\dir{-};
(5,6)*{}; (1,2)*{} **\dir{-};
(2,5)*+{\scriptstyle Y};
(0,-2)*+{\scriptstyle Z};
\endxy
\]
Then $S(Y,Z)$ is the difference
of a contractible compact pair and therefore $\K^*\bigl(S(Y,Z)\bigr)=0$.

\item
\emph{Let $i$ and $j$ be odd, $a_1\neq 1$ and $b_1\neq n_j$.}
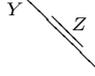
\[
\xy
(-9,9)*{}; (1,-1)*{} **\dir{-};
(-5,6)*{}; (-1,2)*{} **\dir{-};
(-10,7)*+{\scriptstyle Y};
(-1.5,5)*+{\scriptstyle Z};
\endxy
\]
Again, $S(Y,Z)$ is the difference
of a contractible compact pair and we get $\K^*\bigl(S(Y,Z)\bigr)=0$.

\item
\emph{Let $i$ be odd, $j$ even, $a_1\neq n_i$ and $b_1\neq n_j$.}
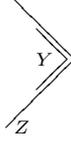
\[
\xy
(0,0)*{}; (-4,4)*{} **\dir{-};
(0,0)*{}; (-4,-4)*{} **\dir{-};
(1,0)*{}; (-7,8)*{} **\dir{-};
(1,0)*{}; (-8,-9)*{} **\dir{-};
(-3,0)*+{\scriptstyle Y};
(-6,-9)*+{\scriptstyle Z};
\endxy
\]
In this case $S(Y,Z)$ is the difference
of a compact pair $(K,L)$ as in Case 1a(iv) and thus
$\K^*\bigl(S(Y,Z)\bigr)=\Z[1]$.
\end{enumerate}

\textbf{Case 3b} $a_1^i< a_2^k\leq b_2^l< b_1^j$

For this constellation we find:
\begin{enumerate}[label=\textup{(\roman*)}]
\item
\emph{If $k$ is even, $l$ odd, $a_2\neq n_k$ and $b_2\neq n_l$,}
then $\K^*\bigl(S(Y,Z)\bigr)=\Z[0]$.

\item
\emph{If $k$ is odd, $l$ even, $a_2\neq 1$ and $b_2\neq 1$,}
then $\K^*\bigl(S(Y,Z)\bigr)=\Z[1]$.

\item
\emph{In all other cases} $\K^*\bigl(S(Y,Z)\bigr)=0$.
\end{enumerate}

\medskip
The computations concerning
$\NT_*(Y,Z)\cong\K^*\bigl(S(Y,Z)\bigr)$
carried out in this section are summarised in the
following observation; recall that we abbreviate $\NT(W)$ by $\NT$.

\begin{observation}
	\label{obs:trafos}
Let $Y,Z\in\LC(W)^*$.
\begin{enumerate}[label=\textup{(\roman*)}]
\item
$\NT_*(Y,Z)\cong\Z[0]$ if and only if
$Y\cap Z$ is non-empty, closed in~$Y$ and open in~$Z$.
\item
$\NT_*(Y,Z)\cong\Z[1]$ if and only if
\begin{itemize}
\item[\emph{either:}]
$Y\cup Z$ is connected, and $Y\cap Z$ is a proper open subset of~$Y$
and a proper closed subset of~$Z$;
\item[\emph{or:}]
$Z$ is a proper open
subset of~$Y$ and $Y\setminus Z$ has two connected components;
\item[\emph{or:}]
$Y$ is a proper closed
subset of~$Z$ and $Z\setminus Y$ has two connected components.
\end{itemize}
\item
$\NT_*(Y,Z)=0$ in all other cases.
\end{enumerate}
\end{observation}

\subsection{Generators and relations of natural transformations}
	\label{sec:genrel2}
We determine generators of the groups of natural transformations
listed in Observation~\ref{obs:trafos} and their relations in the
category $\NTC$.

In case (i), we have the grading-preserving natural transformation
$\mu_Y^Z\perdef i_{Y\cap Z}^Z\circ r_Y^{Y\cap Z}$ induced by the natural
non-zero $^*$-homomorphism
\[
A(Y)\twoheadrightarrow A(Y\cap Z)\rightarrowtail A(Z).
\]
In fact, by Corollary~\ref{cor:evengenerators}, the natural
transformation~$\mu_Y^Z$ generates the group
$\NT_0(Y,Z)\cong\Z$.

\begin{lemma}
	\label{lem:relations}
Let $Y,Z,V\in\LC(W)^*$ such that $V\cap Y$ is non-empty, closed in $V$ and open in $Y$,
and such that $Y\cap Z$ is non-empty, closed in $Y$ and open in $Z$.
With the above convention, we have $\mu_Y^Z\circ\mu_V^Y=\mu_V^Z$
if $V\cap Z$ is non-empty, closed in $V$ and open in $Z$. Otherwise,
we have $\mu_Y^Z\circ\mu_V^Y=0$.
\end{lemma}

\begin{proof}
Proposition~\ref{pro:evenrels} yields the commutative
diagram in~$\NT$
\[
\xymatrix{
& V\cap Y\ar[rd]^r\ar[r]^i &
Y\ar[r]^r & Y\cap Z\ar[rd]^i & \\
V\ar[ru]^r\ar[rr]^r & & V\cap Y\cap Z\ar[ru]^i\ar[rr]^i &
& Z.
}
\]
Since $V\cap Y$ is closed in $V$ and $Y\cap Z$ is open in~$Z$
the subset $V\cap Y\cap Z$ is clopen in $V\cap Z$. Thus we have either
$V\cap Y\cap Z=\emptyset$ or $V\cap Y\cap Z=V\cap Z$ because $V\cap Z$
is connected---it is a specific property of the space~$W$ that the
intersection of two connected subsets is again connected.

In the case $V\cap Y\cap Z=\emptyset$, we get $\mu_Y^Z\circ\mu_V^Y=0$.
However, as $V\cap Y\neq\emptyset$ and $Y\cap Z\neq\emptyset$, the
constellation $V\cap Y\cap Z=\emptyset$ can only occur if
$V\cap Z=\emptyset$. This is because $V$, $Y$ and $Z$ are intervals
with respect to the total order~$\preceq$ on~$W$. Hence we are in
the second case, and the proclaimed relation for this case holds.

For $V\cap Y\cap Z=V\cap Z$ the above diagram shows that
$\mu_Y^Z\circ\mu_V^Y=\mu_V^Z$. Hence the desired relation for the
first case holds as well.
\end{proof}

\begin{corollary}
	\label{cor:universality_even}
The category $\NT_0$ of grading-preserving natural transformations
$\FK_Y\Rightarrow\FK_Z$ for $Y,Z\in\LC(W)^*$ is the pre-additive
category generated by natural transformations~$\mu_Y^Z$ for all
$Y,Z\in\LC(W)^*$ such that $Y\cap Z$ is non-empty, closed in $Y$ and open in $Z$,
whose relations are generated by the following: 
\begin{itemize}
\item
$\mu_Y^Z\circ\mu_V^Y=\mu_V^Z$
for $Y,Z,V\in\LC(W)^*$ such that $V\cap Y$ is non-empty, closed in $V$ and open in $Y$,
and such that $Y\cap Z$ is non-empty, closed in $Y$ and open in $Z$;
\item
$\mu_Y^Z\circ\mu_V^Y=0$ otherwise.
\end{itemize}
\end{corollary}

\begin{proof}
We have verified the relations above in Lemma~\ref{lem:relations}. Computing the morphism groups for the universal pre-additive category $\mathcal U$ with generators and relations as above yields precisely the groups $\NT_0(Y,Z)$ as in Observation~\ref{obs:trafos}. This shows that the canonical functor $\mathcal U\to\NT_0$ is an isomorphism.
\end{proof}

The list of generators can of course be shortened by restricting to
indecomposable transformations. These are discussed in the next section.

Now we incorporate the odd natural transformations into our
investigation. Observation~\ref{obs:trafos}(ii) describes the three
(disjoint) cases in which an odd transformation from~$Y$ to~$Z$ occurs.

In the first case, $Y\cup Z$ is connected, and $Y\cap Z$ is a proper
open subset of~$Y$ and a proper closed subset of~$Z$. Under these
assumptions, $Z$ is open in $Y\cup Z$ and we have the odd
transformation
\[
\delta_Y^Z\colon
\xymatrix{
Y\ar[r]^<<<<<r & Y\setminus(Y\cap Z)\ar[r]|>>>>>\circ & Z.
}
\]

In the second case, $Z$ is a proper open subset of~$Y$ and
$Y\setminus Z$ has two connected components. We define $Y^<$ to be
the smaller component with respect to~$\leq$, and $Y^>$ to be the
greater component. Then $Z$ is open in $Z\cup Y^<$ and in
$Z\cup Y^>$ and we have two odd transformations
\begin{align*}
(\delta_Y^Z)^< \colon\xymatrix{Y\ar[r]^r & Y^< \ar[r]|\circ & Z,}\\
(\delta_Y^Z)^> \colon\xymatrix{Y\ar[r]^r & Y^> \ar[r]|\circ & Z.}
\end{align*}
By Proposition~\ref{pro:vanishingsum}, we have
$(\delta_Y^Z)^< = -(\delta_Y^Z)^>$. We define
$\delta_Y^Z\perdef (\delta_Y^Z)^<$.

In the third case, we similarly define~$\delta_Y^Z$ as the composite
\[
 \delta_Y^Z\perdef (\delta_Y^Z)^< \colon\xymatrix{Y\ar[r]|\circ & Z^< \ar[r]^i & Z,}
\]
where~$Z^<$ is the component of $Z\setminus Y$ which is smaller with respect to the total order~$\leq$.

\begin{lemma}
	\label{lem:odd generators}
Let $Y,Z\in\LC(W)^*$ as in Observation~\textup{\ref{obs:trafos}(ii).} The natural transformation $\delta_Y^Z$
generates the group $\NT_1(Y,Z)\cong\Z$.
\end{lemma}

\begin{proof}
We begin with the first case. Then $Y\cup Z$ is connected, and $Y\cap Z$
is a proper open subset of~$Y$ and a proper closed subset of~$Z$.
Let $C\perdef Y\setminus(Y\cap Z)$.
By Corollary~\ref{cor:oddgenerators}, it suffices to check that
${\K^1\bigb{S(C,C)\cup S(C,Z)}=0}$ and
$\K^1\bigb{S(Y,Z)\setminus S(C,Z)}=0$.
These $\K^1$-groups vanish because both $S(C,C)\cup S(C,Z)$ and
$S(Y,Z)\setminus S(C,Z)$ are a difference of a contractible compact pair.

Now we turn to the second case. Then $Z$ is a proper open subset of~$Y$
and $Y\setminus Z$ has two connected components~$Y^<$ and~$Y^>$.
As in the first case the assertion follows from
$\K^0\bigb{S(Y^<,Y^<)}\cong\Z$ and
$\K^1\bigb{S(Y^<,Z)}\cong\Z$, together with
$\K^1\bigb{S(Y^<,Y^<)\cup S(Y^<,Z)}=0$ and
$\K^1\bigb{S(Y,Z)\setminus S(Y^<,Z)}=0$.
The proof for the third case in analogous.
\end{proof}

\begin{lemma}
	\label{lem:odd compositions}
The composition of any two odd natural transformations in $\NTC$
vanishes.
\end{lemma}

\begin{proof}
An application of Proposition~\ref{pro:oddrels}(i) and~(ii) shows
that it suffices to consider a composition of two boundary transformations
coming from boundary pairs. The assertion for this
special case follows from Corollary~\ref{pro:vanishingboundarycomp}
because the union of two connected, locally closed subsets of~$W$ with
non-empty intersection is again locally closed.
\end{proof}

Our investigations above show that all morphisms in the category $\NTC$
arise as compositions of six-term sequence transformations associated to
open inclusions of locally closed subsets of~$X$, that is, $\NTC=\NTCS$.

By Lemma~\ref{lem:odd compositions}, the category $\NTC$ is a split extension
of its even subcategory $\NT^\con_0$ by the bimodule $\NT^\con_1$ of odd transformations.
Propositions~\ref{pro:oddrels}, \ref{pro:oddrel2} and~\ref{pro:vanishingsum}
show that the bimodule structure is as follows: a product
$\mu_Y^Z\circ\delta_W^Y$ or $\delta_Y^Z\circ\mu_W^Y$ is equal to~$\delta_W^Z$
or~$-\delta_W^Z$ whenever all three natural transformations are defined, and
zero otherwise. The occurrence of the minus sign is due to the non-canonical
definition of~$\delta_W^Z$.

An argument as in the proof of
Corollary~\ref{cor:universality_even} now shows that the relations in $\NTC$
are generated by the canonical ones listed in Definition~\ref{def:canonical relations}.
The above description of $\NTC$ as a split extension  was given
in~\cite{meyernestCalgtopspacfiltrKtheory} for the category of natural
transformations corresponding to the totally ordered space.

% The description of the category $\NTC$ we have given above shows
% that $\NTC=\NTCS$, and that the relations in $\NTCS$ are
% generated by the canonical ones from \S\ref{sec:relationsNT}. Hence the
% indecomposability criteria Proposition~\ref{pro:evenindecomp} and
% Corollary~\ref{cor:oddindecomp} hold in $\NT^\con_\textup{even}$ and $\NTC$,
% respectively.
% Since any product of two odd transformations in $\NTC$
% vanishes, the criteria in Proposition~\ref{pro:evenindecomp}
% hold in $\NTC$ as well.
We now apply our indecomposability criteria established in \S\ref{sec:relationsNT}.
Using that compositions of odd transformations vanish, we find that the characterisations
of even and odd indecomposable transformations in Proposition~\ref{pro:evenindecomp}
and Corollary~\ref{cor:oddindecomp} are valid in the category~$\NTC$.
We obtain a complete list of indecomposable natural transformations in the
category~$\NTC$ consisting of essentially only five different types.

\begin{observation}
	\label{list:indecomposabletransformations}
The category $\NTC$ is generated by the following indecomposable natural transformations:
\begin{enumerate}[label=\textup{(\arabic*)}]
%extension
%1
\item
an extension
$\langle (a+1)^i,b^j\rangle \to\langle a^i,b^j\rangle$
whenever $i$ is odd, $a\neq 1$, $a\neq n_i$, and $a^i\neq b^j$;
%2
\item
an extension
$\langle 2^{i+1},b^j\rangle \to\langle n_i^i,b^j\rangle$
whenever $i$ is even and $b^j> 1^{i+1}$;
%3
\item
an extension
$\langle a^i,(b+1)^j\rangle \to\langle a^i,b^j\rangle$
whenever $j$ is even, $b\neq 1$, $b\neq n_j$, and $a^i\neq b^j$;
%4
\item
an extension
$\langle a^i,2^{j-1}\rangle \to\langle a^i,n_j^j\rangle$
whenever $j$ is odd and $a^i< 1^{j-1}$;
%restrictions
%5
\item
a restriction
$\langle (a+1)^i,b^j\rangle \to\langle a^i,b^j\rangle$
whenever $i$ is even, $a\neq n_i$ and $a\neq n_i-1$;
%6
\item
a restriction
$\langle 1^{i-1},b^j\rangle \to\langle (n_i-1)^i,b^j\rangle$
whenever $i$ is even;
%7
\item
a restriction
$\langle a^i,(b+1)^j\rangle\to\langle a^i,b^j\rangle$
whenever $j$ is odd, $b\neq n_j$ and $b\neq n_j-1$;
%8
\item
a restriction
$\langle a^i,1^{j+1}\rangle\to\langle a^i,(n_j-1)^j\rangle$
whenever $j$ is odd;
%boundaries
\item
a boundary transformation
$\langle 1^i,(a-1)^i\rangle\to\langle a^i,n_i^i\rangle$
whenever $i$ is odd and $a\neq 1$;
\item
a boundary transformation
$\langle (b-1)^j,1^j\rangle\to\langle n_j^j,b^j\rangle$
whenever $j$ is even and $b\neq 1$.
\end{enumerate}
\end{observation}

We make two further observations which are direct consequences
of Observation~\ref{list:indecomposabletransformations} and which
will be used in the next section.

\begin{observation}
	\label{obs:singularsubsets}
There are precisely $n+1$ sets $C\in\LC(W)^*$ with the
property that there is only one indecomposable transformation to~$C$,
namely:
\begin{itemize}
\item 
the singletons $\{1^1\}$, $\{1^m\}$ and $\{a^i\}$ with
$i\in\{1,\ldots,m\}$ and $a\not\in\{1,n_i\}$;
\item
the maximal totally ordered subsets $\{1^i,2^i,\ldots,n_i^i\}$ for
$i\in\{1,\ldots,m\}$.
\end{itemize}
Moreover, these are precisely the sets $C\in\LC(W)^*$ such
that there is only one indecomposable transformation \emph{out of}~$C$.
We call these sets \emph{singular subsets} of~$W$.

For all other subsets $D\in\LC(W)^*$ there are precisely
two indecomposable transformations to~$D$ and precisely two
indecomposable transformations out of~$D$. Altogether, the category
$\NTC$ is thus generated by $n^2-1$ indecomposable transformations.
\end{observation}

As we have seen, every indecomposable natural transformations~$\mu$
in~$\NTC$ is a six-term sequence transformation. We therefore have
an associated \emph{subsequent} six-term sequence transformation~$\eta$.
That is, if~$\mu$ is an extension transformation~$i_U^Y$,
then $\eta=r_Y^{Y\setminus U}$;
similarly for restriction and boundary transformations.
\begin{observation}
	\label{obs:subsequent}
Let $Y$ be a singular subset and $\mu\colon Y\to Z$ an indecomposable
transformations. Let $\eta\colon Z\to V$ be the
subsequent six-term sequence transformation. Then~$\eta$ is
indecomposable and~$V$ is singular.
\end{observation}

\subsection[Construction of an isomorphism of categories]
	{Construction of an isomorphism of categories}
	\label{sec:ringtheory2}
In this section, we show that the category $\NTC$ essentially
depends only on the number~$n$, the total number of points
in~$W$. We illustrate our approach by an example first.

\begin{example}
	\label{exa:four_point_spaces}
As an example, we compare the two categories $\NTC(O_4)$
and $\NTC(W_4)$ for the topological spaces $O_4$
and $W_4$ which correspond to the partial orders
$1\prec 2\prec 3\prec 4$ and $1\prec 2\prec 3\succ 4$ on the set $\{1,2,3,4\}$, respectively.
The indecomposable transformations in these categories are
displayed in Figure~\ref{fig:indecomposables_ordered_4} and~\ref{fig:indecomposables_accordion_4},
where we use the abbreviation $234\perdef\{2,3,4\}$, and so on.
\begin{figure}[htbp]
\[
\xymatrix@-0.8pc{
1\ar[rd]|\circ^\delta & & 1234\ar[rd]^r & & 4\ar[rd]^i & \\
& 234\ar[ru]^i\ar[rd]^r & & 123\ar[ru]|\circ^\delta\ar[rd]^r & & 34 \\
34\ar[ru]^i\ar[rd]^r & & 23\ar[ru]^i\ar[rd]^r & & 12\ar[ru]|\circ^\delta\ar[rd]^r & \\
& 3\ar[ru]^i & & 2\ar[ru]^i & & 1
}
\]
\caption{Diagram of indecomposable natural transformations in $\NTC(O_4)$}
\label{fig:indecomposables_ordered_4}
\end{figure}
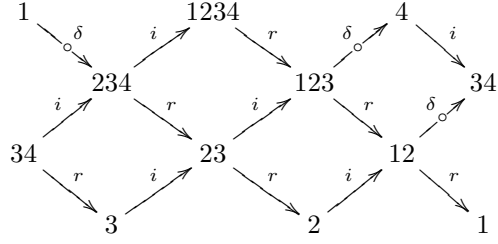
\begin{figure}[htbp]
\[
\xymatrix@-0.8pc{
1\ar[rd]|\circ^\delta & & 123\ar[rd]^i & & 4\ar[rd]|\circ^\delta & \\
& 23\ar[ru]^i\ar[rd]^i & & 1234\ar[ru]^r\ar[rd]^r\ar@{}[rr]|{-} & & 3\\
3\ar[ru]^i\ar[rd]^i & & 234\ar[ru]^i\ar[rd]^r & & 12\ar[ru]|\circ^\delta\ar[rd]^r & \\
& 34\ar[ru]^i & & 2\ar[ru]^i & & 1
}
\]
\caption{Diagram of indecomposable natural transformations in $\NTC(W_4)$}
\label{fig:indecomposables_accordion_4}
\end{figure}
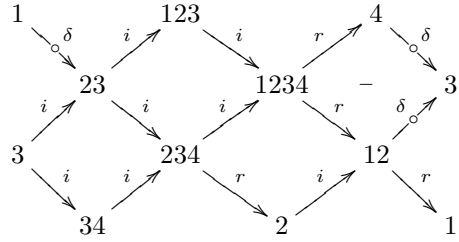
In Figure~\ref{fig:indecomposables_ordered_4} all squares are
commutative. This is also true for Figure~\ref{fig:indecomposables_accordion_4}, except for the single square
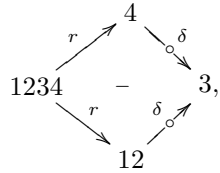
\[
\xymatrix@-0.8pc{
& 4\ar[rd]|\circ^\delta & \\
1234\ar[ru]^r\ar[rd]^r\ar@{}[rr]|{-} & & 3,\\
& 12\ar[ru]|\circ^\delta & 
}
\]
which anti-commutes. Moreover, the compositions of indecomposable transformations
of the form $S\to\sharp\to S'$ for singular subsets $S$, $S'$ all vanish as part
of exact six-term sequences. For a proof of these relations, see \S\ref{sec:relationsNT}.
Arguing as in the proof of Corollary~\ref{cor:universality_even}, we see that the relations
above generate all relations in the category $\NTC(W_4)$.

By replacing the generator $\delta_4^3$ with its
additive inverse, we can make all squares in Figure~\ref{fig:indecomposables_accordion_4} commute. Now it can be verified
by a direct check that the obvious bijection between the chosen sets of
generators of the two categories extends to an isomorphism of
categories. This isomorphism is not grading-preserving. However, it
has the following property: a subset $\{U,Y,C\}\subset\LC(O_4)^*$
consisting of a boundary pair $(U,C)$ and its union $Y=U\cup C$ is
mapped to a subset of $\LC(W_4)^*$ of the same kind,
though the roles of each particular set may be interchanged.
This shows that the isomorphism respects exactness of modules.
\end{example}

Now we generalise the observations in Example~\ref{exa:four_point_spaces} to the general situation. We begin with
describing certain chains of indecomposable natural transformations
connecting two singular subsets of~$W$. Every chain consists
of $n-1$ transformations.

Starting with the point $1^1$, we have the chain
\begin{align*}
\{1^1\}
\xrightarrow{\delta}\langle 2^1,n_1^1\rangle
\xrightarrow{i}\langle 2^1,(n_2-1)^2\rangle
\xrightarrow{i}&\cdots
\xrightarrow{i}\langle 2^1,2^2\rangle\\
\xrightarrow{i}\langle 2^1,n_3^3\rangle
\xrightarrow{i}\langle 2^1,(n_4-1)^4\rangle
\xrightarrow{i}&\cdots
\xrightarrow{i}\langle 2^1,2^4\rangle\\
\vdots\\
\xrightarrow{i}\langle 2^1,n_{m-1}^{m-1}\rangle
\xrightarrow{i}\langle 2^1,(n_{m}-1)^{m}\rangle
\xrightarrow{i}&\cdots
\xrightarrow{i}\langle 2^1,1^m\rangle\\
\xrightarrow{r}\langle 2^1,(n_{m-1}-1)^{m-1}\rangle
\xrightarrow{r}&\cdots
\xrightarrow{r}\langle 2^1,1^{m-1}\rangle\\
\vdots\\
\xrightarrow{r}\langle 2^1,(n_{1}-1)^{1}\rangle
\xrightarrow{r}&\cdots
\xrightarrow{r}\{2^1\}
\end{align*}
from $\{1^1\}$ to $\{2^1\}$, which we denote by
$\{1^1\}\Longrightarrow\{2^1\}$.

In the following, we make the underlying rule for this procedure precise.
Fix an indecomposable transformation $\nu\colon Y\to Z$. We distinguish
two cases:

If $Z$ is a singular subset, then there is precisely one indecomposable transformation
$S(\nu)$ out of~$Z$.
If, on the other hand, $Z$ is a non-singular subset, then there are precisely two indecomposable transformations
out of~$Z$ (cf.\ Observation~\ref{obs:singularsubsets}), and we want to choose the
``right'' one.

The following lemma describes the indecomposable transformations out of a non-singular
subset~$Z$ with respect to an indecomposable transformation into~$Z$.
It provides us with a way to define the successor of an indecomposable transformation
into a non-singular subset.

\begin{lemma}
	\label{obs:successor}
Let $Z$ is be non-singular subset. Let $\nu\colon Y\to Z$ be an indecomposable
transformation.
\begin{enumerate}[label=\textup{(\roman*)}]
\item
If $Y$ is singular, then there is precisely one of the two indecomposable
transformations out of $Z$---denoted by $S(\nu)$---which is not the subsequent
transformation in the six-term sequence of~$\nu$.
\item
If $Y$ is non-singular, then there is precisely one of the two indecomposable
transformations out of $Z$---denoted by $S(\nu)$---such that the composition
$S(\nu)\circ\nu$ cannot be factorised into a product of two other 
indecomposable transformations.
\end{enumerate}
\end{lemma}

The underlying rule for our chains of indecomposable transformations is now
simply as follows. The well-definition of this rule is a consequence of Lemma~\ref{obs:successor}.

\begin{definition}
	\label{def:successor}
\emph{The successor of an indecomposable transformation $\nu$ is the
indecomposable transformation $S(\nu)$.}
\end{definition}

\begin{proof}[Proof of Lemma~\textup{\ref{obs:successor}}]
The first assertion is a consequence of Observation~\ref{obs:subsequent}. The second assertion can be checked by a case differentiation using the list in Observation~\ref{list:indecomposabletransformations}. By symmetry considerations, it suffices to check the cases~(1), (2), (5), (6), and~(9) from that list. As an example, we discuss case~(1) here. In the thus remaining four cases, the assertion can be verified in the straight-forward but lengthy manner outlined below.

Consider the indecomposable extension
$i_{\langle (a+1)^i,b^j\rangle}^{\langle a^i,b^j\rangle}$
with $i$ is odd, $a\neq 1$, $a\neq n_i$, and $a^i\neq b^j$.

% The set $\langle (a+1)^i,b^j\rangle$ is singular if and only if $(a+1)^i=b^j$.
% In this case, the indecomposable transformations out of $\langle a^i,b^j\rangle$
% are $r_{\langle a^i,b^j\rangle}^{\langle a^i,a^i\rangle}$ and
% $\begin{cases}
% i_{\langle a^i,b^j\rangle}^{\langle (a-1)^i,b^j\rangle} & \textup{if $a>2$,} \\
% i_{\langle a^i,b^j\rangle}^{\langle n_{i-1}^{i-1},b^j\rangle} & \textup{if $a=2$.} 
% \end{cases}$
% 
% Indeed,
% $r_{\langle a^i,b^j\rangle}^{\langle a^i,a^i\rangle}\circ i_{\langle (a+1)^i,b^j\rangle}^{\langle a^i,b^j\rangle}=0$,
% whereas
% $i_{\langle a^i,b^j\rangle}^{\langle (a-1)^i,b^j\rangle}\circ i_{\langle (a+1)^i,b^j\rangle}^{\langle a^i,b^j\rangle}$ and
% $i_{\langle a^i,b^j\rangle}^{\langle n_{i-1}^{i-1},b^j\rangle}\circ i_{\langle (a+1)^i,b^j\rangle}^{\langle a^i,b^j\rangle}$ do not vanish.
% 
% If, on the other hand, $\langle (a+1)^i,b^j\rangle$ is non-singular, that is, if
% $(a+1)^i\prec b^j$,

The set $\langle (a+1)^i,b^j\rangle$ is non-singular if and only if $(a+1)^i\prec b^j$. In this case, the indecomposable transformations out of $\langle a^i,b^j\rangle$ are
$\begin{cases}
i_{\langle a^i,b^j\rangle}^{\langle (a-1)^i,b^j\rangle} & \textup{if $a>2$,} \\
i_{\langle a^i,b^j\rangle}^{\langle n_{i-1}^{i-1},b^j\rangle} & \textup{if $a=2$,} 
\end{cases}$
and
\[\begin{cases}
r_{\langle a^i,b^j\rangle}^{\langle a^i,(b-1)^j\rangle} & \textup{if $j$ odd, $b\neq 1$,} \\
r_{\langle a^i,b^j\rangle}^{\langle a^i,(n_{j-1}-1)^{j-1}\rangle} & \textup{if $j$ even, $b=1$,} \\
i_{\langle a^i,b^j\rangle}^{\langle a^i,(b-1)^j\rangle} & \textup{if $j$ even, $b\neq 1,2$,} \\
i_{\langle a^i,b^j\rangle}^{\langle a^i,n_{j+1}^{j+1}\rangle} & \textup{if $j$ even, $b=2$.} 
\end{cases}\]
While
$i_{\langle a^i,b^j\rangle}^{\langle (a-1)^i,b^j\rangle}\circ i_{\langle (a+1)^i,b^j\rangle}^{\langle a^i,b^j\rangle}=i_{\langle (a+1)^i,b^j\rangle}^{\langle (a-1)^i,b^j\rangle}$
and
$i_{\langle a^i,b^j\rangle}^{\langle n_{i-1}^{i-1},b^j\rangle}\circ i_{\langle (a+1)^i,b^j\rangle}^{\langle a^i,b^j\rangle}=i_{\langle (a+1)^i,b^j\rangle}^{\langle n_{i-1}^{i-1},b^j\rangle}$
do not factorise in a non-trivial way different from the given one
(which may also be read from the list in Observation~\ref{list:indecomposabletransformations} since
we know how these generators multiply), we have
\begin{align*}
r_{\langle a^i,b^j\rangle}^{\langle a^i,(b-1)^j\rangle}\circ i_{\langle (a+1)^i,b^j\rangle}^{\langle a^i,b^j\rangle}
&=i_{\langle (a+1)^i,(b-1)^j\rangle}^{\langle a^i,(b-1)^j\rangle}\circ r_{\langle (a+1)^i,b^j\rangle}^{\langle (a+1)^i,(b-1)^j\rangle},\\
r_{\langle a^i,b^j\rangle}^{\langle a^i,(n_{j-1}-1)^{j-1}\rangle}\circ i_{\langle (a+1)^i,b^j\rangle}^{\langle a^i,b^j\rangle}
&=i_{\langle (a+1)^i,(n_{j-1}-1)^{j-1}\rangle}^{\langle a^i,(n_{j-1}-1)^{j-1}\rangle}\circ r_{\langle (a+1)^i,b^j\rangle}^{\langle (a+1)^i,(n_{j-1}-1)^{j-1}\rangle},\\
i_{\langle a^i,b^j\rangle}^{\langle a^i,(b-1)^j\rangle}\circ i_{\langle (a+1)^i,b^j\rangle}^{\langle a^i,b^j\rangle}
&=i_{\langle (a+1)^i,(b-1)^j\rangle}^{\langle a^i,(b-1)^j\rangle}\circ i_{\langle (a+1)^i,b^j\rangle}^{\langle (a+1)^i,(b-1)^j\rangle},\\
i_{\langle a^i,b^j\rangle}^{\langle a^i,n_{j+1}^{j+1}\rangle}\circ i_{\langle (a+1)^i,b^j\rangle}^{\langle a^i,b^j\rangle}
&=i_{\langle (a+1)^i,n_{j+1}^{j+1}\rangle}^{\langle a^i,n_{j+1}^{j+1}\rangle}\circ i_{\langle (a+1)^i,b^j\rangle}^{\langle (a+1)^i,n_{j+1}^{j+1}\rangle},
\end{align*}
providing factorisations into products of two other 
indecomposable transformations, respectively.
\end{proof}

In addition to the previously described chain of indecomposable transformations
from $\{1^1\}$ to $\{2^1\}$, we obtain the following chains of indecomposable
transformations between singular subsets when applying the
rule from Definition~\ref{def:successor}:

If $n_1>2$, we have a chain $\{2^1\}\Longrightarrow\{3^1\}$,
namely
\begin{align*}
\{2^1\}
\xrightarrow{i}\langle 2^1,1^1\rangle
\xrightarrow{\delta}\langle 3^1,n_1^1\rangle
\xrightarrow{i}\langle 3^1,(n_2-1)^2\rangle
\xrightarrow{i}&\cdots
\xrightarrow{i}\langle 3^1,2^2\rangle\\
\xrightarrow{i}\langle 3^1,n_3^3\rangle
\xrightarrow{i}\langle 3^1,(n_4-1)^4\rangle
\xrightarrow{i}&\cdots
\xrightarrow{i}\langle 3^1,2^4\rangle\\
\vdots\\
\xrightarrow{i}\langle 3^1,n_{m-1}^{m-1}\rangle
\xrightarrow{i}\langle 3^1,(n_{m}-1)^{m}\rangle
\xrightarrow{i}&\cdots
\xrightarrow{i}\langle 3^1,1^m\rangle\\
\xrightarrow{r}\langle 3^1,(n_{m-1}-1)^{m-1}\rangle
\xrightarrow{r}&\cdots
\xrightarrow{r}\langle 3^1,1^{m-2}\rangle\\
\vdots\\
\xrightarrow{r}\langle 3^1,(n_{1}-1)^{1}\rangle
\xrightarrow{r}&\cdots
\xrightarrow{r}\{3^1\}.
\end{align*}
In the same way, we obtain chains of indecomposable transformations
$\{3^1\}\Longrightarrow\{4^1\}\Longrightarrow\cdots\Longrightarrow\{(n_1-1)^1\}$.
This is followed by the chains
\begin{align*}
\{(n_1-1)^1\}
\xrightarrow{i}\langle (n_1-2)^1,(n_1-1)^1\rangle
\xrightarrow{i}&\cdots
\xrightarrow{i}\langle 1^1,(n_1-1)^1\rangle\\
\xrightarrow{\delta}\{n_1^1\}
\xrightarrow{i}\langle n_1^1,(n_2-1)^2\rangle
\xrightarrow{i}&\cdots
\xrightarrow{i}\langle n_1^1,2^2\rangle\\
\xrightarrow{i}\langle n_1^1,n_3^3\rangle
\xrightarrow{i}\langle n_1^1,(n_4-1)^4\rangle
\xrightarrow{i}&\cdots
\xrightarrow{i}\langle n_1^1,2^4\rangle\\
\vdots\\
\xrightarrow{i}\langle n_1^1,n_{m-1}^{m-1}\rangle
\xrightarrow{i}\langle n_1^1,(n_{m}-1)^{m}\rangle
\xrightarrow{i}&\cdots
\xrightarrow{i}\langle n_1^1,1^m\rangle\\
\xrightarrow{r}\langle n_1^1,(n_{m-1}-1)^{m-1}\rangle
\xrightarrow{r}&\cdots
\xrightarrow{r}\langle n_1^1,1^{m-2}\rangle\\
\vdots\\
\xrightarrow{r}\langle n_1^1,(n_{3}-1)^{3}\rangle
\xrightarrow{r}&\cdots
\xrightarrow{r}\langle n_1^1,1^2\rangle
\end{align*}
from $\{(n_1-1)^1\}$ to $\langle n_1^1,1^2\rangle$, and
\begin{align*}
\langle n_1^1,1^2\rangle
\xrightarrow{i}\langle (n_1-1)^1,1^2\rangle
\xrightarrow{i}&\cdots
\xrightarrow{i}\langle 1^1,1^2\rangle\\
\xrightarrow{r}\langle (n_2-1)^2,1^2\rangle
\xrightarrow{r}&\cdots
\xrightarrow{r}\{1^2\}\\
\xrightarrow{\delta}\langle 2^3,n_3^3\rangle
\xrightarrow{i}\langle 2^3,(n_4-1)^4\rangle
\xrightarrow{i}&\cdots\;\cdots
\xrightarrow{i}\langle 2^3,1^m\rangle\\
\xrightarrow{r}\langle 2^3,(n_{m-1}-1)^{m-1}\rangle
\xrightarrow{r}&\cdots\;\cdots
\xrightarrow{r}\{2^3\}
\end{align*}
from $\langle n_1^1,1^2\rangle$ to $\{2^3\}$.
Continuing this procedure, we obtain the following long chain of
indecomposable natural transformations:
\begin{align}
  \label{eq:longchain}
\begin{split}
\{1^1\}
\Longrightarrow\{2^1\}
\Longrightarrow &\cdots
\Longrightarrow\{(n_1-1)^1\}
\Longrightarrow\langle n_1^1,1^2\rangle\\
\Longrightarrow\{2^3\}
\Longrightarrow &\cdots
\Longrightarrow\{(n_3-1)^3\}
\Longrightarrow\langle n_3^3,1^4\rangle\\
\vdots\\
\Longrightarrow\{2^{m-1}\}
\Longrightarrow &\cdots
\Longrightarrow\{(n_{m-1}-1)^{m-1}\}
\Longrightarrow\langle n_{m-1}^{m-1},1^m\rangle\\
\Longrightarrow\{1^m\}
\Longrightarrow\{2^m\}
\Longrightarrow &\cdots
\Longrightarrow\{(n_m-1)^m\}
\Longrightarrow\langle 1^{m-1},n_m^m\rangle\\
\Longrightarrow\{2^{m-2}\}
\Longrightarrow &\cdots
\Longrightarrow\{(n_{m-2}-1)^{m-2}\}
\Longrightarrow\langle 1^{m-3},n_{m-2}^{m-2}\rangle\\
\vdots\\
\Longrightarrow\{2^2\}
\Longrightarrow &\cdots
\Longrightarrow\{(n_2-1)^2\}
\Longrightarrow\langle 1^1,n_1^1\rangle
\Longrightarrow\{1^1\}.
\end{split}
\end{align}
This long chain is the composition of $n+1$ of the previously described
chains, each of them connecting two singular subsets of~$W$.
We obtain an enumeration (without repetitions) of the singular
subsets of~$W$. We denote the so enumerated singular subsets
of~$W$ by $S_i$ with $i\in\{1,\ldots,n+1\}$.

In fact, each of the $n^2-1$ indecomposable transformations in $\NTC$
occurs precisely once in the above long cyclic chain. This is simply
because this long chain consists of $(n+1)(n-1)=n^2-1$
indecomposable transformations and non of them occurs more than once.
To see this, observe that we return to the singular subset $\{1^1\}$
only after $n^2-1$ steps, and that the successing transformations is
well-defined. Hence, we obtain an enumeration of the indecomposable
natural transformations in $\NTC$ as well.

Each non-singular subset of~$W$ is listed
precisely twice in~\eqref{eq:longchain}.
Figures~\ref{fig:indecomposables_odd} and~\ref{fig:indecomposables_even}
indicate a way in which the long chain~\eqref{eq:longchain} can be
entangled in order to list each element of $\LC(W)^*$ only
once. The singular subsets of $W$ and the chains of indecomposable
transformations between them are indicated explicitly. At each
intersection point of two chains, a non-singular subset of~$W$
is situated. The diagram is periodic in the horizontal
direction; the dashed arrows indicate that the vertical order of the
repeating objects is reversed after one period.

\begin{figure}[htbp]
\[
\xymatrix@-0.8pc{
\ar@{-->}[dddddd] &
S_1\ar@{=>}[rdrdrdrdrdrd] & &
S_{n+1}\ar@{=>}[rdrdrdrdrdrd] & &
\cdots & &
S_\frac{n+5}{2}\ar@{=>}[rdrd] & &
S_\frac{n+3}{2} & \\
\\
& \ar@{=>}[ruru] & & & & & & & &  & \\
\\
& \ar@{=>}[rdrd] & & & & & & & &  & \\
\\
& S_\frac{n+3}{2}\ar@{=>}[rurururururu] & &
S_\frac{n+1}{2}\ar@{=>}[rurururururu] & &
\cdots & &
S_2\ar@{=>}[ruru] & &
S_1 & \ar@{-->}[uuuuuu]
}
\]
\caption{Diagram of indecomposable natural transformations in $\NTC$
for the space $W$ with odd number of points}
\label{fig:indecomposables_odd}
\end{figure}
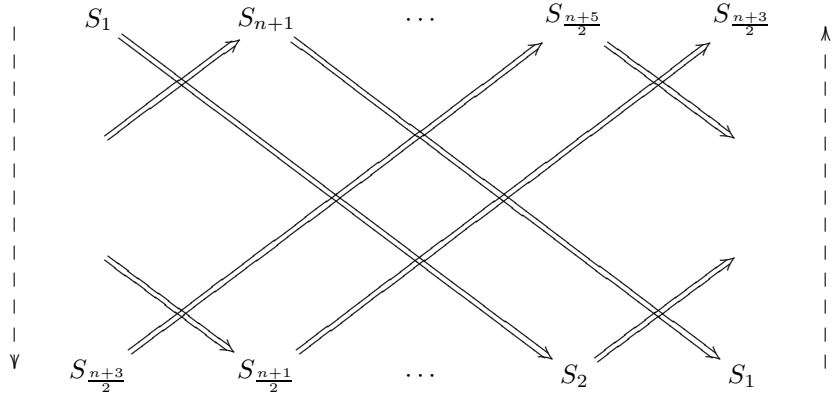

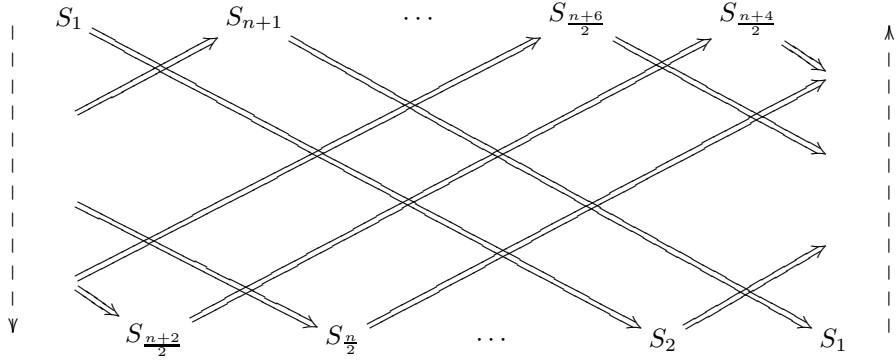
\begin{figure}[htbp]
\[
\xymatrix@-1.2pc{
\ar@{-->}[ddddddd] &
S_1\ar@{=>}[rdrdrdrdrdrdrd] & &
S_{n+1}\ar@{=>}[rdrdrdrdrdrdrd] & &
\cdots & &
S_\frac{n+6}{2}\ar@{=>}[rdrdrd] & &
S_\frac{n+4}{2}\ar@{=>}[rd] & & \\
& & & & & & & & & & & \\
& \ar@{=>}[ruru] & & & & & & & & & & \\
& & & & & & & & & & & \\
& \ar@{=>}[rdrdrd] & & & & & & & & & & \\
& & & & & & & & & & & \\
& \ar@{=>}[rd]\ar@{=>}[urururururur] & & & & & & & & & & \\
& & S_\frac{n+2}{2}\ar@{=>}[ururururururur] & &
S_\frac{n}{2}\ar@{=>}[urururururur] & &
\cdots & &
S_2\ar@{=>}[ruru] & &
S_1 & \ar@{-->}[uuuuuuu]
}
\]
\caption{Diagram of indecomposable natural transformations in $\NTC$
for the space $W$ with even number of points}
\label{fig:indecomposables_even}
\end{figure}

As remarked in the previous subsection, the relations in the category~$\NTC(W)$ are generated by the canonical ones from Definition~\ref{def:canonical relations}. In the following, we shall spell these out more explicitly for the case at hand.

% Since all squares in these diagrams contain only well-understood natural transformations coming from six-term exact sequences, t
The canonical relations translate to the vanishing of all compositions of consecutive six-term sequence transformations together with commutativity relations for all squares in the above diagram; more precisely, all squares either commute or anti-commute. The only squares that anti-commute are those of the form
\begin{equation}
	\label{eq:anti-commutativity relations}
\begin{split}
\xymatrix{
& U_1\ar[rd]^i & & \\
C\ar[ru]^\delta|\circ\ar[rd]^\delta|\circ & & Z,\ar@{}[ll]|{-} & \text{or} \\
& U_2\ar[ru]^i & &
}
\qquad\qquad
\xymatrix{
& C_1\ar[rd]^\delta|\circ & \\
Y\ar[ru]^r\ar[rd]^r & & U.\ar@{}[ll]|{-} \\
& C_2\ar[ru]^\delta|\circ &
}
\end{split}
\end{equation}
To see that the sum-relations in Propositions~\ref{pro:evenrels} and~\ref{pro:vanishingsum} do not contribute anything beyond the anti-commutativity relations~\eqref{eq:anti-commutativity relations}, notice that intersections of connected subsets of~$W$ are again connected, and that complements of connected subsets of~$W$ have at most \emph{two} connected components.

We can make all the anti-commuting squares in~\eqref{eq:anti-commutativity relations} commute by replacing the boundary transformations $\delta_C^U$ for all boundary pairs $(U,C)$ with $C< U$ by their additive inverses. Recall that $\leq$ denotes the total order on $W$ defined in \S\ref{sec:ordcom2}. This change in the choice of generators does not affect the
commutativity of the remaining squares because each of them contains either no boundary transformations or two boundary transformations with the same orientation concerning the order~$\leq$.

By Observation~\ref{obs:subsequent}, the above relations imply that the compositions
\begin{equation}
	\label{eq:singular_composition}
\begin{split}
\xymatrix{
S_i\ar[rd] & & S_{i-1}\\
& \sharp\ar[ru] &
}
\end{split}
\end{equation}
vanish for all $i\in\{1,\ldots,n+1\}$; here we set $S_0\perdef S_{n+1}$,
and~$\sharp$ denotes the unique object sitting between $S_i$ and $S_{i-1}$
in the above diagram of indecomposable transformations.

\begin{lemma}
	\label{lem:generationbysingulars}
The relations in $\NTC(W)$ are generated by the commutativity relations for
all squares and the vanishing of the compositions~\eqref{eq:singular_composition}.
\end{lemma}

\begin{proof}
By the list of generating relations given above, it suffices to show that,
for an arbitrary boundary pair~$(U,C)$ in~$W$, all three compositions of
two successive transformations in the diagram
\[
\xymatrix{
U\ar[r]^i & U\cup C\ar[d]^r \\
& C,\ar[lu]^\delta|\circ
}
\]
vanish.
Let $u\in U$ and $c\in C$ be the unique elements such that $\{u,c\}$ is
connected. Let $F(u)$ denote $\overline{\{u\}}\setminus\overline{\{c\}}$
and let $F(c)$ denote $\widetilde{\{c\}}\setminus\widetilde{\{u\}}$.
Then $F(u)$ and $F(c)$ are either non-closed and non-open singletons or
maximal totally ordered subsets of~$W$. In either case, $F(u)$
and $F(c)$ are singular subsets of~$W$, and the composition
$U\xrightarrow{i} U\cup C\xrightarrow{r}C$ factors as
\[
U\xrightarrow{}F(u)\xrightarrow{i}F(u)\cup F(c)
\xrightarrow{r}F(c)\xrightarrow{}C.
\]
The transformations $U\to F(u)$ and $F(c)\to C$ are either extensions
or restrictions, depending on the form of $F(u)$ and $F(c)$.
Notice that the transformations $F(u)\xrightarrow{i}F(u)\cup F(c)$ and
$F(u)\cup F(c)\xrightarrow{r}F(c)$ are indecomposable.
We have thus verified that the vanishing of the first composition
follows from the given relations;
the other two compositions can be proven to vanish similarly.
\end{proof}

The above description shows that the ungraded isomorphism class of the
category~$\NTC(W)$ depends only on the number~$n$. More precisely,
let~$O_n$ denote the totally ordered space with~$n$ points. Forming
the long chains~\eqref{eq:longchain} for both~$W$
and~$O_n$, we obtain a bijection between a set of generators of
$\NTC(W)$ and a set of generators of $\NTC(O_n)$. The
foregoing arguments on relations in the two categories
show that this bijection extends to an isomorphism
$\Phi$ between the (ungraded) categories $\NTC(W)$ and $\NTC(O_n)$.

Finally, we convince ourselves that $\Phi$ is compatible with the
notion of \emph{exactness of modules.} Of course it is in general not
true that~$\Phi$ maps boundary pairs to boundary pairs.

\begin{lemma}
	\label{lem:universality_of_successor}
Let $V\xrightarrow{\mu} Y$ be a natural transformation in $\NT_*(V,Y)$
coming from a six-term exact sequence, and let $Y\xrightarrow{\eta} Z$
be the subsequent natural transformation in this six-term exact sequence.
Then every natural transformation $Y\xrightarrow{\eta'} Z'$ with
$\eta'\circ\mu=0$ factors through~$\eta$.
\end{lemma}

\begin{proof}
Consider the exact sequence
$
\NT_*(Z,Z')\xrightarrow{\eta^*}\NT_*(Y,Z')\xrightarrow{\mu^*}\NT_*(V,Z').
$
We have $\mu^*(\eta')=\eta'\circ\mu=0$ and thus $\eta'\in\im(\eta^*)$.
\end{proof}

In other words, the transformation~$\eta$ is the universal
transformation out of~$Y$ with $\eta\circ\mu=0$. It is uniquely
determined up to sign by this property. This is because the only
isomorphisms in the category $\NTC(W)$ are automorphisms of the
form $\pm\id_Z^Z$ for some object $Z$.

\begin{corollary}
	\label{cor:universality_of_successor}
A composite $V\xrightarrow{\mu} Y\xrightarrow{\eta} Z$ in $\NTC$ is
\textup{(}up to sign\textup{)} part of a six-term exact sequence
%associated to some boundary pair in $\NT^*(W)$
\textup{(}in the sense of two successive transformations\textup{)}
if and only if $\mu$ is a six-term exact sequence transformation,
$\eta\circ\mu=0$, and every transformation $\eta'$ out of $Y$ with
$\eta'\circ\mu=0$ factors through $\eta$.
\end{corollary}

The characterisation in Corollary~\ref{cor:universality_of_successor}
shows that the isomorphism $\Phi$ and its inverse respect the property
``being part of a six-term exact sequence'' for pairs of composable
natural transformations.
Hence an ungraded $\NTC(O_n)$-module $M$ is exact if and only if $\Phi^*(M)$
is an exact ungraded $\NTC(W)$-module.

The obtained facts are summarised in Theorem~\ref{thm:ungraded_iso}.
The last assertion in this theorem follows from the fact that $\Phi$
maps identities to identities and morphisms between different objects
to morphisms between different objects.

\section{Counterexamples}
	\label{sec:counterexamples}
In this section we discuss several examples of finite $T_0$-spaces for which $\neg UCT(X)$ holds. First we describe a general approach to obtain counterexamples. The ideas are due to Meyer and Nest~\cite{meyernestCalgtopspacfiltrKtheory}.

If our method for finding resolutions of length 1 described in~\S\ref{sec:positive results} fails, we would like to find counterexamples to Lemmas~\ref{lem:projmodules} and~\ref{lem:projresolutions}, and to the hypothesis that $\Cst$-algebras
over~$X$ in the bootstrap class are classified up to $\KK(X)$-equivalence
by filtrated $\K$-theory.

The general procedure is as follows. If, while trying to show that a given space~$X$
has Property~\ref{ass:kernel}, we encounter a subset $Y\in\LC(X)^*$ for which this
is impossible, then we consider the $\NTC$-module homomorphism
\[
j\colon P_Y\to P^0\perdef\bigoplus \left\lbrace P_Z\mid
\text{there is an indecomposable transformation $Z\to Y$}\right\rbrace
\]
induced by all indecomposable transformations $Z\to Y$ in $\NTC$.

If this homomorphism happens to be injective, then the module
$M\perdef P^0/j(P_Y)$ has the projective resolution
\[
0\to P_Y\xrightarrow{j} P^0\twoheadrightarrow M.
\]
If, moreover, this resolution does not split---for instance, when there
is no non-zero homomorphism from $P^0$ to $P_Y$---then the module $M$
is not projective. However, it is always exact by the two-out-of-three
property, and in all cases we will consider it happens to have
free entries. In this situation the module $M$ yields a counterexample
to Lemma~\ref{lem:projmodules}.

We then go on and define the $\NTC$-module $M_k\perdef M/k\cdot M$ for
some natural number $k\in\N_{\geq 2}$. This module is exact and has the
following projective resolution of length~2:
\[
0\to P_Y\xrightarrow{(-k,j)} P_Y\oplus P^0\xrightarrow{(j,k)} P^0
\twoheadrightarrow M_k.
\]
Under the above assumption that there is no non-zero homomorphism from
$P^0$ to $P_Y$ we can therefore compute
\begin{align*}
\Ext_\NTC^2(M_k,P_Y)
&\cong\Hom_\NTC(P_Y,P_Y)
	/(-k,j)^*\bigb{\Hom_\NTC(P_Y\oplus P^0,P_Y)}\\
&\cong\Hom_\NTC(P_Y,P_Y)/k\cdot\Hom_\NTC(P_Y,P_Y)\\
&\neq 0,
\end{align*}
which shows that $M_k$ has projective dimension~2 and provides a
counterexample to Lemma~\ref{lem:projresolutions}. The above term
never vanishes because
$\Hom_\NTC(P_Y,P_Y)\cong\NT_*(Y,Y)\cong\K^*\bigb{\Ch(Y)}$ is a finitely
generated Abelian group containing at least one free summand.

By Lemma~\ref{lem:projresolutions} there is a $\Cst$-algebra~$A$
for which $\FK^\con(A)$ is isomorphic to~$M$. The K\"unneth
Theorem for the $\K$-theory of tensor products~\cite{Blackadar:Op_Algs}*{V.1.5.10} shows that the filtrated $\K$-theory
of the tensor product $A_k\perdef A\otimes\Cuntz_{k+1}$ with the Cuntz
algebra~$\Cuntz_{k+1}$ is isomorphic to~$M_k$. This is because
$\FK^\con(A)\cong M$ is torsion-free.

\begin{theorem}
	\label{thm:nonprojective}
In the above situation, the $\Cst$-algebra $A_k$ is not
$\ker(\FK^\con)^2$-projective.
\end{theorem}

\begin{proof}
The above assumptions are precisely what is used in the proof of~\cite{meyernestCalgtopspacfiltrKtheory}*{Theorem~5.5}.
\end{proof}

For clarity we list all assumptions made above once again:
\begin{itemize}
\item
the module homomorphism $j\colon P_Y\to P^0$ is injective;
\item
there is no non-zero homomorphism $P^0\to P_Y$;
\item
the module $M=\coker(j\colon P_Y\rightarrowtail P^0)$ has free entries.
\end{itemize}
These assumptions have to be checked in each particular case.
The following theorem from~\cite{meyernestCalgtopspacfiltrKtheory}
then provides two non-isomorphic $\Cst$-algebras
over~$X$ in the bootstrap class with isomorphic filtrated $\K$-theory.

\begin{theorem}[\cite{meyernestCalgtopspacfiltrKtheory}*{Theorem~4.10}]
	\label{thm:I2counterexample}
Let $\catI$ be a homological ideal in a triangulated category~$\Tri$
with enough projective objects. Let $F\colon\Tri\to\AIT$ be a universal
$\catI$-exact stable homological functor. Suppose that $\catI^2\neq 0$.
Then there exist non-isomorphic objects $B,D\inin\Tri$ for which
$F(B)\cong F(D)$ in~$\AIT$.

The objects $B$ and $D$ can be obtained as follows:
Choose a non-$\catI^2$-projective object $A\inin\Tri$ and embed it into
an exact triangle
\[
\xymatrix{
\Sigma N_2\ar[r] & \tilde{A_2}\ar[r] & A\ar[r]^{\iota_2} & N_2
}
\]
with $\iota_2\in\catI^2$ and an $\catI^2$-projective
object~$\tilde{A_2}$. Then $F(\tilde{A_2})\cong F(A)\oplus F(N_2)[1]$
whereas $\tilde{A_2}\not\cong A\oplus N_2[1]$.
\end{theorem}

Now we apply the above procedure to certain explicit examples which will be used in the next chapter.
If $X$ is a space, let $X^{\op}$ denote its dual space, that is, $X^{\op} =X$ as a set and 
the open sets in $X^{\op}$ are exactly the closed sets in $X$.
Let us define the following spaces:
\begin{enumerate}[label=\textup{(\arabic*)}]
\item $X_1=\{1,2,3,4\}$, a basis of the topology~$\tau_{X_1}$ is given by $$\bigl\{ \emptyset, X_1, \{1\},\{2\},\{3\}\bigr\};$$
\item $X_2=X_1^{\op}$;
\item $X_3=\{1,2,3,4\}, \  \tau_{X_3}=\bigl\{ \emptyset, \{1\},\{2\},\{1,2\},\{1,2,3\}, X_3\bigr\}$;
\item $X_4=X_3^{\op}$;
\item $S=\{1,2,3,4\} , \ \tau_{S} = \bigl\{ \emptyset, S, \{1\},\{1,2\},\{1,3\}, \{1,2,3\}\bigr\}$
\item ${C_n}=\{1,2\}\times\Z_n$, $n \geq 2$; a subbasis of the topology~$\tau_{{C_n}}$ is given by
\[\bigl\{(2,k),(1,k)(2,k+[1])\bigr\}_{k \in\Z_n}.\]
\end{enumerate}
Here $\Z_n$ denotes the set $\{0,1,2,\ldots,n-1\}$. In the following we write elements of ${C_n}=\{1,2\}\times\Z_n$ in the
form $a^k$ instead of $(a,k)$.
The directed graphs corresponding to these topological spaces are displayed in Figure~\ref{fig:counterexamples}.
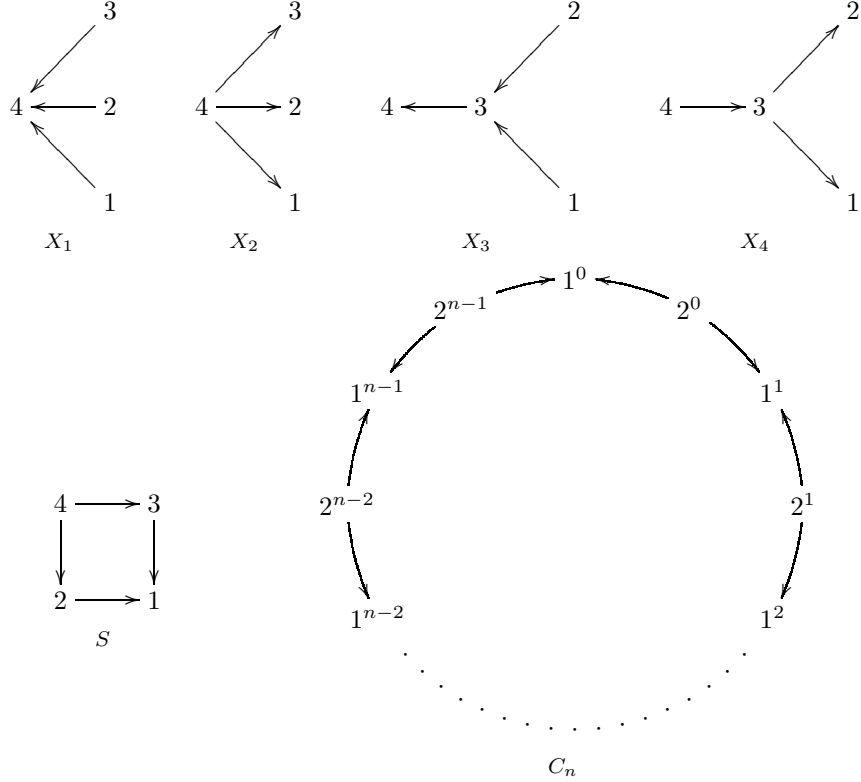
\begin{figure}
\centering
\subfloat[$X_1$]{\label{fig:X1}
\xymatrix{
& 3\ar[dl] \\
4 & 2\ar[l] \\
& 1\ar[ul]
}
}
\qquad
\subfloat[$X_2$]{\label{fig:X2}
\xymatrix{
& 3 \\
4\ar[rd]\ar[ru]\ar[r] & 2 \\
& 1
}
}
\qquad 
\subfloat[$X_3$]{\label{fig:X3}
\xymatrix{
& & 2\ar[ld] \\
4 & 3\ar[l] & \\
& & 1\ar[lu]
}
}
\qquad 
\subfloat[$X_4$]{\label{fig:X4}
\xymatrix{
& & 2 \\
4\ar[r] & 3\ar[ru]\ar[rd] & \\
& & 1
}
}
\\
\subfloat[$S$]{\label{fig:pseudosquare}
\xymatrix{
4\ar[r]\ar[d] & 3\ar[d] \\
2\ar[r] & 1
}
}
\qquad\qquad\quad
\subfloat[${C_n}$]{\label{fig:pseudocircle}
\xy
(0,30)*+{1^0}="10";
(15,25.9808)*+{2^0}="20"; (-15,25.9808)*+{2^{n-1}}="2n-1";
(25.9808,15)*+{1^1}="11"; (-25.9808,15)*+{1^{n-1}}="1n-1";
(30,0)*+{2^1}="21"; (-30,0)*+{2^{n-2}}="2n-2";
(25.9808,-15)*+{1^2}="12"; (-25.9808,-15)*+{1^{n-2}}="1n-2";
{\ar@/_0.2pc/ "20"; "10"}; {\ar@/^0.2pc/ "20"; "11"};
{\ar@/_0.2pc/ "21"; "11"}; {\ar@/^0.2pc/ "21"; "12"};
{\ar@/_0.2pc/ "2n-1"; "1n-1"}; {\ar@/^0.2pc/ "2n-1"; "10"};
{\ar@/_0.2pc/ "2n-2"; "1n-2"}; {\ar@/^0.2pc/ "2n-2"; "1n-1"};
(0,-30)*+{\cdot};
(3.1356,-29.8357)*+{\cdot}; (-3.1356,-29.8357)*+{\cdot};
(6.2374,-29.3444)*+{\cdot}; (-6.2374,-29.3444)*+{\cdot};
(9.2705,-28.5317)*+{\cdot}; (-9.2705,-28.5317)*+{\cdot};
(12.2021,-27.4064)*+{\cdot}; (-12.2021,-27.4064)*+{\cdot};
(15,-25.9808)*+{\cdot}; (-15,-25.9808)*+{\cdot};
(17.6336,-24.2705)*+{\cdot}; (-17.6336,-24.2705)*+{\cdot};
(20.0739,-22.2943)*+{\cdot}; (-20.0739,-22.2943)*+{\cdot};
(22.2943,-20.0739)*+{\cdot}; (-22.2943,-20.0739)*+{\cdot};
\endxy
}
\caption{Directed graphs corresponding to the finite spaces under consideration}
\label{fig:counterexamples}
\end{figure}

\begin{theorem}
\label{counterexamples}
If $X \in \{X_1, X_2 , X_3,X_4,S\} \cup \{{C_n} \mid n \geq 2\}$, then $\neg UCT(X)$ holds.
\end{theorem}
More precisely, our procedure provides the desired counterexamples for all spaces in the above list. For the space $X_2$ this was shown in~\cite{meyernestCalgtopspacfiltrKtheory}, and for $X_4$, $S$ and all ${C_n}$ it was verified in~\cite{Rasmus}. Hence we investigate the spaces $X_1$ and $X_3$ here. We also include the discussion of ${C_n}$ from~\cite{Rasmus}.

\begin{remark}
The investigations cited above and those that follow at this place show that the categories $\NT(X)$ for $X\in\{X_1, X_2, X_3, X_4, S\}$ are all isomorphic in the same sense as described in Theorem~\ref{thm:ungraded_iso}.
\end{remark}

\subsection{Discussion of the space \texorpdfstring{$X_3$}{}}
We begin with the case of the space $X_3$ which we describe in most detail.
The specialisation order on $X_3 =\{1,2,3,4\}$ is generated by the relations
$1\succ 3$, $2\succ 3$, and $3\succ 4$.
%The corresponding directed graph is displayed in Figure~\ref{fig:counterexamples}.
We use abbreviatory notation like $134\perdef\{1,3,4\}$, and similarly.
By~\cite{MR2545613}*{Lemma~2.35},
a $\Cst$-algebra over $X_3$ is a $\Cst$-algebra $A$ with three
distinguished ideals
\[
I_1\perdef A(1),\quad I_2\perdef A(2),\quad I_3\perdef A(123),
\]
subject to the conditions $I_1\cup I_2\subset I_3$ and $I_1\cap I_2=\{0\}$.

The connected, nonempty, locally closed subsets of $X_3 $ are
\[
\LC(X_3 )^*=\{4,34,134,234,3,1234,13,23,123,1,2\}.
\]
% \subsection{Computations with the order complex for \texorpdfstring{$X_3$}{}}
% 	\label{subsec:ordcom3}
%  is a graph with four vertices
% 1, 2, 3, 4,  with edges between any two of them except for 1 and  2,
% and with two 2-simplices joining the triples $(1,3,4)$ and~$(2,3,4)$
% \[
% \Ch(X_3)=
% \xy
% (-10,12)*+[o][F]{2}="p2"; (0,0)*+[o][F]{3}="p3";
% (-10,-12)*+[o][F]{1}="p1"; (20,0)*+[o][F]{4}="p4";
% {\ar@{-} "p4"; "p2"}; {\ar@{-} "p4"; "p3"};
% {\ar@{-} "p4"; "p1"}; {\ar@{-} "p3"; "p1"}; {\ar@{-} "p3"; "p2"};
% \endxy
% \]
% In Table~\ref{tab:closures3} we list the closures and boundaries
% defined in Definition~\ref{def:boundaries} for all $W\in\LC(X_3)^*$.
% 
% \begin{table}[htbp]
% \renewcommand{\arraystretch}{1.5}
% \[
% \begin{array}[t]{c|l|l|l|l|l|l|l|l|l|l|l|}
% 	W					 & 4    & 34   & 134  & 234  & 3   & 1234 & 13 & 23 & 123 & 1 & 2 \\ \hline
% 	\overline W			 & 4    & 34   & 134  & 234  & 34  & 1234   & 134 & 234 & 1234 & 134 & 234 \\ \hline
% 	\overline\partial W  & \emptyset & \emptyset & \emptyset & \emptyset & 4 & \emptyset & 4 & 4 & 4 & 34 & 34 \\ \hline
% 	\widetilde W		 & 1234 & 1234 & 1234 & 1234 & 123 & 1234 & 123 & 123 & 123 & 1 & 2 \\ \hline
% 	\widetilde\partial W & 123  & 12   & 2    & 1    & 12  & \emptyset   & 2 & 1 & \emptyset & \emptyset & \emptyset \\ \hline
% \end{array}
% \]
% \caption{Closures of locally closed subsets of the space $X_3 $}
% \label{tab:closures3}
% \end{table}
The top-dimensional simplices of the geometric realisation of the order complex $\Ch(X_3)$ are the two 2-simplices corresponding to the triples~$(1,3,4)$ and~$(2,3,4)$. Their intersection is the 1-simplex corresponding to the pair~$(3,4)$.

Table~\ref{tab:groups3} contains the isomorphism classes of the groups
$\KSYZ\cong\NT(Y,Z)$ for $Y,Z\in\LC(X_3)^*$. The determination of the
spaces $S(Y,Z)$ is straight-forward from their definition, and the
computation of the $\K$-groups is elementary as well.

\begin{table}[htbp]
\renewcommand{\arraystretch}{1.3}
\[
\begin{array}[t]{c|l|l|l|l|l|l|l|l|l|l|l|l|}
    Y\backslash Z & 4     & 34      & 134   & 234   & 3     & 1234  & 13 & 23 & 123 & 1  & 2  \\ \hline
    4			  & \Z & 0 & 0 & 0 & \Z[1] & 0 & \Z[1]& \Z[1]& \Z[1] & 0  & 0 \\ \hline
    34			  & \Z & \Z & 0 & 0 & 0 & \Z[1] & \Z[1] & \Z[1] & \Z[1]^2 & \Z[1] & \Z[1]  \\ \hline
    134		      & \Z & \Z & \Z & 0 & 0 & 0 & 0 & \Z[1] & \Z[1] &  0 & \Z[1]  \\ \hline
    234			  & \Z & \Z &  0 & \Z & 0 & 0 & \Z[1] & 0 &  \Z[1] &\Z[1] & 0 \\ \hline
    3			  & 0 & \Z & 0 & 0 &\Z & \Z[1] & 0 & 0 &  \Z[1] & \Z[1] &\Z[1] \\ \hline
    1234		  &  \Z & \Z & \Z & \Z &0 &\Z & 0 & 0 & 0 & 0 &0 \\ \hline
    13			  & 0 &\Z &\Z & 0 &\Z & 0 &\Z & 0 & 0 &\Z & 0  \\ \hline
    23			  & 0 &\Z & 0 &\Z &\Z & 0 & 0 &\Z & 0 & 0 &\Z  \\ \hline
    123			  & 0 &\Z &\Z &\Z &\Z &\Z &\Z &\Z &\Z &\Z &\Z  \\ \hline
    1  			  & 0 & 0 &\Z & 0 & 0 &\Z &\Z & 0 &\Z &\Z & 0  \\ \hline
    2   		  & 0 & 0 & 0 &\Z & 0 &\Z & 0 &\Z &\Z & 0 &\Z  \\ \hline
\end{array}
\]
\caption{Groups $\NT(Y,Z)$ of natural transformations for $X_3$}
\label{tab:groups3}
\end{table}

% \subsection{Generators and products of the natural transformations for \texorpdfstring{$X_3$}{}}
% 	\label{subsec:genrel3}
Using the general results from \S\ref{sec:representability} one can simply
determine generators of the groups $\NT_*(Y,Z)\cong\KSYZ$ computed above.

For instance, for all pairs $(Y,Z)$ of subsets $Y,Z\in\LC(X_3)^*$ with
$\NT_*(Y,Z)\cong\Z[0]$, the intersection $Y\cap Z$ is non-empty,
closed in~$Y$ and open in~$Z$. Thus, by Corollary~\ref{cor:evengenerators},
$\NT_*(Y,Z)$ is generated by $\mu_Y^Z\perdef i_{Y\cap Z}^Z\circ r_Y^{Y\cap Z}$.

Similarly, all odd natural transformations arise by composing the transformations
$\mu_Y^Z$ induced by natural \Star homomorphisms with boundary transformations
in $\K$-theory exact six-term sequences. For example, the group $\NT_1(34,123)\cong\Z^2$
is generated by the two transformations $i_1^{123}\circ\delta_{34}^1$ and
$i_2^{123}\circ\delta_{34}^2$. The corresponding generators in
$\K^*\bigb{S(34,123)}=\K^*\bigb{\Ch(X_3)\setminus\{1,2,4\}}$ can be written as
$f^*(\upsilon)$ and $g^*(\upsilon)$, where $\upsilon$ is a generator
of $\K^*\bigb{(0,1)}$, and $f,g\colon\Ch(X_3 )\rightrightarrows \left[0,1\right]$
are continuous maps defined similarly to the map in Lemma~\ref{lem:correspondences}(iii)
with the property that
\[
f^{-1}(0)=\{1\},\quad f^{-1}(1)=\{4\}
\]
and
\[
g^{-1}(0)=\{2\},\quad g^{-1}(1)=\{4\}.
\]
We have now shown that the category $\NTC$ is generated by
transformations coming from natural six-term sequences. With respect to
the canonical relations established in \S\ref{sec:relationsNT} we
obtain the indecomposable transformations indicated in the following
diagram:
\begin{equation}
	\label{eq:indecomp3}
\begin{split}
\xymatrix{
										 & 13\ar[r]^i\ar[rd]^<<<<r    & 134\ar[rd]^r                    &                               & 2\ar[rd]^i  & \\
123\ar[ru]^r\ar[r]^i\ar[rd]^r & 1234\ar[ru]|\hole^<<<r\ar[rd]|\hole^<<<r & 3\ar[r]^i & 34\ar[ru]|\circ^\delta\ar[r]^r\ar[rd]|\circ^\delta & 4\ar[r]|\circ^\delta & 123 \\
										 & 23\ar[ru]^<<<<r\ar[r]^i    & 234\ar[ru]^r                    &                               & 1\ar[ru]^i   & }
\end{split}								 
\end{equation}
The canonical relations for these indecomposable transformations are the following:
\begin{itemize}
\item
all squares within the cube with vertices $123$, $13$, $\ldots$, $34$
commute;
\item
$i_1^{123}\circ\delta_{34}^1+i_2^{123}\circ\delta_{34}^2=\delta_4^{123}\circ r_{34}^4$;
\item the following compositions vanish (all of them are part of
six-term exact sequences):
\begin{multline*}
134\xrightarrow{r} 34\xrightarrow{\delta} 1,\quad
234\xrightarrow{r} 34\xrightarrow{\delta} 2,\quad
3\xrightarrow{i} 34\xrightarrow{r} 4,\\
1\xrightarrow{i} 123\xrightarrow{r} 23,\quad
2\xrightarrow{i} 123\xrightarrow{r} 13,\quad
4\xrightarrow{\delta} 123\xrightarrow{i} 1234.
\end{multline*}
\end{itemize}
Proceeding as in the proof of Corollary~\ref{cor:universality_even}, we find that $\NT(X_3)$ is the universal pre-additive category with these generators and relations, because the morphism groups of the universal pre-additive category are precisely those in Table~\ref{tab:groups3}.

% 
% \subsection{Ring-theoretic properties of the natural transformations for \texorpdfstring{$X_3$}{}}
% 	\label{subsec:ringtheory3}

\begin{lemma}
	\label{lem:3ass1}
The ideal $\NTnil$ is nilpotent and the category $\NTC$ decomposes as
the se\-mi-di\-rect product $\NTnil\rtimes\NTss$.
\end{lemma}

\begin{proof}
By the computations above, we have
$\NTss=\bigoplus_{Y\in\LC(X_3)^*}\NT_*(Y,Y)$ and
\[
\NTnil=\bigoplus_{Y\neq Z\in\LC(X_3)^*}\NT_*(Y,Z).
\]
Hence $\NTC=\NTnil\oplus\NTss$ as Abelian groups. This implies the
se\-mi-di\-rect product decomposition. The fact that $\NTnil$ is
nilpotent follows immediately from the characterisation of the
composition in $\NTC$ provided in the previous section.
\end{proof}

Therefore, Properties~\ref{ass:semi} and~\ref{ass:free} are fulfilled.
However, Property~\ref{ass:kernel} does not hold for the locally closed
subset $34\in\LC(X)^*$: we have
\begin{multline*}
(\NTnil\cdot M)(34)=\range(r_{134}^{34})+\range(r_{234}^{34})+\range(i_{3}^{34})\\
=\ker(i_1^{1234}\circ\delta_{34}^{1})+\range(i_{3}^{34}).
\end{multline*}
For a further simplification we would need an exact sequence containing
the map~$\delta_{3}^{1234}\perdef i_1^{1234}\circ\delta_{3}^{1}$
which does not exist.

\subsection{The counterexamples for \texorpdfstring{$X_3$}{}}
	\label{subsec:results3}
In the previous section our classification method broke down because
there is no exact sequence with connecting map
$\delta_{3}^{1234}=i_1^{1234}\circ\delta_{3}^1$.
In fact, the desired classification is wrong.
In this section we exhibit
\begin{enumerate}
\item 
an exact, entry-free module~$M$ which is not projective,
\item
an exact module that has no projective resolution of length one,
\item
two non-isomorphic objects in the bootstrap class~$\Boot(X_3 )$
with isomorphic filtrated $\K$-theory.
\end{enumerate}

\paragraph*{The non-projective, exact, entry-free module}
For $Y\in\LC(X_3 )^*$ we have defined the free $\NTC$-module on~$Y$
in Definition~\ref{def:freemodule}. The three transformations
$134,3,234\to 34$ in~\eqref{eq:indecomp3} induce a module
homomorphism
\[
j\colon P_{34}\to P^0\perdef P_{134}\oplus P_{3}\oplus P_{234}.
\]

\begin{lemma}
The homomorphism $j$ is injective.
\end{lemma}

\begin{proof}
The longest transformations out of 34 are those to 13, 1234 and~23.
With this we mean that every transformation out of 34 is a sum of
transformations each factoring one of the three transformations
above and that the list of these three transformations is minimal
with this property.
Therefore, it suffices to check that the maps
\[
P_{34}(13)\to P^0(13),\quad P_{34}(1234)\to P^0(1234),
\quad\text{and}\quad P_{34}(23)\to P^0(23)
\]
are injective. This is true because the maps
\begin{multline*}
\NT_*(34,13)\to\NT_*(234,13),\quad \NT_*(34,1234)\to\NT_*(3,1234),\\
\text{and}\quad \NT_*(34,23)\to\NT_*(134,23)
\end{multline*}
are isomorphisms of free cyclic Abelian groups. This, in turn, follows
from the exactness of free modules and the vanishing of the groups
$\NT_*(2,13)$, $\NT_*(4,1234)$, and $\NT_*(1,23)$.
\end{proof}

Since~$j$ is a monomorphism, we can easily compute the cokernel
\[
M\perdef\coker(j\colon P_{34}\rightarrowtail P^0).
\]
We get the following values $M(Y)$ for $Y\in\LC(X_3 )^*$:
\begin{equation}
\begin{split}
\xymatrix{
										 & 0\ar[r]^i\ar[rd]^<<<<r    & \Z\ar[rd]^r                    &                               & \Z[1]\ar@{=}[rd]^i  & \\
\Z[1]\ar[ru]^r\ar[r]^i\ar[rd]^r & 0\ar[ru]|\hole^<<<r\ar[rd]|\hole^<<<r & \Z\ar[r]^i & \Z\ar[ru]|\circ^\delta\ar[r]^r\ar[rd]|\circ^\delta & \Z\ar@{=}[r]|\circ^\delta & \Z[1]. \\
										 & 0\ar[ru]^<<<<r\ar[r]^i    & \Z\ar[ru]^r                    &                               & \Z[1]\ar@{=}[ru]^i   & }\\						 
\end{split}
\end{equation}
As a quotient of two exact modules, the module~$M$ is exact by the
two-out-of-three property. Therefore, the extension maps $i_1^{123}$
and $i_2^{123}$, and the boundary map $\delta_4^{123}$ act by
isomorphisms on~$M$. The other maps can be described in the following
way: write $M(34)$ as $\Z^3/\langle (1,1,1)\rangle$ and
$M(4)$, $M(2)$, $M(1)$ as $\Z^2/\langle (1,1)\rangle$. Then the three
maps $\Z\to\Z^2$ correspond to the three coordinate embeddings
$\Z\hookrightarrow\Z^3$, and the maps $\Z^2\to\Z$ correspond to the
three projections $\Z^3\twoheadrightarrow\Z^2$ onto coordinate
hyperplanes.

\begin{proposition}
	\label{pro:exactnonprojective}
The module $M$ is exact and entry-free, but it is not projective.
\end{proposition}

\begin{proof}
We have already seen that $M$ is exact and entry-free. The projective resolution
$P_{34}\rightarrowtail P^0\twoheadrightarrow M$
does not split because there is no non-zero module homomorphism
$P^0\to P_{34}$ since
$\K^*\bigb{S(34,134)}\cong\K^*\bigb{S(34,3)}\cong\K^*\bigb{S(34,234)}\cong 0$
by Table~\ref{tab:groups3}. This shows that $M$ is not projective.
\end{proof}

\paragraph*{The exact module with projective dimension $2$}
For $k\in\N_{\geq 2}$ we define $M_k\perdef M/k\cdot M$. This module
is exact by the two-out-of-three property and it
has the following projective resolution of length~2:
\begin{equation}
0\to P_{123}\xrightarrow{(-k,j)} P_{123}\oplus P^0
\xrightarrow{(j,k)} P^0\twoheadrightarrow M_k.
\end{equation}
From this resolution we compute
\begin{align*}
\Ext_\NTC^2(M_k,P_{123})
&\cong\Hom_\NTC(P_{123},P_{123})
	/(-k,j)^*\bigb{\Hom_\NTC(P_{123}\oplus P^0,P_{123})}\\
&\cong\Z/k\cdot\Z
\end{align*}
because $\Hom_\NTC(P_{123},P_{123})\cong\Z$ and
$\Hom_\NTC(P^0,P_{123})=0$. This shows that the projective dimension
of~$M_k$ is~2.

\paragraph*{Non-isomorphic objects in \texorpdfstring{$\Boot(X_3)$}{B(X_3)}
					with isomorphic filtrated \texorpdfstring{$\K$}{K}-theory}
As desribed in in the beginning of this section we can find a $\Cst$-algebra $A_k$
with $\FK^\con(A_k)\cong M_k$. Theorem~\ref{thm:nonprojective} shows that
$A_k$ is not $\catI^2$-projective, and Theorem~\ref{thm:I2counterexample} yields the desired counterexample:

\begin{theorem}
There exist $\Cst$-algebras $B$ and $D$ in the bootstrap
class $\Boot(X_3 )$ that are not $\KK(X_3 )$-equivalent but
have isomorphic filtrated $\K$-theory.
\end{theorem}

\subsection{Counterexamples for \texorpdfstring{$X_1$}{}}
The computations for the space $X_1$ are very similar to those for $X_3$.
We will therefore only state the key results used for the construction
of counterexamples on the different levels. The non-empty, connected,
locally closed subsets are
\[
\LC(X_1)^*=\{1234,124,134,234,34,24,14,4,1,2,3,1234\}.
\]
The computation of the groups $\NT(Y,Z)\cong\K^*\bigb{S(Y,Z)}$ is summarised in Table~\ref{tab:groups1}.
\begin{table}[htbp]
\renewcommand{\arraystretch}{1.3}
\[
\begin{array}[t]{c|l|l|l|l|l|l|l|l|l|l|l|l|}
    Y\backslash Z & 1234    & 124   & 134   & 234   & 34 & 24 & 14 & 4  & 1     & 2     & 3     \\ \hline
    1234          & \Z      & \Z    & \Z    & \Z    & \Z & \Z & \Z & \Z & 0     & 0     & 0     \\ \hline
    124           & 0       & \Z    & 0     & 0     & 0  & \Z & \Z & \Z & 0     & 0     & \Z[1] \\ \hline
    134           & 0       & 0     & \Z    & 0     & \Z & 0  & \Z & \Z & 0     & \Z[1] & 0     \\ \hline
    234           & 0       & 0     & 0     & \Z    & \Z & \Z & 0  & \Z & \Z[1] & 0     & 0     \\ \hline
    34            & \Z[1]   & \Z[1] & 0     & 0     & \Z & 0  & 0  & \Z & \Z[1] & \Z[1] & 0     \\ \hline
    24            & \Z[1]   & 0     & \Z[1] & 0     & 0  & \Z & 0  & \Z & \Z[1] & 0     & \Z[1] \\ \hline
    14            & \Z[1]   & 0     & 0     & \Z[1] & 0  & 0  & \Z & \Z & 0     & \Z[1] & \Z[1] \\ \hline
    4             & \Z[1]^2 & \Z[1] & \Z[1] & \Z[1] & 0  & 0  & 0  & \Z & \Z[1] & \Z[1] & \Z[1] \\ \hline
    1             & \Z      & \Z    & \Z    & 0     & 0  & 0  & \Z & 0  & \Z    & 0     & 0     \\ \hline
    2             & \Z      & \Z    & 0     & \Z    & 0  & \Z & 0  & 0  & 0     & \Z    & 0     \\ \hline
    3             & \Z      & 0     & \Z    & \Z    & \Z & 0  & 0  & 0  & 0     & 0     & \Z    \\ \hline
\end{array}
\]
\caption{Groups $\NT(Y,Z)$ of natural transformations for $X_1$}
\label{tab:groups1}
\end{table}

Again, it turns out that the category $\NT$ is generated by the canonical
morphisms and relations discussed in \S\ref{sec:relationsNT}.
The indecomposable morphisms in $\NT$ are displayed in the following diagram.
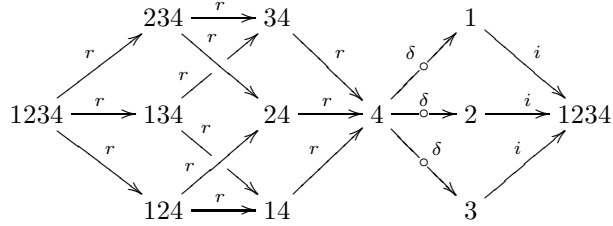
\[
\xymatrix{
                               & 234\ar[r]^r\ar[rd]^<<<<r  & 34\ar[rd]^r &                       & 1\ar[rd]^i & \\
1234\ar[r]^r\ar[ru]^r\ar[rd]^r & 134\ar[ru]|\hole^<<<r\ar[rd]|\hole^<<<r & 24\ar[r]^r  & 4\ar[ru]|\circ^\delta\ar[r]|\circ^\delta\ar[rd]|\circ^\delta & 2\ar[r]^i & 1234 \\
                               & 124\ar[r]^r\ar[ru]^<<<<r  & 14\ar[ru]^r &                       & 3\ar[ru]^i & \\
}
\]
As in the previous example, we construct a non-projective, exact, entry-free module
\[
M\perdef\coker(P_4\rightarrowtail P_{14}\oplus P_{24}\oplus P_{34}),
\]
given by the cokernel of the monomorphisms induced by the natural transformations
$r_{14}^4$, $r_{24}^4$ and $r_{34}^4$. The remaining counterexamples---the exact module
with projective dimension~2 and the non-isomorphic objects in the bootstrap class
$\Boot(X_1)$ with isomorphic filtrated $\K$-theory---can now be obtained as described
in the beginning of \S\ref{sec:counterexamples}.

\subsection{Counterexamples for the space \texorpdfstring{$C_n$}{}}
We apply our method described above for constructing counterexamples
for the space ${C_n}$. We adopt the notation
\[
{C_n}=\{1^0,2^0,1^1,2^1,1^2,\ldots,1^{n-1},2^{n-1},1^n=1^0\}
\]
with the partial order given by the relations
\[
1^0\prec 2^0\succ 1^1\prec 2^1\succ 1^2\prec \ldots \succ 1^{n-1}\prec 2^{n-1}\succ 1^0.
\]
We define
\[
F\perdef{C_n}\setminus\{2^{n-1},1^0,2^0\}
=\{1^1,2^1,1^2\ldots,2^{n-2},1^{n-1}\}.
\]
\begin{definition}
In the following proofs we will say that a topological space is of
\emph{type H} if it is the difference of a contractible compact pair.
We will say that it is of
\emph{type O} if it is the difference of a compact pair $(Z,W)$, where
$Z$ is a contractible space and $W$ is the (topologically) disjoint union
of two contractible subspaces.
\end{definition}

\begin{lemma}
The indecomposable natural transformations in $\NTC$ into $F$ are the two
restriction transformations from $F^0\perdef\{1^0,2^0\}\cup F$ and
$F^n\perdef F\cup\{2^{n-1},1^0\}$ to~$F$.
\end{lemma}

\begin{proof}
For a start, $S(F,F)=\Ch(F)$ and $\NT_*(F,F)\cong\Z$ is generated by
the identity transformation.
We have $S(F^0,F)=S(F^n,F)=\Ch(F)$, so that $\NT_*(F^0,F)\cong\Z$
and $\NT_*(F^n,F)\cong\Z$. Corollary~\ref{cor:evengenerators} implies
that these groups are generated by the natural transformations
$r_{F^0}^F$ and $r_{F^n}^F$, respectively. In the following we will
determine generators of all further groups $\NT_*(Y,F)$ with
$Y\in\LC({C_n})^*$, $Y\neq F$, and verify that each of them
factors through one of the two transformations $r_{F^0}^F$ and
$r_{F^n}^F$.

We begin with supersets of $F$.
Since $S({C_n},F)=\Ch(F)$ is contractible, we have
$\NT_*({C_n},F)\cong\Z$. By Corollary~\ref{cor:evengenerators},
this group is generated by
$r_{{C_n}}^F=r_{F^0}^F\circ r_{{C_n}}^{F^0}$.
Similarly, $S(F\cup\{2^0\},F)=\Ch(F)$, so that
$\NT_*(F\cup\{2^0\},F)\cong\Z$ is generated by the transformation
$r_{F\cup\{2^0\}}^F=r_{F^0}^F\circ i_{F\cup\{2^0\}}^{F^0}$.
The same reasoning applies to the set $F\cup\{2^{n-1}\}$.

Now we consider proper subsets of $F$.
Let $Y=\{1^k, 2^k, \ldots, 1^l\}$ with $1<k\leq l<n-1$.
Then $S(Y,F)$ is of type O and hence $\NT_*(Y,F)\cong\Z[1]$.
We claim that this group is generated by the transformation
$i_D^F\circ\delta_Y^D$, where $D=\{2^l,1^{l+1} \ldots, 1^{n-1}\}$ is
one of the two connected components of $F\setminus Y$. This follows
from Corollary~\ref{cor:oddgenerators} because the spaces
$S(Y,D)\cup S(Y,Y)$ and $S(Y,F)\setminus S(Y,D)$ have trivial
$\K$-theory. We have
$i_D^F\circ\delta_Y^D=r_{F^n}^F\circ i_D^{F^n}\circ\delta_Y^D$.

Let $Y$ be of one of the forms
\[
\{2^k,1^{k+1},\ldots,2^l\},\quad
\{1^1,2^1,\ldots,2^l\},\quad
\{2^k,1^{k+1},\ldots,1^{n-1}\}
\]
for $1\leq k<l<n-1$. Then $S(Y,F)=\Ch(Y)$ and $\NT_*(Y,F)\cong\Z$ is
generated by the transformation $i_Y^F$ which can either be written as
$r_{F^0}^F\circ i_Y^{F^0}$ or as $r_{F^n}^F\circ i_Y^{F^n}$.

For $Y=\{1^k,2^k,\ldots 2^l\}$ with $1<k\leq l<n-1$ we have $\NT_*(Y,Z)=0$
because $S(Y,F)$ is of type~H. The same holds for
$Y=\{2^k,1^{k+1},\ldots, 1^l\}$ with $1\leq k<l<n-1$.

Finally, we investigate the sets $Y\in\LC({C_n})^*$ that are
neither supersets nor subsets of~$F$. For $Y=\{1^k,2^k,\ldots,b\}$
with $k>1$ and $b\in\{2^{n-1},1^0,2^0\}$ the space $S(Y,F)$ is of
type H, so that $\NT_*(Y,F)=0$.

However, if $Y=\{2^k,1^{k+1},\ldots b\}$ with $k\geq 1$ and
$b\in\{2^{n-1},1^0,2^0\}$, we get $S(Y,F)=\Ch(Y\cap F)$ and find that
$\NT_*(Y,F)\cong\Z$ is generated by the natural transformation
$r_{Y\cup F}^F\circ i_Y^{Y\cup F}$. We have already seen that the
transformation $r_{Y\cup F}^F$ factors through $r_{F^0}^F$.
Analogous reasonings can be performed, respectively, for sets of the
form $\{a,\ldots,1^k\}$ or $\{a,\ldots,2^k\}$ with $k\leq n-2$ and
$a\in\{2^{n-1},1^0,2^0\}$.

The last remaining kind of connected, locally closed subsets of
${C_n}$ are those with non-connected intersection with~$F$.
Let
\begin{equation}
	\label{eq:formY}
Y=\{2^k,1^{k+1},\ldots,2^{n-1},1^0,2^0,\ldots,2^l\}
\end{equation}
with $1\leq l<k<n-1$. Then $S(Y,F)$ is the disjoint union of two
contractible sets. Thus $\NT_*(Y,F)\cong\Z^2$. Two generators of this
group are given by the transformations $i_{D_i}^F\circ r_Y^{D_i}$,
where $D_i$ with $i\in\{1,2\}$ denote the two connected components of
$Y\cap F$. Notice that $i_{D_i}^F$ factors through one of the two
transformations $r_{F^0}^F$ and $r_{F^n}^F$.

Given the form~\eqref{eq:formY} for $Y$ with $l<k-1$, adding each of
the points $1^k$ and $1^{l+1}$ turns one of the components of $S(Y,F)$
into a type H space whose $\K$-theory vanishes, and thus removes one of
the above generators. The description of the respective remaining one
does not change.

This completes the list of locally closed, connected subsets
of~${C_n}$.
\end{proof}

\begin{lemma}
	\label{lem:longesttrafosoutofC}
The longest natural transformations in $\NTC$ out of $F$ are the
boundary transformations
$\delta_F^{\{1^0,2^0\}}$, $\delta_F^{\{2^{n-1},1^0\}}$ and
$\delta_F^{C_n}\perdef i_{\{2^0\}}^{C_n}\circ\delta_F^{\{2^0\}}$.
\end{lemma}

\begin{proof}
The space $S(F,\{1^0,2^0\})$ is homeomorphic to the open interval.
This shows that $\NT_*(F,\{1^0,2^0\})\cong\Z[1]$. By
Corollary~\ref{cor:oddgenerators}, this group is generated by the
natural transformation $\delta_F^{\{1^0,2^0\}}$ because the
$\K^1$-group of $S(F,F)\cup S(F,\{1^0,2^0\})$ is trivial.
Symmetrically, $\NT_*(F,\{2^{n-1},1^0\})\cong\Z[1]$ is generated by
the transformation $\delta_F^{\{2^{n-1},1^0\}}$.

The space $S(F,{C_n})$
is of type~O as well, and, by Corollary~\ref{cor:oddgenerators},
we find that $\NT_*(F,{C_n})\cong\Z[1]$ is generated by
$\delta_F^{C_n}$ as defined above because the spaces
$S(F,F)\cup S(F,\{2^0\})$ and $S(F,{C_n})\setminus S(F,\{2^0\})$
have vanishing $\K$-theory. In the following we will determine
generators of all further groups $\NT_*(F,Z)$ with $Z\in\LC({C_n})^*$,
$Z\neq F$, and verify that each of them factors one of the three
transformations $\delta_F^{\{1^0,2^0\}}$, $\delta_F^{\{2^{n-1},1^0\}}$
and $\delta_F^{C_n}$. In fact, we will find that all transformations
out of $F$ factor the transformation $\delta_F^{C_n}$ (except for
$\delta_F^{\{1^0,2^0\}}$ and $\delta_F^{\{2^{n-1},1^0\}}$, of course).
We will not explicitly cite the theorems used for this each time.

We begin with the supersets of $F$ again. Since $S(F,F^0)$ is of type~H,
we get $\NT_*(F,F^0)=0$ and, symmetrically, $\NT_*(F,F^n)=0$. The same
holds for the sets $F\cup\{2^0\}$ and $F\cup\{2^{n-1}\}$. For
$Z=F\cup\{2^0,2^{n-1}\}$, however, the space $S(F,Z)$ is of type~O
so that $\NT_*(F,Z)\cong\Z[1]$. A generator of this group is given by
the composition $i_{\{2^0\}}^Z\circ\delta_F^{\{2^0\}}$. We have
$i_Z^{C_n}\circ\left( i_{\{2^0\}}^Z\circ\delta_F^{\{2^0\}}\right)
=\delta_F^{C_n}$,
which proves that $i_{\{2^0\}}^Z\circ\delta_F^{\{2^0\}}$
factors the transformation $\delta_F^{C_n}$.

Now we examine proper subsets of $F$.
Let $Z=\{1^k, 2^k, \ldots, 1^l\}$ with $1\leq k\leq l\leq n-1$.
Then $S(F,Z)$ is contractible and $\NT_*(Y,F)\cong\Z$ is generated
by the restriction $r_F^Z$. We have
$\delta_Z^{C_n}\circ r_F^Z=\pm\delta_F^{C_n}$, where
$\delta_Z^{C_n}$ denotes the composition
$i_{\{2^{k-1}\}}^{C_n}\circ\delta_Z^{\{2^{k-1}\}}$.
Replacing $Z$ as above by $Z\setminus\{1^k\}$ or $Z\setminus\{1^l\}$
yields a trivial group of natural transformations. For
$Z'=\{2^k,1^{k+1},\ldots, 2^{l-1}\}$ with $1\leq k<l\leq n-1$
we get $\NT_*(F,Z')\cong\Z[1]$ and find the generator
$\delta_D^{Z'}\circ r_F^D$, where $D=\{1^l,2^l,\ldots,1^{n-1}\}$ is
one of the two components of $F\setminus Z'$. We have 
$i_{Z'}^{C_n}\circ(\delta_D^{Z'}\circ r_F^D)=\delta_F^{C_n}$.

For $Z=\{1^k,2^k,\ldots,b\}$ with $k>1$ and $b\in\{2^{n-1},1^0\}$
the space $S(F,Z)$ is of type H, so that $\NT_*(F,Z)=0$.
Yet if $b=2^0$, then $S(F,Z)$ is of type $\text H\sqcup \text O$, that
is, it is the disjoint union of a space of type H and a space of type O,
and $\NT_*(F,Z)\cong\Z[1]$ is generated by
$\delta_{\{1^1\}}^Z\circ r_F^{\{1^1\}}$. Notice that
$i_Z^{C_n}\circ\left( \delta_{\{1^1\}}^Z\circ r_F^{\{1^1\}}\right) 
=\delta_F^{C_n}$. Symmetrical results hold if $Z$ is of the form
$\{a,\ldots,1^k\}$ with $k<n-1$ and $a\in\{2^{n-1},1^0,2^0\}$.

Now let $Z=\{2^k,1^{k+1},\ldots,b\}$ with $k\geq 1$ and
$b\in\{2^{n-1},1^0\}$. Then $S(F,Z)$ is of type O and
$\NT_*(F,Z)\cong\Z[1]$ is generated by
$\delta_{\{1^k\}}^Z\circ r_F^{\{1^k\}}$ and we have
$i_Z^{C_n}\circ\left( \delta_{\{1^k\}}^Z\circ r_F^{\{1^k\}}\right) 
=\pm\delta_F^{C_n}$. For $Z$ as above, but with $b=2^0$, the space
$S(F,Z)$ is of type $\text{O}\sqcup\text{O}$. Hence
$\NT_*(F,Z)\cong\Z[1]^2$. Two generators are given by
$\delta_{\{1^k\}}^Z\circ r_F^{\{1^k\}}$ and
$\delta_{\{1^1\}}^Z\circ r_F^{\{1^1\}}$ if $k>1$, and by
$i_{\{2^0\}}^Z\circ\delta_{\{1^1\}}^{\{2^0\}}\circ r_F^{\{1^1\}}$
and $i_{\{2^1\}}^Z\circ\delta_{\{1^1\}}^{\{2^1\}}\circ r_F^{\{1^1\}}$
for $k=1$. These can be seen to factor
the transformation $\delta_F^{C_n}$ as before.
Again, symmetrical arguments apply to sets $Z$ of the form
$\{a,\ldots,2^k\}$ with $k<n-1$ and $a\in\{2^{n-1},1^0,2^0\}$.

Finally, let
\begin{equation}
	\label{eq:formZ}
Z=\{2^k,1^{k+1},\ldots,2^{n-1},1^0,2^0,\ldots,2^l\}
\end{equation}
with $1\leq l<k<n-1$. Generators of the group $\NT_*(F,Z)\cong\Z[1]$
can be described as in the previous paragraph, including the
factorisation of $\delta_F^{C_n}$. Adding the points $1^k$ and
$1^{l+1}$ to $Z$ as in~\eqref{eq:formZ} with $l<k-1$ removes
respectively one of the afore-stated generators, not violating the
desired characterisation.
\end{proof}

% Let $M$ be an exact $\NT$-module and let $M'\perdef\NTnil\cdot M$.
% We have
% \begin{align*}
% M'(F)&=\range\bigb{r_{F^0}^F\colon M(F^0)\to M(F)}+\range\bigb{r_{F^n}^F\colon M(F^n)\to M(F)}\\
% &=\ker\bigb{\delta_F^{\{1^0,2^0\}}\colon M(F)\to M(\{1^0,2^0\})}+\range\bigb{r_{F^n}^F\colon M(F^n)\to M(F)}.
% \end{align*}
% 
% In order to identify this with the kernel of a natural map out of $M(F)$
% we would need a long exact sequence containing the natural transformation
% \[
% \delta_F^{\{1^0,2^0\}}\circ r_{F^n}^F.
% \]
% Such a long exact sequence does not exist because the sets $F^n$ and
% $\{1^0,2^0\}$ are not disjoint and thus do not form a boundary pair.

The two restriction transformations from~$F^0$ and~$F^n$ to~$F$ induce a module homomorphism
\[
j\colon P_F\to P^0\perdef P_{F^0}\oplus P_{F^n}.
\]

\begin{lemma}
The homomorphism $j$ is injective.
\end{lemma}

\begin{proof}
By Lemma~\ref{lem:longesttrafosoutofC} if suffices to show that the maps
\begin{align*}
P_F(\{1^0,2^0\})&\to P^0(\{1^0,2^0\}),\\
P_F(\{2^{n-1},1^0\})&\to P^0(\{2^{n-1},1^0\}),\\
P_F({C_n})&\to P^0({C_n})
\end{align*}
are injective. This follows from the injectivity of the maps
\begin{align*}
\NT(F,\{1^0,2^0\})&\to\NT(F^n,\{1^0,2^0\}),\\
\NT(F,\{2^{n-1},1^0\})&\to\NT(F^0,\{2^{n-1},1^0\}),\\
\NT(F,{C_n})&\to\NT(F^0,{C_n}),
\end{align*}
which we obtain from the vanishing of the groups
\[
\NT(\{2^{n-1},1^0\},\{1^0,2^0\}),\; \NT(\{1^0,2^0\},\{2^{n-1},1^0\})
\]
and $\NT(\{1^0,2^0\},{C_n})$.
\end{proof}

%Now we define $M\perdef\coker(j\colon P_F\rightarrowtail P^0)$.

\begin{proposition}
The module $M\perdef\coker(j\colon P_F\rightarrowtail P^0)$ is exact and entry-free, but it is not projective.
\end{proposition}

\begin{proof}
The module~$M$ is exact by the two-out-of-three property.
%torsion-free
The fact that~$M$ is entry-free follows from a direct investigation
of the map~$j$ via generators of the Abelian groups involved. This
is particularly easy when $P_F(Z)$ is of rank~1. As an example for one
of the more complicated cases,
we discuss the case $Z=\{2^{n-1},1^0,2^0\}={C_n}\setminus F$.
In the proof of Lemma~\ref{lem:longesttrafosoutofC}, we found that
$P_F(Z)=\NT_*(F,Z)\cong\Z[1]^2$ is generated by
$\delta_{\{1^1\}}^Z\circ r_F^{\{1^1\}}$ and
$\delta_{\{1^k\}}^Z\circ r_F^{\{1^k\}}$. Similarly,
$P_{F^n}(Z)=\NT_*(F^n,Z)$ is generated by
$\delta_{\{1^1\}}^Z\circ r_{F^n}^{\{1^1\}}
=\delta_{\{1^1\}}^Z\circ r_F^{\{1^1\}}\circ r_{F^n}^F$ and
$\delta_{\{1^k\}}^Z\circ r_{F^n}^{\{1^k\}}
=\delta_{\{1^k\}}^Z\circ r_F^{\{1^k\}}\circ r_{F^n}^F$,
and $P_{F^0}(Z)=\NT_*(F^0,Z)$ is generated by
$\delta_{\{1^1\}}^Z\circ r_{F^0}^{\{1^1\}}
=\delta_{\{1^1\}}^Z\circ r_F^{\{1^1\}}\circ r_{F^0}^F$ and
$\delta_{\{1^k\}}^Z\circ r_{F^0}^{\{1^k\}}
=\delta_{\{1^k\}}^Z\circ r_F^{\{1^k\}}\circ r_{F^0}^F$.
Hence the map $j(Z)\colon P_F(Z)\to P^0(Z)$ can be identified with
the map
\[
\Z^2\to\Z^4,\quad (a,b)\mapsto (a,b,a,b)
\]
whose cokernel is entry-free. The computations for all other
subsets~$Z$ in $\LC({C_n})^*$ are similar.
The projective resolution
$
0\to P_F\to P^0\twoheadrightarrow M
$
does not split because there is no non-zero homomorphism from $P^0$ to
$P_F$. This follows from $\NT(F,F^0)=0$ and $\NT(F,F^n)=0$.
\end{proof}

This provides the counterexample on the level of projective modules.
The counterexamples on the two deeper levels now follow as described in
the beginning of this section. The three assumptions listed at the
beginning of \S\ref{sec:counterexamples}
% \begin{itemize}
% \item
% the module homomorphism $j\colon P_F\to P^0$ is injective,
% \item
% there is no non-zero homomorphism $P^0\to P_F$,
% \item
% the module $M=P^0/j(P_Y)$ is entry-free,
% \end{itemize}
have all been verified and we obtain the desired result.

\begin{theorem}
There exist $\Cst$-algebras $B$ and $D$ in the bootstrap
class $\Boot({C_n})$ that are not $\KK({C_n})$-equivalent but
have isomorphic filtrated $\K$-theory.
\end{theorem}

\section{The complete description}
\label{sec: the complete description}

We already know that, if $X$ is of type~(A), then $UCT(X)$ holds. The aim of this section is to prove the converse implication. We want to show that, if $X$ is not of type~(A), then we can ``embed'' one of the counterexamples from \S\ref{sec:counterexamples} into~$X$. Knowing that $\neg UCT$ holds for the counterexample, we will use the results from \S\ref{The UCT criterion} to conclude that $\neg UCT(X)$ holds.

\begin{definition}
A topological subspace $X'$ of a finite $T_0$-space $X$ is \emph{tight} if
\[ y \rightarrow x \text{ in } X' \Longleftrightarrow  y \rightarrow x \text{ in } X,\]
that is, there is a directed edge from~$y$ to~$x$ in $\Gamma(X')$ if and only if there is a directed edge from~$y$ to~$x$ in $\Gamma(X)$ (see Definition~\ref{def:Hasse}).
\end{definition}

So, if $X'$ is a topological subspace of $X$, then $X'$ is tight in $X$ if and only if $\Gamma(X')$ is a subgraph of $\Gamma(X)$. If  $Y$ is another finite $T_0$-space such that there exists an embedding $\Gamma(Y) \hookrightarrow \Gamma(X)$ as directed graphs, then $Y$ may be viewed as a tight subspace of $X$.

\begin{lemma}
Let $X$ be a finite $T_0$-space such that $\Gamma(X)$ contains either $\Gamma(X_1)$ or $\Gamma(X_2)$ as a subgraph. Then $\neg UCT(X)$ holds.
\end{lemma}

\begin{proof}
$\Gamma(X_1) \subseteq \Gamma(X)$ allows us to view $X_1$ as a tight subspace of $X$.
Let $y \in LC(X_1)$ then there are $x_1, x_2 \in X_2$ such that $x_1 \succeq y \succeq x_2$. Without loss of generality we may assume that $x_1 =1$ and $x_2 =4$. Since $1 \rightarrow 4$ we have $y =1$ or $y =4$ by Lemma~\ref{restriction}. Therefore $X_1$ is locally closed in $X$, similarly we see that $X_2$ is locally closed in $X$ if $\Gamma(X_2) \subseteq \Gamma(X)$. Therefore $\neg UCT(X)$ holds by Theorem~\ref{counterexamples} and Proposition~\ref{embedding}(ii).
\end{proof}

\begin{proposition}
\label{auxiliary 1}
Let $X$ be a finite $T_0$-space such that $\Gamma(X)$ contains $\Gamma(X_3)$ as a subgraph. Define
\[ \pi_3\colon LC(X_3) \rightarrow X_3 , \ \ \ \pi_3 (x) = \begin{cases} x & \text{if } x \in X_3, \\   3 & \text{else.}
\end{cases}\]
Then $ \pi_3$ is continuous.
\end{proposition}

\begin{proof}
Let us first show the following claim:
\[\text{If $x \in LC(X_3)\setminus X_3$, then $x \succ 4$, $x \npreceq 3$, $x \nsucceq 3$, $x \nsucceq 1$, $x \nsucceq 2 $.}\]
Let $x \in LC(X_3)\setminus X_3$. Then there are $x_1, \ x_2 \in X_3$ such that $x_1 \prec x \prec x_2$. Since $1\rightarrow 3, \ 2\rightarrow 3, \ 3\rightarrow 4$ Lemma~\ref{restriction} shows that $x_1 = 4$ and $x_2 \in \{1,2\}$. Without loss of generality we may assume that $x_2 =1$. This implies of course that $x \nsucceq 1$ and $x \succ 4$. Assume $x \succeq 2$, then $1 \succ x \succ 2 \succ 3$ this is a contradiction to $1 \rightarrow 3$. By the same argument $x \succeq 3$ leads to a contradiction. Assume $x \preceq 3$, then $4 \prec x \prec 3$. This is a contradiction to $3 \rightarrow 4$. This shows the claim.

To check that $ \pi_3 $ is continuous, we have to check that it is monotone. Let $x,y \in LC(X_3) $, if $x,y \in X_3 $ then $x \preceq y$ clearly implies $ \pi_3(x) \preceq  \pi_3(y)$. If  $x,y \in LC(X_3)\setminus X_3 $ then $ \pi_3(x) =3 = \pi_3(y)$. If $x \in   LC(X_3)\setminus X_3 $, $y \in X_3$ and $y \prec x$, then $y =4$ by Claim \#1. Therefore $ \pi_3(4) =4 \prec 3 =  \pi_3(x)$. If $y \in  X_3 $, $x \in LC(X_3)\setminus X_3 $ and $y \succ x$, then either $y =1$ or $y=2$ by Claim \#1, and in both cases $  \pi_3(y) = y \succ 3 = \pi_3(x)$. This shows that $ \pi_3 $ is continuous. 
\end{proof}

\begin{proposition}
\label{auxiliary 2}
Let $X$ be a finite $T_0$-space such that $\Gamma(X)$ contains $\Gamma(X_4)$ as a subgraph. Define
\[ \pi_4\colon LC(X_4) \rightarrow X_4 , \ \ \  \pi_4 (x) = \begin{cases} x & \text{if } x \in X_4, \\   3 & \text{else.}
\end{cases}\]
Then $ \pi_4$ is continuous.
\end{proposition}

\begin{proof}
This is proven completely analogously to Proposition~\ref{auxiliary 1}---just switch $\prec$ and~$\succ$ in the proof.
\end{proof}

\begin{corollary}
Let $X$ be a finite $T_0$-space such that $\Gamma(X)$ contains either $\Gamma(X_3)$ or $\Gamma(X_4)$ as a subgraph. Then $\neg UCT(X)$ holds.
\end{corollary}

\begin{proof}
Assume $\Gamma(X_3) \subseteq \Gamma(X)$ and let $Y = LC(X_3)$. There is an inclusion
$ \iota_3\colon X_3 \hookrightarrow LC(X_3)$ and $ \pi_3\colon LC(X_3) \rightarrow X_3$ from Proposition~\ref{auxiliary 1}. We clearly have $ \pi_3 \circ  \iota_3 = \ID_{X_3}$. This shows that $\neg UCT(LC(X_3))$ holds by Proposition~\ref{embedding}(ii) and therefore $\neg UCT(X)$ holds by 
Proposition~\ref{embedding}(i). The same arguments using $\iota_4\colon X_4 \hookrightarrow LC(X_4)$ and $\pi_4$ from Proposition~\ref{auxiliary 2} show the corresponding statement for $X_4$.
\end{proof}

\begin{corollary}
\label{coro: degree >2}
Let $X$ be a finite $T_0$-space such that $\Gamma(X)$ has a vertex of unoriented degree at least~$3$. Then $\neg UCT(X)$ holds.
\end{corollary}

\begin{proof}
$\Gamma(X)$ must contain either $\Gamma(X_1)$, $\Gamma(X_2)$, $\Gamma(X_3)$ or $\Gamma(X_4) $ as a subgraph.
\end{proof}

\begin{proposition}
\label{circle}
Let $X$ be a finite connected $T_0$-space such that every vertex of $\Gamma(X)$ has \textup{(}unoriented\textup{)} unoriented degree~$2$. Then $\neg UCT(X)$ holds.
\end{proposition}

\begin{proof}
The assumption means that $\Gamma(X)$ as an undirected graph consists of a cycle.
By the definition of the oriented degree $d_o$ from \S\ref{elementary notions of graph theory}, we have $d_o(x) \in \{-2,0,2\}$ for every $x \in X$ and \[\sum_{x \in X}d_o(x) = 0.\]
This means that there are as many vertices with oriented degree $2$ as vertices with oriented degree $-2$. Let $n$ be the number of vertices with oriented degree $2$.
Since $\Gamma(X)$ cannot be a directed circle, $n$ is at least~$1$.

%\begin{flushenumerate}
% \textbf{Case (a)} $n=1$:
We first consider the case $n=1$. Then there is exactly one vertex $a$ with oriented degree $2$, one vertex $b$ with oriented degree $-2$ and two directed paths $\rho = (v_i)_{i =0,\ldots,n}$ and $\sigma = (w_i)_{i =0,\ldots,m}$ from $a$ to $b$ such that 
\[\rho \cap \sigma= \{a,b\}, \ \ \  \rho \cup \sigma =X. \] 
Define maps  $f\colon X \rightarrow S$ and $g\colon S \rightarrow X$ via
\[ f(x) =  \begin{cases} 1 &\text{ if } x= a, \\  2 &\text{ if } x= v_i  \text { for } i =1, \ldots ,n-1, \\ 3  &\text{ if } x= w_i  \text { for } i =1, \ldots,m-1,\\ 4  &\text{ if } x= b,\end{cases} \text{ and } g(s) =  \begin{cases} a &\text{ if } s= 1, \\  v_1 &\text{ if } s= 2,   \\ w_1  &\text{ if } s= 3, \\ b  &\text{ if } s= 4.\end{cases}\]
These maps are continuous since they are monotone. It is clear that $f \circ g = \ID_{S}$. Therefore $\neg UCT(X)$ holds by Theorem~\ref{counterexamples} and Proposition~\ref{embedding}(ii).

% \textbf{Case (b)} $n >1$:
Now we investigate the case $n>1$.
We will basically proceed as in the previous case; only notation becomes a bit more complicated. Let $\cyc n$ denote the cyclic group of order~$n$. Ordering the vertices of oriented degree $2$ and $-2$ clockwise, we obtain sequences $(a_k)_{k \in \cyc n}$ and $(b_k)_{k \in \cyc n}$ in $X$ such that $d_o(a_k) =2$ and $d_o(b_k) =-2$ for all $k \in \cyc n$.
Analogously to the previous case, there is a sequence of directed paths $\left(\rho^k = (v_i^k)_{i=1,\ldots,n_k}\right)_{k \in\cyc n}$ from $a_k$ to $b_k$ and a sequence of directed paths $\left(\sigma^k = (w_i^k)_{i=1,\ldots,m_k}\right)_{k \in\cyc n}$ from $a_k$ to $b_{k-[1]}$ such that 
\[\rho^k \cap \rho^l =\sigma^k \cap \sigma^l = \varnothing \text{ if } k \neq l, \ \ \ \rho^k \cap \sigma^l = \begin{cases}  a_k & \text{ if } k=l,\\ b_k & \text{ if } k =l -[1],\\ \varnothing  & \text{else,}\end{cases}\]
and
\[ \bigcup_{k\in\cyc k}\rho^k \cup  \bigcup_{k\in\cyc k}\sigma^k = X.\]
Define maps  $f\colon X \rightarrow C_n$ and $g\colon C_n \rightarrow X$ via
\[ f(x) =  \begin{cases} (k,a) &\text{ if } x= a_k, \\  (k,b) &\text{ if } x= v^k_i  \text { for } i =1, \ldots ,n_k, \\ (k-[1],b)  &\text{ if } x= w^k_i  \text { for } i =1, \ldots,m_k-1,\end{cases}
\]
and
\[ g((k,y)) =  \begin{cases} a_k &\text{ if } y= a, \\  b_k &\text{ if } y= b.  \end{cases}\] 
The maps $f$ and $g$ are monotone and thus continuous. Clearly, $f \circ g = \ID_{C_n}$. Therefore $\neg UCT(X)$ holds by Theorem~\ref{counterexamples} and Proposition~\ref{embedding}(ii).
%\qedhere
%\end{flushenumerate}
\end{proof}

The following lemma provides an alternative characterization of type~(A) spaces.

\begin{lemma}
\label{lem: Char of type A}
Let $X$ be a finite connected $T_0$-space with more than one point. The following statements are equivalent:
\begin{enumerate}[label=\textup{(\arabic*)}]
\item $X$ is of type \textup{(A)};
\item there are exactly two vertices in $X$ with unoriented degree~$1$, all other vertices have unoriented degree~$2$.
\end{enumerate}
\end{lemma}

\begin{proof}
The direction (1)$\Longrightarrow$(2) is obvious (see Figure~\ref{fig:genform} on page~\pageref{fig:genform}). For the converse direction, notice that (2) implies that the graph $\Gamma(X)$ corresponding to the specialisation preorder on $X$ is isomorphic as an undirected graph to the graph corresponding to the specialisation preorder on the totally ordered space with the same number of points. This shows that $\Gamma(X)$ is isomorphic as a directed graph to the graph corresponding to the specialisation preorder on some type~(A) space as displayed in Figure~\ref{fig:genform}. Since we are dealing with $T_0$-spaces this implies that $X$ is homeomorphic to that type~(A) space.
\end{proof}

\begin{theorem}
	\label{thm:final}
Let $X$ be a finite $T_0$-space.
Then $UCT(X)$ holds if and only if~$X$ is a disjoint union of spaces of type \textup{(A)}.
Otherwise, $\neg UCT(X)$ holds.
\end{theorem}

\begin{proof}
That $UCT(X)$ holds if $X$ is  a disjoint union of spaces of type (A) follows from Theorem~\ref{theorem: positive} and Lemma~\ref{lem: UCT + connectedness}. Now let~$X$ be a space such that $\neg UCT(X)$ does \emph{not} hold. By Lemma~\ref{lem: UCT + connectedness}, it suffices to show that $X$ is of type (A) under the assumption that~$X$ is connected (and hence, by Lemma~\ref{lem: graphs + connectedness}, that $\Gamma(X)$ is connected as an undirected graph). By Corollary~\ref{coro: degree >2}, all vertices~$x$ of $\Gamma(X)$ have unoriented degree less than~$3$. By the last remark and Proposition~\ref{circle} there is at least one vertex of unoriented degree less than~$2$. Since $\Gamma(X)$ is connected as an undirected graph and finite, there are exactly two vertices of unoriented degree~$1$ and all other vertices have unoriented degree~$2$. Thus~$X$ is of type~(A) by Lemma~\ref{lem: Char of type A} as claimed.
\end{proof}

\subsection*{Acknowledgement}
This paper emerged from the first-named author's Diplom thesis~\cite{Rasmus} and the second-named author's dissertation~\cite{Manuel} which were supervised by Ralf Meyer at the University of G\"ottingen. The first-named author thanks Takeshi Katsura for insightful discussions.

\end{document}